\documentclass[reqno,11pt]{amsart}
\usepackage{amsmath,amssymb,mathrsfs,amsthm,amsfonts}
\usepackage[inline]{enumitem} 
\usepackage[usenames,dvipsnames]{xcolor}
\usepackage{hyperref}
\usepackage{tikz}
\usepackage{bbm}
\tikzstyle{block} = [rectangle, draw, 
   ]
\usetikzlibrary{shapes,positioning}

\usepackage{comment}
\hypersetup{%
	colorlinks=true, linkcolor=ForestGreen,
	citecolor=ForestGreen
}
\usepackage[paper=letterpaper,margin=1in]{geometry}
\usepackage{acronym}
\usepackage{dsfont}

\newcommand{\Dim}{\ensuremath{\mathrm{Dim}_{\mathds{H}}}}

\newcommand\tree{\begin{tikzpicture}[level distance=0.8mm]
\tikzstyle{every node}=[fill,circle,inner sep=0.5pt]
\tikzstyle{level 1}=[sibling distance=1.3mm]
  \node{}[grow'=up]
    child {[fill] circle (0.5pt)}
    child {[fill] circle (0.5pt)};
\end{tikzpicture}}
\newcommand{\ttree}{\begin{tikzpicture}[level distance=0.8mm]
\tikzstyle{every node}=[fill,circle,inner sep=0.5pt]
\tikzstyle{level 1}=[sibling distance=1.3mm]
\tikzstyle{level 2}=[sibling distance=1.3mm]
\tikzstyle{level 3}=[sibling distance=1mm]
  \node{}[grow'=up]
    child {[fill] circle (0.5pt)}
     child {node{} child {[fill] circle (0.5pt)} child {[fill] circle (0.5pt)}};
\end{tikzpicture}}
\newcommand{\tttree}{\begin{tikzpicture}[level distance=0.8mm]
\tikzstyle{every node}=[fill,circle,inner sep=0.5pt]
\tikzstyle{level 1}=[sibling distance=1.3mm]
\tikzstyle{level 2}=[sibling distance=1.3mm]
\tikzstyle{level 3}=[sibling distance=1.3mm]
  \node{}[grow'=up]
    child {[fill] circle (0.5pt)}
     child {node{} child {[fill] circle (0.5pt)} child {node{} child {[fill] circle (0.5pt)} child {[fill] circle (0.5pt)} }};
\end{tikzpicture}}
\newcommand{\Tree}{\begin{tikzpicture}[level distance=0.8mm]
\tikzstyle{every node}=[fill,circle,inner sep=0.5pt]
\tikzstyle{level 1}=[sibling distance=1.3mm]
\tikzstyle{level 2}=[sibling distance=0.75mm]
\tikzstyle{level 3}=[sibling distance=1mm]
  \node{}[grow'=up]
     child {node{} child {[fill] circle (0.5pt)} child {[fill] circle (0.5pt)}}
     child {node{} child {[fill] circle (0.5pt)} child {[fill] circle (0.5pt)}};
\end{tikzpicture}}
\newcommand{\vvert}[1]{{\left\vert\kern-0.25ex\left\vert\kern-0.25ex\left\vert #1 
    \right\vert\kern-0.25ex\right\vert\kern-0.25ex\right\vert}}

\newcommand{\1}{\mathds{1}}
\renewcommand{\P}{\mathds{P}}

\newcommand{\E}{\mathds{E}}
\newcommand{\R}{\mathds{R}}

\newcommand{\Z}{\mathds{Z}}

\newcommand{\B}{\mathcal{B}}

\newcommand{\e}{\mathrm{e}}

\newcommand\dP{\mathds{P}}
\newcommand\dR{\mathds{R}}
\newcommand\dE{\mathds{E}}
\newcommand\dZ{\mathds{Z}}
\newcommand\dN{\mathds{N}}

\newcommand\dS{\mathds{S}}
\newcommand\dQ{\mathds{Q}}
\newcommand\dV{\mathds{V}}

\newcommand\dT{\mathds{T}}

\newcommand\bR{}

\newcommand\sC{\mathscr{C}}
\newcommand\sD{\mathscr{D}}

\newcommand\sH{\mathscr{H}}
\newcommand\sS{\mathscr{S}}

\newcommand\sX{\mathscr{X}}

\newcommand\cB{\mathcal{B}}

\newcommand\cD{\mathcal{D}}
\newcommand\cF{\mathcal{F}}
\newcommand\cG{\mathcal{G}}
\newcommand\cH{\mathcal{H}}
\newcommand\cI{\mathcal{I}}
\newcommand\cJ{\mathcal{J}}
\newcommand\cK{\mathcal{K}}
\newcommand\cL{\mathcal{L}}
\newcommand\cP{\mathcal{P}}

\newcommand\cM{\mathcal{M}}
\newcommand\cT{\mathcal{T}}
\newcommand\cO{\mathcal{O}}
\newcommand\cZ{\mathcal{Z}}
\newcommand\cQ{\mathcal{Q}}

\newcommand\cU{\mathcal{U}}

\newcommand\blambda{\boldsymbol{\lambda} }

\newcommand\bmu{\boldsymbol{\mu}}

\newcommand\fa{\mathfrak{a}}

\newcommand\fc{\mathfrak{c}}

\newcommand\fn{\mathfrak{n}}
\newcommand\fX{\mathfrak{X}}

\newcommand\fZ{\mathfrak{Z}}

\newtheorem{stat}{Statement}[section]
\newtheorem{proposition}[stat]{Proposition}
\newtheorem{corollary}[stat]{Corollary}

\newtheorem{theorem}[stat]{Theorem}
\newtheorem{lemma}[stat]{Lemma}
\theoremstyle{definition}
\newtheorem{definition}[stat]{Definition}
\newtheorem{remark}[stat]{Remark}

\numberwithin{equation}{section}

\makeatletter
\renewcommand{\tocsection}[3]{%
  \indentlabel{\@ifnotempty{#2}{\bfseries\ignorespaces#1 #2\quad}}\bfseries#3}
\renewcommand{\tocsubsection}[3]{%
  \indentlabel{\@ifnotempty{#2}{\ignorespaces#1 #2\quad}}#3}

\newcommand\@dotsep{4.5}
\def\@tocline#1#2#3#4#5#6#7{\relax
  \ifnum #1>\c@tocdepth 
  \else
    \par \addpenalty\@secpenalty\addvspace{#2}%
    \begingroup \hyphenpenalty\@M
    \@ifempty{#4}{%
      \@tempdima\csname r@tocindent\number#1\endcsname\relax
    }{%
      \@tempdima#4\relax
    }%
    \parindent\z@ \leftskip#3\relax \advance\leftskip\@tempdima\relax
    \rightskip\@pnumwidth plus1em \parfillskip-\@pnumwidth
    #5\leavevmode\hskip-\@tempdima{#6}\nobreak
    \leaders\hbox{$\m@th\mkern \@dotsep mu\hbox{.}\mkern \@dotsep mu$}\hfill
    \nobreak
    \hbox to\@pnumwidth{\@tocpagenum{\ifnum#1=1\bfseries\fi#7}}\par
    \nobreak
    \endgroup
  \fi}
\AtBeginDocument{%
\expandafter\renewcommand\csname r@tocindent0\endcsname{0pt}
}
\def\l@subsection{\@tocline{2}{0pt}{2.5pc}{5pc}{}}
\makeatother

\begin{document}

\title[Fractal geometry of the PAM in 2D and 3D with white noise potential]{Fractal geometry of the parabolic Anderson model in 2D and 3D with white noise potential
}
%
\author[P.\ Ghosal]{Promit Ghosal}
\address{P.\ Ghosal,
	Department of Mathematics, Massachusetts Institute of Technology (MIT),
	\newline\hphantom{\quad \ \ P. Ghosal}
	77 Massachusetts Avenue, Cambridge, MA 02139, USA
}
\email{promit@mit.edu}

\author[J.\ Yi]{Jaeyun Yi }
\address{J.\ Yi,
	School of Mathematics, Korea Institute for Advanced Study (KIAS),
	\newline\hphantom{\quad \ \ J. Yi}
	85 Hoegiro Dongdaemun-gu, Seoul 02455, Republic of Korea
	}
\email{jaeyun@kias.re.kr}

\date{\today}

\begin{abstract} We study the parabolic Anderson model (PAM)
\begin{equation*}
     \begin{cases} {\partial \over \partial t}u(t,x) =\frac{1}{2}\Delta u(t,x) + u(t,x)\xi(x), \quad t>0, x\in \dR^d,\\
   u(0,x) \equiv 1, \quad x\in \dR^d,
   \end{cases}
   \end{equation*} where $\xi$ is spatial white noise on $\dR^d$ with $d \in\{2,3\}$. We show that the peaks of the PAM are macroscopically multifractal. More precisely, we prove that the spatial peaks of the PAM have infinitely many distinct values and we compute the macroscopic Hausdorff dimension (introduced by Barlow and Taylor \cite{BT89,BT92}) of those peaks. As a byproduct, we obtain the exact spatial asymptotics of the solution of the PAM. We also study the spatio-temporal peaks of the PAM and show their macroscopic multifractality. %
   Some of the major tools used in our proof techniques include paracontrolled calculus and tail probabilities of the largest point in the spectrum of the \emph{Anderson Hamiltonian}. 


\vspace{1cm} 
 
\noindent{\it Keywords:} Parabolic Anderson model, Anderson Hamiltonian, macroscopic Hausdorff dimension, Paracontrolled Calculus. \\
	
	\noindent{\it \noindent AMS 2020 subject classification:}
	Primary. 60H15; Secondary. 35R60, 60K37.
\end{abstract}

\maketitle

\tableofcontents

\section{Introduction}\label{sec_intro} 
We consider the parabolic Anderson model on $\R^d$ with $d\in \{2,3\}$
\begin{equation}\label{eq:PAM}
     \begin{cases} {\partial \over \partial t}u(t,x) =\frac{1}{2}\Delta u(t,x) + u(t,x)\xi(x), \quad t>0, x\in \R^d,\\
   u(0,x) \equiv 1, \quad x\in \R^d.
   \end{cases}
   \end{equation} where the random potetial $\xi$ is the spatial white noise on $\dR^d$ which is a mean zero Gaussian field with delta correlation between any two spatial points. PAM is one of the prototypical framework for modelling conduction of electron in crystals filled with defects. There is a competition between the two terms appearing in the operator: While the eigenfunctions of
the Laplacian which depicts the behavior of the electron waves being spread out over the whole space, the multiplication-by-$\xi$ operator which models the random defects, tends to concentrate
the mass of the eigenfunctions in very small regions. A discrete version of the above Hamiltonian was introduced in a seminal paper of Anderson \cite{And58} where he showed that the bottom part of the spectrum consists of localized eigenfunctions. This phenomenon is often termed as the \emph{Anderson localization} which triggered an enormous amount of research activities in last several decades (see \cite{DL2020} for detailed references).    

The solution theory of \eqref{eq:PAM} is obtained by using a  mollified version of noise $\xi_\epsilon$ minus a correction $c_{\epsilon} = \frac{1}{2\pi} \log \epsilon$ and it is proved that the solution $u_{\epsilon}$ of the PAM with potential $\xi_{\epsilon} - c_{\epsilon}$ has a limit as $\epsilon \to 0$. It was first constructed on torus $\dT^2$ by Hairer \cite{Hai14} using the regularity structure and by Gubinelli, Imkeller and Perkowski \cite{GIP2015} by using the framework of para-controlled calculus. Later Hairer and Labb\'e extended the solution theory for the whole $\dR^2$ in \cite{HL2015} and  furthermore for whole $\dR^3$ in \cite{HL2018}.

With some particular choices of random potential, PAM admits an intriguing concentration property for its tall peaks
on large space-time scales which is often referred as \emph{intermittency}. A vast amount of previous works
of the PAM on $\dZ^d$ with i.i.d. potential, and on $\dR^d$ with regular potential, have revealed that the solution of PAM is highly concentrated on few small islands that are far from each other and carry
most of the total mass of the solution. This phenomenon can be attributed to the following spectral representation in terms of the eigenvalues $\lambda_1\geq \lambda_2\geq \cdots$ and corresponding
$L_2$-orthonormal basis of eigenfunctions $e_1, e_2, e_3, .....$ of $\frac{1}{2}\Delta + \xi$,
\begin{align}\label{eq:SpectralRep}
u(t,x) = \sum_{n} e^{tn} e_n(x) e_n(0).
\end{align}   
From this representation, the intermittency of the system comes as consequence of Anderson localization which dictates  the leading eigenfunctions $e_1, e_2,\cdots$
 to be concentrated in small islands. This phenomenon has been proved inside large centered boxes for few instances including the case where $\xi$ an i.i.d. potential on $\dZ^d$ with double exponential tails \cite{BKS18}. See \cite{Konig16} for the past developments on PAM. 

 In the case of white noise potential, the phenomenon of intermittency or Anderson localization makes sense for dimensions $d=1,2,3$. The one dimensional case is well understood due to three beautiful works by Laure Duma\'z and Cyril Labb\'e \cite{DL2020,DL21a,DL21b}. There is no known solution theory for $d\geq 4$ since the PAM with white noise potential is scaling-critical/supercritical in those cases. As we have mentioned earlier, the cases $d = 2, 3$ are dealt with regularity structure as developed by Hairer or paracontrolled calculus by Gubinelli and Perkowski.

Intermittency is intimately tied with macroscopic fractality which was studied in 
\cite{KKX17, KKX18} for a large collection of parabolic stochastic PDEs including the $(1+1)$-d stochastic heat equation with multiplicative space-time white noise. 
They had shown that when the intermittency holds, the peaks of those stochastic PDEs form complex macroscopic multifractal structures. More precisely, their results show that the macroscopic Hausdorff dimension (introduced by Barlow and Taylor \cite{BT89,BT92}, see Definition~\ref{def:definition of dimension}) of the tall peaks take distinct and nontrivial values as the level of the peaks vary, a property which symbolizes the
multifractality. The same phenomenon does not hold in the case of Brownian motion where the tall peaks
demonstrate a constant Hausdorff dimension (see \cite[Theorem 1.4]{KKX17}) along a different length scale.

In a recent work, \cite{KPZ20} showed that the sizes of the tall peaks in boxes of width $t^{\alpha}$ for $\alpha\in (0,1)$ and deep valleys in parabolic Anderson model in 2 dimension are asymptotically same. They have also commented that similar result is expected for PAM in 3 dimension. This property is in apparent contradiction with the intermittency property for the PAM. In this paper, we seek to study the fractality of the PAM. Our main theorems which are stated below shows that the spatial (Theorem~\ref{thm:spatial multifractality}) and spatio-temporal peaks of the PAM (Theorem~\ref{thm:spatio-temporal multifractality}) are macroscopically multifractal (see Section~\ref{sec2} for definition) for $d=2,3$. 

  Before proceeding to the main statement of those results, we introduce few notations. For $\alpha, \beta, v,t > 0$, define the set of peaks 
\begin{equation*}
  \cP^d_t(\alpha): = \left\{ x\in \R^d \, : \, u(t,x) \geq e^{\alpha t(\log|x|)^{\frac{2}{4-d}}} \right\},
\end{equation*}
and 
\begin{equation*}
  \cP^d(\beta, v):= \left\{ (e^{t/v}, x) \in (e,\infty) \times \R^d \, :\, u(t,x) \geq e^{\beta t^{ \frac{6-d}{4-d}} }\right\}. 
\end{equation*}
We also introduce 
\begin{equation}\label{eq:definition of c}
  \fc_d : = \frac{8}{d^{\frac{d}{2}}(4-d)^{2-\frac{d}{2}}\kappa^4_d }, 
\end{equation} where
\begin{equation*}
  \kappa_d: = \sup_{f \in H^1(\dR^d)} \frac{\| f\|_{L^4(\dR^d)}}{\|\nabla f \|^{d/4}_{L^2(\dR^d)} \| f\|^{1-d/4}_{L^2(\dR^d)}}.
\end{equation*}

Our first result which is stated below finds the macroscopic Hausdorff dimension (denoted as $\Dim[\cdot]$) of the peaks of PAM for $d=2,3$ in the spatial direction for all large time $t$. Furthermore, it also finds the asymptotic shape of the peaks in the spatial direction for any fixed large $t$.

\begin{theorem}[\bf Spatial Multifractality and Asymptotics of the PAM]\label{thm:spatial multifractality}
  For $\alpha>0$, there exists $t_0= t_0(\alpha,d)>0$ such that for all $t\geq t_0$, we have 
  \begin{equation}\label{eq:dimension of spatial tall peaks}
    \Dim[\cP^d_t(\alpha) ] =(d-\alpha^{\frac{4-d}{2}} \fc_d) \vee 0, \quad \text{a.s.}
  \end{equation} In addition, there exists $t_1= t_1(d)>0$ such that for all $t>t_1$,
\begin{equation}\label{eq:limsup of spatial tall peaks}
  \limsup_{|x|\rightarrow \infty} \frac{\log_+u(t,x)}{(\log|x|)^{\frac{2}{4-d}}} \stackrel{a.s.}{=}  \left(\frac{d}{\fc_d}\right)^{\frac{1}{2-d/2}}t.
\end{equation}
\end{theorem}

For $d=2$, the long time asymptotics of the solution of PAM had been found in \cite{KPZ20}. They have show that $\sup_{x\in }\log u(t,x)$ is approximately equal to $\frac{2t}{\fc_2}$ as $t$ gets larger when the initial data is Dirac delta. Our result shows that the tall peaks of $\log u(t,x)$ in the spatial direction  take the shape of $\frac{2t}{\fc_2}$ even for finite value of $t$. For one dimensional PAM, similar results were proven by Xia Chen \cite{Chen15} using the moment asymptotics. However, the proof techniques for $d=1$ breaks down in the case of $d=2,3$ since the moments of PAM in dimension larger than $1$ blow up in finite time. This poses a serious technical difficulty which we are able to circumvent in this paper by introducing new techniques.  
 \
Our next result shows the macroscopic Hausdorff dimension of the spatio-temporal peaks of PAM for $d=2,3$. 


\begin{theorem}[\bf Spatio-Temporal Multifractality of the PAM]\label{thm:spatio-temporal multifractality}
  For every $\beta>0$ and $v>0$, we have
\begin{equation}\label{eq:dimension of spatio-temporal tall peaks}
    \Dim[\cP^d(\beta, v)] = (d+1- \beta^{\frac{4-d}{2}} v \fc_d) \vee d, \quad \text{a.s.} 
  \end{equation}
\end{theorem}
Macroscopic fractal dimension of the spatio-temporal peaks of the parabolic stochastic PDEs with multiplicative white noise had been investigated by Khoshnevisan, Kim and Xiao \cite{KKX18}. This class of stochastic PDEs contains the $(1+1)$-dimensional stochastic heat equation with multiplicative spatio-temporal white noise. Recently macroscopic fractal dimension of the the peaks and valleys of $(1+1)$-d Kardar-Parisi-Zhang (KPZ) equation has been found in \cite{DG21,GY21}.  
The case of $(2+1)-$ dimensional stochastic heat equation with spatio-temporal white noise remained completely unclear since the solution theory was only known for the sub-critical regime so far \cite{CD2020,CSZ2020}. Although the solution theory of parabolic Anderson model in $(2+1)$-d and $(3+1)$-d are well studied by now, the depiction of the macroscopic fractal structures in those cases were missing. 
Theorem~\ref{thm:spatio-temporal multifractality} filled this gap by showing the multifractality of spatio-temporal peaks for higher dimensional PAM.

Multifractality of the peaks of intermittent systems were discussed in many occasions in the previous literature including \cite{GD05} in the context of turbulence and \cite{Z.et.Al00} for stochastic Allen-Cahn equation with multiplicative forcing. \cite[Theorem~1.1]{KKX18} showed that the spatio-temporal peaks of the $(1+1)-$dimensional stochastic heat equation (SHE) with space-time white noise form multifractals with peaks of height $e^{\beta t}$ for every $\beta>0$. This result leverages on the (moment) intermittency of the $(1+1)$-d, which means the exponential moment $\dE[u(t,x)^p]$ of the solution behaves as $ \exp ( \gamma(p) t ) $ where $p\mapsto \gamma(p)$ is a strictly convex function (see \cite{CM94}). Indeed, the proof of \cite[Theorem~1.1]{KKX18} utilized this exponential moment to obtain the tail estimates of the solution. However, the moments of parabolic Anderson model in $(2+1)$-d or $(3+1)$-d cases blow up. As a result, the previous approach based on moment intermittency breaks down in those two cases. We rather use the asymptotics of the spectrum of the {\it Anderson Hamiltonian} and the Feynman-Kac representation of the PAM built using the theory of para-controlled distributions \cite{GIP15}. Theorem~\ref{thm:spatio-temporal multifractality} exposes that even though there is no moment intermittency, the spatio-temporal tall peaks of the solution to $(2+1)$-d (resp. $(3+1)$-d) PAM exhibit multifractality of order $e^{\beta t^2} $ (resp. $e^{\beta t^3}$), which displays a chaotic nature of the high dimensional multiplicative noise. In a recent work \cite{GGL22}, the first author of this paper and his collaborators have introduced the idea of \emph{finite time intermittency} for the PAM in higher dimension with asymptotically singular noise. We believe that many of our proof techniques can be extended to study the macroscopic fractality of the peaks in those settings.

\subsection{Proof Ideas}
In this section, we discuss the proof ideas behind Theorem~\ref{thm:spatial multifractality} and Theorem~\ref{thm:spatio-temporal multifractality}. 
Proving fractal dimension of any given set can be done in two steps: first showing a lower bound to the fractal dimension and finally, showing the appropriate upper bound. While showing an appropriate upper bound can pose serious challenges, proving a lower bound to the fractal dimension very often requires more precise insights about the geometry of the associated models. The case of PAM in higher dimension is not exception to this folklore. One of the major challenges in showing both upper and lower bounds to the fractal dimension is to control the tail probabilities of the maximum of PAM solution in compact sets. Since the moments of $(d+1)$-dimensional PAM blows up in finite time when $d=2,3$, previous approaches based on moment asymptotics \cite{KKX18,DG21} fail to work in this situation. To get around this difficulty, we seek to use the connection between the solution of PAM and the spectrum of \emph{Anderson Hamiltonian} as formally stated in \eqref{eq:SpectralRep}. 

Showing a lower bound to the fractal dimension of a set $E\subset \R^d$ requires to show that the associated set is '\emph{sufficiently thick}'. See the illustration in Figure~\ref{fig:Fig2} for the definition of thickness of a set $E$. In order to show enough thickness of $E$, we first embed $E$ into a large $d$-dimensional box and then divide the large box into smaller boxes. It is then sufficient to show there are enough of such small boxes which carries the points from $E$ or rather the probability of the set of points of $E$ escaping most of those small boxes is close to $0$. Controlling this probability will require two important ingredients which were lacking before: $(a)$ near independence between the solutions of PAM restricted on any two such smaller boxes and $(b)$ upper bound on the probability of the PAM to be bounded above by a large value.

\begin{figure}
\hspace{2cm}\includegraphics[height=7cm, width = 10cm]{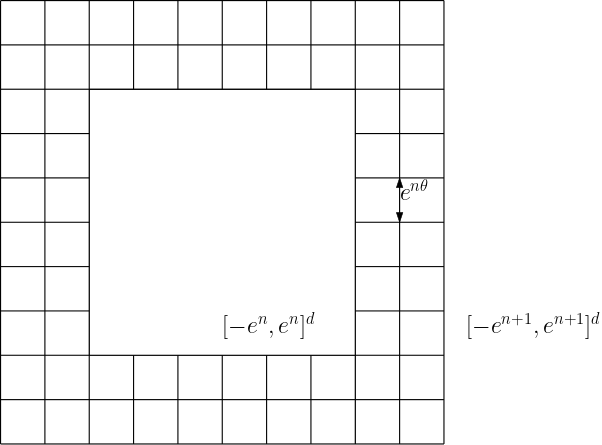}
\caption{A set $E$ is called $\theta$-thick for some $\theta\in (0,1)$ if $E$ contains points each cell of side-length $e^{n\theta}$ in the outer shell of $[-e^{n},e^{n}]^d$ for all large $n$.}
\label{fig:Fig2}
\end{figure}

On the other hand, following the definition of macroscopic Hausdorff dimension from \ref{def:definition of dimension}, $\Dim(E)$ is upper bounded by $\rho$ if $\rho$-dimensional Hausdorff content of $E$ often computed as $\dE[\sum_{n=1}^{\infty} \nu^{n}_{\rho}(E)]$ is finite. See the paragraph before Definition~\ref{def:definition of dimension} for $\nu^{n}_{\rho}(E)$. Bounding the above expected value requires one more important ingredient which is $(c)$ to bound  the tail probability of the supremum of the solution of PAM in a small ball.

We obtain those ingredients via combinations of different tools that we develop throughout the paper. Inception of these tools and carrying out the rest of proof of our main results can be broadly divided into three steps: the first step is to show appropriate bounds on the solution of PAM in terms of spectrum of Anderson Hamiltonian, the second step is to derive some tail probability on the solution and the third is to integrate the first two steps with a series expansion (coming from Feynman-Kac representation of the PAM in higher dimension) of the solution to complete the proof. Below we discuss each steps in more details. Figure~\ref{fig:Fig1} describes schematic representation of where the different tools are introduced and how they are combined to prove Theorem~\ref{thm:spatial multifractality} and~\ref{thm:spatio-temporal multifractality}.

\textbf{Step 1.} We derived appropriate bound on the solution  of PAM in Section~\ref{subsec:bound on the enhanced noise}. We mainly use three tools to show such bounds on the solution. These three tools are respectively, the Feynman-Kac representation of the solution of PAM, transition kernel estimates and appropriate bounds on the noise. The Feynman-Kac representation of the solution is derived in Theorem~\ref{thm:Feynman-Kac representation} which shows that the solution of \eqref{eq:PAM} $u^{\phi}_{L,y}$ started from the initial data $\phi$ and restricted on $y+[-\frac{L}{2}, \frac{L}{2}]^d$ with Dirichlet boundary condition can be written as 
\begin{align}\label{eq:modified Feynman-Kac in intro}
    u^{\phi}_{L,y}(t,x) = \mathbb{E}\Big[ \exp\left( \int_r^t (Z^y_L+\eta Y^y_L)(X_s)ds + (Z^y_L+Y^y_L)(X_r)-(Z^y_L+Y^y_L)(X_t) \right) \phi(X_t)\1^X\Big]
\end{align}
where $\eta>0$ is a small number,  $\1^X:= \1^X_{X_{[0,t]} \subset y+[-\frac{L}{2}, \frac{L}{2}]^d}$ and $X_t$ is a diffusion defined by $$X_t = x+ \int^{t}_0\nabla(Z^y_L+Y^y_L)(X_s)ds + B_t$$ such that $B_t$ is Brownian motion independent of $Z^y_L$ and $Y^y_L$ where $Z:= (1-\frac{1}{2}\Delta)^{-1}\xi\in \sC^{\frac{1}{2}-}$ and $Y$ solves $$(\eta -\frac{1}{2})Y^y_L=\frac{1}{2}|\nabla Z^y_L|^2+ \nabla Y^y_L\cdot \nabla Z^y_L + \frac{1}{2} |\nabla Y^y_L|^2.$$ These two random processes are introduced in Proposition~\ref{prop:resolvent eqaution} of Section~\ref{sec:feynman-kac representation}. The reason that the expression \eqref{eq:modified Feynman-Kac in intro} differs from the classical Feynman-Kac representation is the roughness of the noise $\xi$ (see \eqref{eq:classical Feynman-Kac} and the following discussion). We showed the equivalence between the classical form of Feynman-Kac and the modified form by using the Girsanov's theorem along the similar line as in \cite{GP17}. Similar results have been shown for $d=2$ by \cite{KPZ20}. However the $d=3$ case requires handling of sufficient technical difficulties which has been overcome in the present paper using similar tools as in \cite{CC18} based on para-controlled distributions. See Remark~\ref{rem:3D_Feynman_Kac} for more details. 

The next main tool is the bound on the transition density of the diffusion $X_t$. To this end, the transition density of $X_t$ is a solution of Cauchy problem as shown in \eqref{eq:cauchy problem} (see Theorem~\ref{thm:well-posedenss of martingale problem}). Since $\nabla (Z^y_L+Y^y_L)$ is distribution valued, it requires non-trivial fixed point argument to show that the transition density kernel exists. We first lift $\nabla (Z^y_L+Y^y_L)$ in the space of rough distributions and then employ tools from para-controlled calculus to achieve this in Section~\ref{sec:feynman-kac representation}. Similar problem had been considered before in \cite{CC18}. On the way of proving our result, we have extended their result (especially \cite[Theorem~3.10]{CC18}) to cover the singular initial data case.

Once the existence is shown, the upper and lower bound on the transition kernel is derived using the ideas of \cite{Str08}. The next main tool is to bound the mollified noise $\xi_{\epsilon}$, or more precisely $(1-\frac{1}{2}\Delta)^{-1}\xi_{\epsilon}$ uniformly in $\epsilon$ using hyper-contractivity of Gaussian noise (see Proposition~\ref{prop:bounds on the enhance noise} and~\ref{prop:bounds on Z,Y,eta}).

\textbf{Step 2.} The second step is to derive the tail probabilities of the solution of PAM in $(d+1)$ case where $d=2,3$. This is shown in Proposition~\ref{prop:left tail probability of spatial maximum} \& \ref{prop:right tail probability of spatial maximum}. More precisely, we find the tail probability of the supremum value of PAM where the supremum is taken over a finite set of points. The main idea behind its proof lies in using a local 'representations' of the solution of PAM. Construction of such proxy is done via the Feynman-Kac representation obtained in Theorem~\ref{thm:Feynman-Kac representation}. In more concrete words, the diffusion $X_t$ in the Feynman-Kac formula could be restricted into a set of disjoint but space-filling boxes to write $u(t,x) $ as a sum of local representations like $u^{\phi}_{L,y} $ of \eqref{eq:modified Feynman-Kac in intro}. 
This is termed as series expansion of the solution of PAM on $\dR_{\geq 0}\times \dR^d$. See Lemma~\ref{lemma:convergence of series to the solution} for more details. By construction,  those local representations of $u(t,x)$ are independent when they are taken from two separate far away boxes. Furthermore, those could be bounded from above and below by some functional of the largest point in the spectrum of Anderson Hamiltonian as shown in Proposition~\ref{prop:upper and lower bound of the localized solution}.   
Finally the tail probabilities of the solution of PAM restricted in finite boxes are found in terms of the tail probabilities of the largest point in the spectrum of the Anderson Hamiltonian and tail probabilities of $(1-\frac{1}{2}\Delta)^{-1}\xi_{\epsilon}$ obtained from Lemma~\ref{prop:bounds on the enhance noise}. Tail probabilities of the largest point in the spectrum were investigated before in many occasions in the past (see Proposition~\ref{lemma:tail probability of eigenvalue}). 

\textbf{Step 3.} The third step is to complete the proof of Theorem~\ref{thm:spatial multifractality} and~\ref{thm:spatio-temporal multifractality} using the tail probabilities and series expansion of $u(t,x)$ of Lemma~\ref{lemma:convergence of series to the solution}. This is mainly done in Section~\ref{sec:proof of spatial multifractality in 2d} and~\ref{sec:proof of spatio-temporal multifractality}. As we have indicated earlier, the proof of the lower bound in fractal dimension goes by showing that the probability of the maximum of $u(t,x)$ over a finite set of points (satisfying the conditions in Proposition~\ref{prop:left tail probability of spatial maximum} or~\ref{prop:left tail probability of spatio-temporal maximum}) being less than a certain value decays fast to $0$. Since the \emph{local representations} of $u(t,x)$ around those points (as described in \textbf{Step 2}) can be made independent and the tail probabilities of the local representations are determined through Proposition~\ref{prop:upper and lower bound of the localized solution} and~\ref{lemma:tail probability of eigenvalue}, these two tools are combined to bound the lower tail probability of the maximum of $u(t,x)$ which finally leads to the lower bound in Hausdorff dimensions in Theorem~\ref{thm:spatial multifractality} and~\ref{thm:spatio-temporal multifractality}. The upper bound part of Theorem~\ref{thm:spatial multifractality} and~\ref{thm:spatio-temporal multifractality} is proved using the set of tools from \textbf{Step 1} and \textbf{Step 2}. These tools provide the upper tail probability of the maximum of $u(t,x)$ in compact sets which is used to control the expected value of Hausdorff contents of the given level sets.


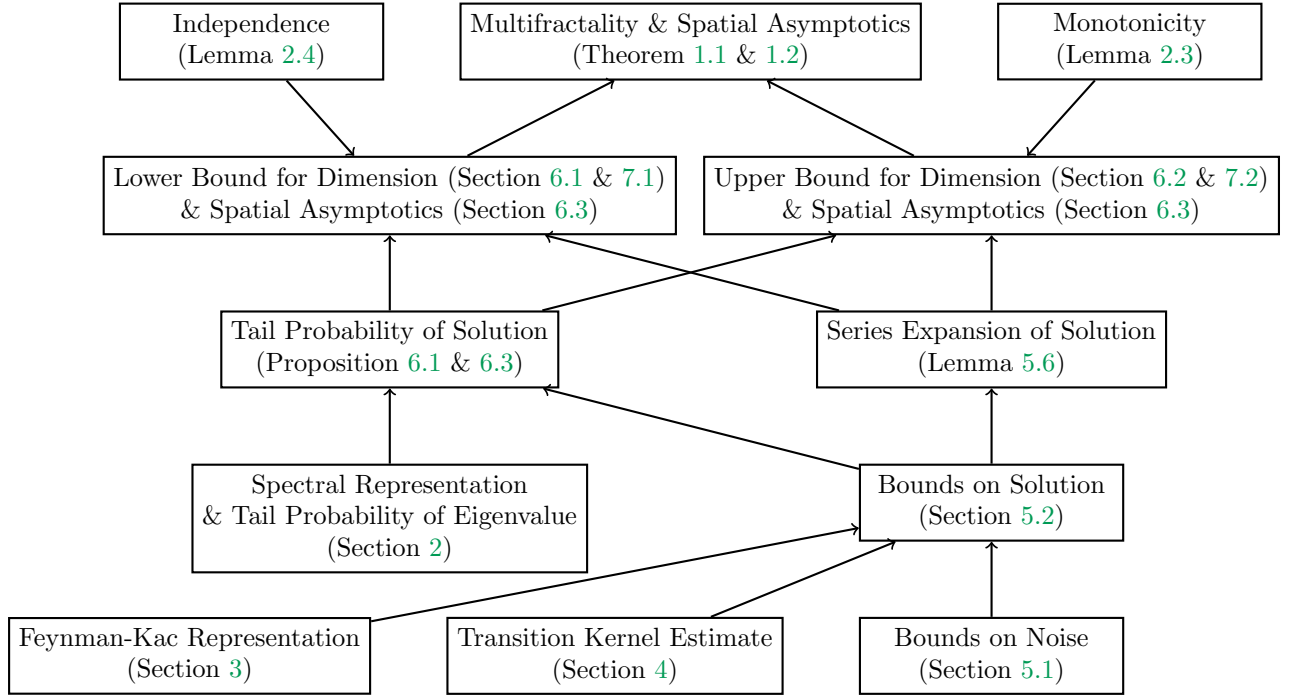
\begin{figure}
\begin{center}
\begin{tikzpicture}[font=\small,thick]

    \node[ draw,
    align=center,
    minimum width=3.5cm,
    minimum height=1cm
] (block1) { Multifractality \& Spatial Asymptotics  \\ (Theorem~\ref{thm:spatial multifractality} \& \ref{thm:spatio-temporal multifractality})};

\node[draw,
align=center,
below =of block1,
xshift=-4cm,
 minimum width=3.5cm,
 minimum height=1cm
] (block2) { Lower Bound for Dimension (Section~\ref{subsec:lower bound in thm1} \& \ref{subsec:lower bound in thm2}) \\ \& Spatial Asymptotics (Section~\ref{subsec:spatial asymptotics})};

\node[draw,
    align=center,
   below=of block1,
   xshift=4cm,
    minimum width=3.5cm,
    minimum height=1cm
] (block3) { Upper Bound for Dimension (Section~\ref{subsec:upper bound in thm1} \& \ref{subsec:upper bound in thm2}) \\ \& Spatial Asymptotics (Section~\ref{subsec:spatial asymptotics})};

\node[draw, align=center, left= of block1, minimum width=3.5cm,
    minimum height=1cm ] (independence) { Independence \\ (Lemma~\ref{lem:independence of eigenvalues})};

\node[draw, align=center, right= of block1, minimum width=3.5cm,
    minimum height=1cm ] (monotonicity) { Monotonicity \\ (Lemma~\ref{lemma:monotonicity of eigenvalue}) };

\node[draw, align=center, below= of block2, minimum width=3.5cm,
    minimum height=1cm ] (tail_sol) { Tail Probability of Solution \\ (Proposition~\ref{prop:left tail probability of spatial maximum} \& \ref{prop:right tail probability of spatial maximum})};
        
\node[draw, align=center, below= of block3, minimum width=3.5cm,
    minimum height=1cm ] (series) { Series Expansion of Solution \\ (Lemma~\ref{lemma:convergence of series to the solution})};

\node[draw, align=center, below= of tail_sol, minimum width=3.5cm,
    minimum height=1cm ] (tail_eigen) {Spectral Representation \\
    \& Tail Probability of Eigenvalue \\ (Section~\ref{sec:AndersonHamiltonian})};

\node[draw, align=center, below= of series, minimum width=3.5cm,
    minimum height=1cm ] (bound_sol) {Bounds on Solution \\ (Section~\ref{subsec:reduction to a box and asymtotic bounds for the solution})};

\node[draw, align=center, below= of bound_sol, minimum width=3.5cm,
    minimum height=1cm ] (bound_noise) {Bounds on Noise \\ (Section~\ref{subsec:bound on the enhanced noise})};

\node[draw, align=center, left= of bound_noise, minimum width=3.5cm,
    minimum height=1cm ] (transition) {Transition Kernel Estimate \\ (Section~\ref{sec:transition density estimate})};
    
\node[draw, align=center, left= of transition, minimum width=3.5cm,
    minimum height=1cm ] (feynman-kac) {Feynman-Kac Representation \\ (Section~\ref{sec:feynman-kac representation}) };

\draw[->] (block2) -- (block1);
\draw[->] (block3) -- (block1);
\draw[->] (independence) -- (block2);
\draw[->] (monotonicity) -- (block3);
\draw[->] (tail_sol) -- (block2);
\draw[->] (tail_sol) -- (block3);
\draw[->] (series) -- (block2);
\draw[->] (series) -- (block3);
\draw[->] (tail_eigen) -- (tail_sol);
\draw[->] (bound_sol) -- (tail_sol);
\draw[->] (bound_sol) -- (series);
\draw[->] (bound_noise) -- (bound_sol);
\draw[->] (feynman-kac) -- (bound_sol);
\draw[->] (transition) -- (bound_sol);

\end{tikzpicture}
\end{center}
\caption{Flowchart of the proof of Theorem~\ref{thm:spatial multifractality} and \ref{thm:spatio-temporal multifractality}}
\label{fig:Fig1}
\end{figure}

\subsection*{Acknowledgement}
 PG was supported by NSF grant DMS-2153661 and JY was supported by Samsung Science and Technology Foundation under Project Number SSTF-BA1401-51. Part of the research of this paper was done when JY was visiting the department of Mathematics of MIT during the summer of 2022.

\section{Spectrum of the Anderson Hamiltonian}\label{sec:AndersonHamiltonian}
In this section, we discuss some preliminary facts about Anderson Hamiltonian and parabolic Anderson model which will be used throughout the rest of the paper. On our way, we introduce many notations, provide the context of their use in later sections and explain their roles in proving the main results of this paper.   

We define $Q_L(d) := [-\frac{L}{2}, \frac{L}{2}]^d \subset \dR^d$. We often use the symbol $Q_L$ in place of $Q_L(d)$ and the value $d$ will be clear from the context. For any $y\in \dR^d$, we set $Q^y_L(d): = y+ Q_L(d)$.
We consider the PAM with the Dirichlet boundary condition on $Q^y_L(d)$ started from initial data $\phi$ and with \emph{enhanced noise} $\xi_L$ as constructed in \cite[Section~6]{CZ21}
\begin{equation}\label{eq:localized PAM}
     \begin{cases} {\partial \over \partial t}u^\phi_{L,y}(t,x) =\frac{1}{2}\Delta u^\phi_{L,y} (t,x) + u^\phi_{L,y}(t,x)\xi_{L,y}(x), \quad t>0, x\in Q^y_L,\\
   u(0,x) = \phi, \quad x\in Q^y_L, \quad \text{and} \quad u^\phi_{L,y}\mid_{\partial Q^y_L} =0,
   \end{cases}
   \end{equation}
Construction of $\xi_L$ goes by first projecting the white noise on the Neumann space of the box and then take
the regularisation corresponding to a Fourier multiplier. Fix any even function $\tau \in C^{\infty}_c(^d, [0,1])$ and define
$$\xi^{y}_{L,\varepsilon} = \sum_{k\in \mathbb{N}^d} \tau\big(\frac{\varepsilon}{L}k\big) \langle \xi, \mathfrak{n}_{k,L}\rangle \mathfrak{n}_{k,L}, \quad \mathfrak{n}_{k,L}(x) := 2^{-\frac{1}{2}\{i:k_i=0\}} \mathbbm{1}(x\in Q^y_L(d))\big(\frac{2}{L}\big)^d \prod_{i=1}^d\cos\big(\frac{\pi}{L}k_ix_i\big).$$
Theorem 6.7 of \cite{CZ21} shows that $\xi^{y}_{L,\varepsilon}$ converges almost surely to the white noise $\xi^y_{L}\in \sC^{\alpha}$ as $\varepsilon$ goes to $0$ for $\alpha<-\frac{d}{2}$ and the limit does not depend on the choice of $\tau$. 
  The solution of \eqref{eq:localized PAM} could be realized as the weak limit of the following system where we replace $\xi^y_{L}$ by $\xi^{y}_{L,\varepsilon}$, i.e., 
  \begin{equation}\label{eq:localized PAM_renorm}
     \begin{cases} {\partial \over \partial t}u^{\phi,y}_{L,\varepsilon}(t,x) =\frac{1}{2}\Delta u^{\phi,y}_{L,\varepsilon} (t,x) + u^{\phi,y}_{L,\varepsilon}(t,x)(\xi^{y}_{L,\varepsilon}(x)-c_{\varepsilon}), \quad t>0, x\in Q^y_L,\\
   u(0,x) = \phi, \quad x\in Q^y_L(d), \quad \text{and} \quad u^\phi_{L,y}\mid_{\partial Q^y_L(d)} =0,
   \end{cases}
   \end{equation}
where $c_{\varepsilon}$ denotes the renormalization constant which we set as $c_{\varepsilon}= \frac{1}{2\pi}\log \frac{1}{\varepsilon}$. It has been shown in Section~2 of \cite{KPZ20} for all $T>0$,  $u^{\phi,y}_{L,\varepsilon}(t,x)$ converges in $C([0,T], B^{\varrho, \beta}_{\infty,\infty}(Q^y_L(d)))$ uniformly on $[0,T]\times Q^y_L(d)$ in probability to $u^{\phi}_{L,y}$. Here $B^{\varrho, \beta}_{\infty,\infty}(Q^y_L(d))$ is the Dirichlet Besov space defined in \cite[Section~4]{CZ21}. 
   
Our next goal is to introduce the spectral representation of $u^{\phi}_{L,y}$ in $B^{\varrho, \beta}_{\infty,\infty}(Q^y_L(d))$ in terms of the spectrum of Anderson Hamiltonian. To this end, we recall definition of the Anderson Hamiltonian operator from \cite{CZ21} and \cite{Lab19}. Theorem~5.4 of \cite{CZ21} characterized the spectrum of Anderson Hamiltonian for $d=2$ case wheres \cite[Theorem~1]{Lab19} did the same for $d=3$. We summarize below their results. Denote the enhancement of $B^{\varrho, \beta}_{\alpha,\infty}(Q^y_L(d))$ by $\fX^\alpha(Q^y_L(d))$ and their respective Neumann extensions as  $B^{\varrho, \beta}_{\alpha,\infty,\mathfrak{n}}(Q^y_L(d))$ and $\fX^\alpha_\fn(Q^y_L(d))$.

 Let $L>0$, $y\in \dR^d$ and $\xi$ be a $d-$dimensional spatial white noise. In dimension $d\in\{2,3\}$, there exists $\sH_{\xi}$ that is densely defined on $L^2(Q^y_L(d))$, a closed and self-adjoint operator given by 
  \begin{equation*}
     \sH_{\xi} = \Delta u + \xi u,
   \end{equation*} with values in $L^2(Q^y_L(d))$. $\sH_{\xi}$ has a pure spectrum consisting of eigenvalues $\blambda_1(Q^y_L(d))>\blambda_2(Q^y_L(d))\geq\blambda_3(Q^y_L(d))\geq \cdots$. We let $v^y_{n,L}$ be an eigenvector with eigenvalue $\blambda_n(Q^y_L(d))$ such that $\{v^y_{n,L}\}_{n\in\dN}$ is an orthonormal basis of $L^2(Q^y_L(d)).$  Due to the lack of regularity of $\xi$, the product $\xi u$ is not well-defined in a classical sense. For the rigorous definition of the product, we refer the readers to \cite[Theorem~5.4]{CZ21} and \cite[Theorem~1]{Lab19}.

\begin{lemma}[{\bf Spectral representation}, Lemma~2.11 and Theorem~2.12 of \cite{KPZ20}]\label{lemma:spectral representation} For $L,t>0, y\in \R^d$ and $\phi \in L^2(Q^y_L(d))$, we have 
\begin{equation}\label{eq:spectral representation}
  u^\phi_{L,y}(t,\cdot) = \sum_{n \in \dN} e^{t \blambda_{n} (Q^y_L(d))} \langle v^y_{n,L},\phi \rangle v^y_{n,L}.
\end{equation} Moreover, this representation holds for $\phi= \delta_z$. In other words,
\begin{equation}\label{eq:spectral representation with delta initial data}
   u^{\delta_z}_{L,y}(t,x) = \sum_{n \in \dN} e^{t \blambda_{n} (Q^y_L(d))}  v^y_{n,L}(z) v^y_{n,L}(x), \quad \text{for }x,z\in Q^y_L(d).
\end{equation}
  
\end{lemma}
\begin{proof} 
  We refer to \cite[Lemma~2.11, Theorem~2.12]{KPZ20} for the proof in the case of $d=2$. For $d=3$, the proof follows from a similar argument as when $d=2$ combining \cite[Theorem~1.1]{Lab19} and Theorem~\ref{thm:extension to the delta initial data}.
\end{proof}

Eigenvalues of the Anderson Hamiltonian play very important roles in proving our main results. In Section~\ref{sec:asymptotic bounds for the pam}, we discuss how  the supremum of the PAM restricted over a growing rectangle can be described in term of the largest eigenvalue of the Anderson Hamiltonian.  In the next three results, we record some useful properties of the eigenvalues of $\sH_\xi$ such as the tail probabilities, monotonocity and the independence. These results will be instrumental in obtaining asymptotics of the solution of PAM in Section~\ref{sec:asymptotic bounds for the pam}.

The first result is about the tail probabilities of  the eigenvalues of Anderson Hamiltonian. We refer to \cite[Theorem~2.17]{CZ21} for $d=2$ case and \cite[Theorem~2]{HL22} for $d=3$ case. 
\begin{proposition}\label{lemma:tail probability of eigenvalue}
  Fix $\epsilon \in (0,1)$. There exist $c_2>c_1>0$ and $s_0$ such that for all $L\geq 1$ and $s\geq s_0$ 
  \begin{align}
    \P\left( \blambda_1(Q^y_L(d))) \leq s\right) &\leq \exp \left(-c_2 s^{d/2} e ^{d\log L -(1+\epsilon)\fc_d s^{2-d/2}} \right),\\
    \P\left(  \blambda_1(Q^y_L(d))) \geq  s\right) &\leq c_1s^{\frac{d}{2}}e^{d\log L- (1-\epsilon) \fc_d s^{2-d/2}},
  \end{align} where $\fc_d$ is defined in \eqref{eq:definition of c}.
\end{proposition}

The second result says that the eigenvalues of Anderson Hamiltonian grows monotonically as the size of the box grows. For the proof of this result, we refer to \cite[Theorem~8.6]{CZ21} for $d=2$ case and \cite[Proposition~2.1]{HL22} for $d=3$ case.
\begin{proposition}[{\bf Monotonocity of eigenvalues}]\label{lemma:monotonicity of eigenvalue}
  Let $L\geq r \geq 1 $. For all $x,y \in \R^d$ such that $Q^y_r(d) \subseteq Q^x_L(d)$,
  \begin{equation}
    \blambda_{1}(Q^y_r(d)) \leq \blambda_{1}(Q^x_L(d)).
  \end{equation}
\end{proposition} 

The final result of this section is about the domain Markov property for the spectrum of the Anderson Hamiltonian with white-noise potential, i.e., the spectrum of $\sH_\xi$ restricted on two disjoint regions are independent of each other. We refer to Lemma~7.4 of \cite{CZ21} for the proof. 
\begin{proposition}{\bf Independence of eigenvalues}]\label{lem:independence of eigenvalues}
  Suppose that $y_1,...,y_m\in \dR^d$ satisfy $\min_{1\leq i\neq j \leq m}|y_i-y_j| \geq 3L $ and $x_i \in Q^{y_i}_L$ for each $i$. For $1\leq i \neq j \leq m$,  $\blambda_{n} (Q^{y_i}_{L}(d))$ and $\blambda_{n} (Q^{y_j}_{L}(d))$ are independent. 
\end{proposition}

\section{Feynman-Kac Representation}\label{sec:feynman-kac representation}


This section is devoted to proving the Feynman-Kac representation of the Anderson Hamiltonian in 3 dimension. The main purpose of deriving such Feynman-Kac representation is to derive useful upper bound on the tail probabilities of the eigenvalues of Anderson Hamiltonian. In the 2d case, \cite[Theorem~2.17]{KPZ20} provides the modified version of the Feynman-Kac representation using Girsanov's transformation which has been used to derive tail probabilities.
However, this upper bound works only when $d=2$. We slightly refine the representation to cover the $3$-dimensional case. Before we present the representation, we introduce an equation related to the noise. For the rest of this section, we only consider the case of $d=3$.

\begin{proposition}[{\bf Resolvent equation}]\label{prop:resolvent eqaution}
  Let $L>0$ $y\in \dR^d$, $\alpha\in(\frac{2}{5}, \frac{1}{2})$ and $\xi^y_L$ be the spatial white noise on $Q^y_L$. Set $Z^y_L: = (1-\frac{1}{2}\Delta)^{-1} \xi^y_L$. Then there exists $\eta_L>0$ such that for all $\eta\geq \eta_L$ there exists a unique solution $Y^y_L\in \sC^{2\alpha}$ to the following resolvent equation
\begin{equation}\label{eq:original resolvent eq for Y}
  (\eta - \frac{1}{2}\Delta ) Y^y_L = \frac{1}{2} |\nabla Z^y_L|^2  + \nabla Y^y_L \cdot \nabla Z^y_L +\frac{1}{2}|\nabla Y^y_L|^2.
\end{equation} Furthermore, if $\{\xi^y_{L,\epsilon}\}_{\epsilon}$ is a mollification of $\xi^y_L$ such that  $\xi^y_{L,\epsilon}  \rightarrow \xi^y_L$ in  $\sC^{\alpha-2}$, then $Y^y_{L,\epsilon} \rightarrow Y^y_L $ in $\sC^{2\alpha}$.
\end{proposition}
 Before proceeding to the proof of the above proposition, we state below the main result of this section which shows a Feynman-Kac representation of $u^{\phi,y}_{L,\epsilon}$.    

\begin{theorem}[{\bf Modified Feynman-Kac representation}]\label{thm:Feynman-Kac representation}
  For $L,t>0, y\in \dR^d, \epsilon\in (0,1]$ and $\phi \in C_b(Q^y_L)$, we have
\begin{equation}\label{eq:feynman-Kac representation}
  u^{\phi,y}_{L,\epsilon}(t,x) = \E_{\dQ^{x,y}_{L,\epsilon}} \left[\sD^y_{L,\epsilon}(0,t) \phi(X_t) \1_{X_{[0,t]} \subset Q^y_L}  \right], \quad \text{for }x\in Q^y_L,
\end{equation} where $\sD^y_{L,\epsilon}(r,t)$ for $r,t\in\dR_+$ is defined by 
\begin{equation}\label{eq:exponent in Feynman-Kac representation}
  \sD^y_{L,\epsilon}(r,t):= \exp\left( \int_r^t (Z^y_{L,\epsilon}+\eta Y^y_{L,\epsilon})(X_s)ds + (Z^y_{L,\epsilon}+Y^y_{L,\epsilon})(X_r)-(Z^y_{L,\epsilon}+Y^y_{L,\epsilon})(X_t) \right) 
\end{equation} with $\eta>0$ and $\dQ^{x,y}_{L,\epsilon}$ be the probability measure on $C([0,\infty), \dR^d)$ such that the coordinate process $X_t$ satisfies $\dQ^{x,y}_{L,\epsilon}$-almost surely
\begin{equation}\label{eq:diffusion_X}
  X_t=x+ \int_0^t \nabla(Z^y_{L,\epsilon} + Y^y_{L,\epsilon})(X_s)ds + B'_t, \quad t\geq0,
\end{equation} for a Brownian motion $B'$. Moreover, if $\eta = \eta_L$ as in Proposition~\ref{prop:resolvent eqaution}, we have 
\begin{equation}\label{eq:convergence of Feynman-Kac}
    \lim_{\epsilon\rightarrow 0} \E_{\dQ^{x,y}_{L,\epsilon}} \left[\sD^y_{L,\epsilon}(0,t) \phi(X_t) \1_{X_{[0,t]} \subset Q^y_L}  \right] = \E_{\dQ^{x,y}_{L}} \left[\sD^y_{L}(0,t) \phi(X_t) \1_{X_{[0,t]} \subset Q^y_L}  \right] = u^{\phi,y}_L(t,x).
\end{equation} 
\end{theorem}
Notice that the classical Feynman-Kac representation for the smooth mollifier $\xi^y_{L,\epsilon}$ of $\xi^y_L$ is different than \eqref{eq:feynman-Kac representation}.  Let $L\in (0,\infty)$ and $y\in \dR^d$. For $\phi \in C_b(Q^y_L)$, $\epsilon>0$, and $(t,x) \in \dR_+ \times Q^y_L$, the classical Feynman-Kac formula takes the form 
\begin{equation}\label{eq:classical Feynman-Kac}
    u^{\phi, y}_{L,\epsilon}(t,x) = \dE_x \left[ \exp \left( \int^t_0(\xi^y_{L,\epsilon} - c_\epsilon ) (B_s) ds  \right)  \phi(B_t) \1_{B[0,t] \subset Q^y_L}\right]
\end{equation} where $B[0,t] : = \{ B(s) : s \in [0,t]\}$. Due to low regularity of the spatial white noise $\xi^y_{L}$, the $L_\infty-$norm of $\xi^y_{L,\epsilon} - c_\epsilon  $ blows up as $\epsilon\rightarrow 0$. We deal with this difficulty by adopting the ideas from the proof of \cite[Lemma~2.16, Theorem~2.17]{KPZ20}. In a similar way as \cite{KPZ20}, we use Girsanov's transform and Proposition~\ref{prop:resolvent eqaution} to obtain a modified version of the Feynman-Kac representation.
\begin{remark}\label{rem:3D_Feynman_Kac}
Theorem~\ref{thm:Feynman-Kac representation} is a variant of \cite[Theorem~2.17]{KPZ20} in dimension $3$. The difference in the Feynman-Kac representation of \cite[Theorem~2.17]{KPZ20} and the one in \eqref{eq:feynman-Kac representation} is the absence of the term $\frac{1}{2}|\nabla Y|^2$ (see (24) of \cite{KPZ20}). Note that since the expected regularity of $Y$ is $1^-$, $\frac{1}{2}|\nabla Y|^2$ is not controlled in $L_\infty-$norm. Thus, we used the partial Girsanov theorem 
via solving the resolvent equation \eqref{eq:original resolvent eq for Y} which is different from the expression in (23) of \cite{KPZ20}. 

\end{remark}

Without loss of generality, we drop the superscript $y$ and let $y=0$ for convenience. All the results of this section can be extended for any $y\in\dR^d$ easily. We prove Theorem~\ref{thm:Feynman-Kac representation} after showing Proposition~\ref{prop:resolvent eqaution}.
To this end, that the products $|\nabla Z_L|^2$ and $\nabla Y_L \cdot \nabla Z_L$ in Proposition~\ref{prop:resolvent eqaution} are ill-defined due to the lack of the regularity. Indeed, for instance in $d=3$, $Z_L \in \sC^{\frac{1}{2}^- }$ and the expected regularity of $Y_L$ is $1^-$. Thus the sum of the regularity of $\nabla Z_L$ and $\nabla Y_L$ is negative which makes the solution of \eqref{eq:original resolvent eq for Y} ill-posed.  In order to overcome this difficulty, we use the idea of {\it enhancing the noise} following similar arguments as in \cite[Section~6]{CC18}.

\begin{definition}[\bf Enhanced noise]\label{def:enhanced noise} Let $L>0$ and $\varrho<1/2$. For $(a,b,\theta)\in \dR\times \dR\times \sC^2$, define $\fZ_{L}$ as 
\begin{equation}
 \fZ_L := \fZ_{L}(a,b,\theta) := ( Z_L, Z_L^{\tree} -a, Z_L^{\ttree}, Z_L^{\tttree} , Z_L^{\Tree}-b,\nabla Q_L \circ \nabla Z_L ),
\end{equation} where $\cI_\eta : = (\eta -\frac{1}{2} \Delta )^{-1}$, 

\begin{equation}\label{def:DefinitionOfTree}
  \begin{aligned}
    &Z_L : = \cI_\eta (\theta), \quad Z^{\tree}_L : = \cI_\eta (|\nabla Z_L|^2),\\
    &Z_L^{\ttree} : = \cI_\eta( \nabla Z_L^{\tree}\cdot \nabla Z), \quad Z_L^{\tttree}: = \cI_\eta(\nabla Z_L^{\ttree} \cdot \nabla  Z_L),\\
    &Z_L^{\Tree}: = \cI_\eta (|\nabla Z_L^{\tree}|^2),
  \end{aligned}
\end{equation} and 
\begin{equation*}
  \cQ_L : = \cI(\nabla Z_L), \qquad \nabla \cQ_L\circ \nabla Z_L := (\partial_i(\cQ_L)^j \circ \partial_iZ_L)_{i,j=1,2,3}.
\end{equation*} We define the space $\cZ_L^{\varrho}$ of enhanced noise as 
\begin{equation*}
  \cZ_L^{\varrho}: = \textrm{cl}_{\sH^{\varrho}}\{\fZ_L(\theta,a,b) : (a,b,\theta)\in \dR\times\dR \times \sC^2 \},
\end{equation*} where $\textrm{cl}_{\sH^{\varrho}}$ denotes the closure with respect to the topology of $\sH^{\varrho} : = \sC_\fn^{\varrho} \times \sC_\fn^{2\varrho} \times \sC_\fn^{3\varrho} \times \sC_\fn^{\varrho+1} \times \sC_\fn^{4\varrho} \times \sC_\fn^{2\varrho-1}$ equipped the usual norm. We call $\fZ$ a enhancement of $\theta$. 
  
\end{definition}

In the sequel, we fix $\theta= \theta_{L,\epsilon} = \xi_{L,\epsilon}  $ so that $Z_{L,\epsilon} : = \cI_\eta( \xi_{L,\epsilon})$ where $\{ \xi_{L,\epsilon}\}_{\epsilon\in (0,1]}$ is a mollification of the spatial white noise $\xi_L$ restricted on $Q_L$. The following theorem ensures that $Z_L = (1-\frac{1}{2}\Delta)^{-1} \xi_L$ can be enhanced.

\begin{theorem}[Theorem~6.12 of \cite{CC18} ]\label{thm:renormalization of enhanced noise} Let $\varrho<\frac{1}{2}$ and $\xi_L $ is the spatial white noise on $Q_L$ on some probability space $(\Omega , \cF , \dP_\xi)$. Then there exists a mollification $\{\xi_{L,\epsilon} \}_{\epsilon\in(0,1]}$ such that there exist the renormalizing constants $c^{\tree}_\epsilon, c^{\Tree}_\epsilon\in\dR$ (not depending on $L$) and the sequence 
\begin{equation}
     \fZ_\epsilon := ( Z_{L,\epsilon}, Z_{L,\epsilon}^{\tree} -c^{\tree}_{L,\epsilon}, Z_{L,\epsilon}^{\ttree}, Z_{L,\epsilon}^{\tttree} , Z_{L,\epsilon}^{\Tree}-c^{\Tree}_\epsilon,\nabla \cQ_{L,\epsilon} \circ \nabla Z_{L,\epsilon} )
\end{equation} which converges to a limit $\fZ_L:= ( Z_L, Z_L^{\tree} , Z_L^{\ttree}, Z_L^{\tttree} , Z_L^{\Tree},\nabla \cQ_L \circ \nabla Z_L )\in \sH^\varrho$ in  $L^p(\Omega, \sH^\varrho)$ for every $p>1.$ 

\end{theorem}

In order to prove Proposition~\ref{prop:resolvent eqaution}, we will prove a fixed point problem by using the theory of paracontrolled distributions. This is an analogous result of \cite[Proposition~6.8]{CC18}. At this moment, we omit the subscripts for simplicity like $Z = Z_{L,\epsilon}$. We rewrite \eqref{eq:original resolvent eq for Y} as 
  \begin{equation}\label{eq:resolvent eq for Y}
    Y = \cI_\eta \Big( \frac{1}{2} |\nabla Z|^2 -c  + \nabla Y \cdot \nabla Z +\frac{1}{2}|\nabla Y|^2 \Big).
  \end{equation} Now set 
\begin{equation}\label{eq:definition of v}
  v = Y- \frac{1}{2} Z^{\tree} -\frac{1}{2} Z^{\ttree}.
\end{equation}
Substituting \eqref{eq:definition of v} into \eqref{eq:resolvent eq for Y}, we observe that \eqref{eq:resolvent eq for Y} is equivalent to 
\begin{equation}\label{eq:resolventEquation for v}
  v = \frac{1}{2} Z^{\tttree} + \cI_\eta( \nabla v \cdot \nabla Z) + R^v,
\end{equation} where $R^v$ denotes
\begin{equation}
  R^v = \frac{1}{8} Z^{\Tree} + \frac{1}{2} \cI_\eta( \nabla(v+ \frac{1}{2}Z^{\ttree}) \cdot \nabla Z^{\tree}) + \frac{1}{2} \cI_\eta ( |\nabla(v+ \frac{1}{2} Z^{\ttree})|^2).
\end{equation} Note that $v$ has a higher regularity $\alpha+1$ than $Y$ by \eqref{eq:resolventEquation for v} and the definition of $\fZ_L$ in Theorem~\ref{thm:renormalization of enhanced noise}. However, this is not sufficient for the well-definedness of $\nabla v \cdot \nabla Z$ in \eqref{eq:resolventEquation for v} since its expected regularity is $\alpha +\alpha -1 <0$. The key idea is to introduce a new object $v^\#$ with a paracontrolled term $v' \prec \cQ $ where $v'$ denotes the pseudo derivative of $v$.

\begin{definition}[\bf Paracontrolled distributions] Let $\alpha \in (\frac{2}{5}, \frac{1}{2})$. Recall the notations of paraproduct from Section~\ref{sec:Paracontrolled}. For $\cQ \in \sC^{\alpha+1}$, we define the space of paracontrolled distributions $\cD^\alpha_{Q} $ as the set of $(v, v') \in \sC^{\alpha+1} \times \sC_{\dR^3}^\alpha$ such that 
\begin{equation}
  v^\# := v - (v' \prec \cQ) \in \sC^{4\alpha}.
\end{equation} We equip $\cD^\alpha_{\cQ}$ with the norm
\begin{equation}
  \| (v,v') \|_{\cD^\alpha_{\cQ}} : = \|v \|_{3\alpha} + \| v' \|_{3\alpha-1} + \|v^\#\|_{\alpha+\beta+1},
\end{equation} where $\beta \in (0,3\alpha-1)$. 
\end{definition}

By Proposition~6.6 and 6.7 of \cite{CC18}, for $\fZ\in \cZ^\varrho$ we have $ (v,v') \in \cD^\alpha_{\cQ}$, $\nabla \cQ \circ \nabla Z\in \sC^{2\alpha-1}$ and hence, $\nabla v \circ \nabla Z$ are well-defined. Note that $v$ solves \eqref{eq:resolventEquation for v} if and only if $v^\# $ solves 
  \begin{align}\label{eq:v_hash}
    v^\# = \cI_\eta \Big( \nabla\big( \frac{1}{2}Z^{\ttree} + v\big) \prec \nabla z\Big) - v' \prec Q+\cI_\eta(\nabla v \circ \nabla Z) + R^v.
  \end{align} 
  \begin{proposition}[Proposition~6.7 of \cite{CC18}]
  Suppose that 
  \begin{equation*}
    v' : = \nabla v + \frac{1}{2}\nabla Z^{\ttree}.
  \end{equation*} Then for $\beta\in(0,3\alpha-1)$ and $\epsilon>0$, we have 
  \begin{equation}
    \| \cI(v' \prec \nabla Z) - v' \prec Q\|_{\alpha+\beta +1} \lesssim \eta^{-\epsilon}  \| v'\|_{3\alpha -1} \| \| Z\|_\varrho
  \end{equation} for the first coordinate of $Z$ of $\fZ \in \cZ^\varrho$ with $5/2< \alpha <\varrho < 1/2 $.
\end{proposition}

Now we present the key proposition to obtain the contractivity of the solution map of \eqref{eq:v_hash} which is a variation of \cite[Proposition~6.8]{CC18}

\begin{proposition}\label{prop:contraction of resolvent equation}
  Let $2/5< \alpha < \varrho < 1/2$ and $\fZ\in \cZ^\varrho$. For $(v,v') \in \cD^\alpha_{\cQ}$, let $\cG : \cD^\alpha_{\cQ} \rightarrow \sC^{\alpha+1} \times \sC^\alpha $ be the map defined by $\cG(v,v') = (\tilde{v} , \tilde{v}')$ where 
  \begin{equation}
    \tilde{v} : = \frac{1}{2} Z^{\tttree} + \cI_\eta( \nabla v \cdot \nabla Z) + R^v, \qquad \tilde{v}' : =\nabla v + \frac{1}{2}\nabla Z^{\ttree}.
  \end{equation} Then $\cG(v,v')\in \mathcal{D}^{\alpha}_{Q}$ and there exists $\vartheta>0$ such that 
  \begin{equation}\label{eq:cGBound}
    \| \cG( v,v') \|_{\cD^\alpha_{\cQ}} \lesssim (1+ \| \fZ \|_{\cZ^\varrho})^2(1+ \eta^{-\vartheta}\| (v,v')\|_{\cD^\alpha_{\cQ}})^2,
  \end{equation} and for $(v_1,v_1'),(v_2,v_2') \in \cD^\alpha_{\cQ}$, 
  \begin{equation}\label{eq:cGdistanceBound}
  \begin{aligned}
      d_{\cD^\alpha_{\cQ}}&( \cG(v_1,v_1'),\cG(v_2,v_2')) \\ 
      &\lesssim \eta^{-\vartheta}  d_{\cD^\alpha_{\cQ}}((v_1,v_1'), (v_2,v_2')) (1+ \| (v_1,v_1') \|_{\cD^\alpha_{\cQ}} +   \| (v_2,v_2') \|_{\cD^\alpha_{\cQ}}  ) ( 1+ \| \fZ\|_{\cZ^\varrho})^2,
      \end{aligned}
  \end{equation} where $d_{\cD^\alpha_{\cQ}} ((v_1,v_1') , (v_2,v_2')) : = \| (v_1-v_2, v_1'-v_2') \|_{\cD^\alpha_{\cQ}}$.
  
\end{proposition}

\begin{proof}
  The proof is very similar to the proof of \cite[Proposition~6.8]{CC18} where the transition density operator $\cI$ is used instead of $\cI_\eta$. Since $(\eta- \frac{1}{2}\Delta )^{-1} = \int_0^\infty e^{-\eta t } P_t dt$, the estimations in \cite[Proposition~6.8]{CC18} can be applied to our setting with replacing  $T^{\vartheta}$ \cite[Proposition~6.8]{CC18} into $\eta^{-\vartheta}$. We omit the detail.
\end{proof}

We are now in the position to prove Proposition~\ref{prop:resolvent eqaution}. 

\begin{proof}[Proof of Proposition~\ref{prop:resolvent eqaution}]
  Define $M_L : = \| \fZ_L \|_{\cZ^\varrho}$ and take 
  \begin{equation}\label{eq:definition of eta_L}
\eta_L :=  A(1+M_L)^{\aleph},
  \end{equation}
 where $\aleph := \frac{3}{2\vartheta}$ and $A>0$. Here $\vartheta$ is the same constant which appear in Proposition~\ref{prop:contraction of resolvent equation}. By taking $A$ sufficiently large, for $(v,v') \in \cD^{\alpha}_Q$ with $\| (v,v')\|_{\cD^{\alpha}_Q} \leq M_L$, we have 
\begin{equation*}
  \| \cG(v,v')\|_{\cD^\alpha_{\cQ}} \leq  \text{const.} \cdot A^{-2\vartheta} M_L \leq M_L
\end{equation*} using Proposition~\ref{prop:contraction of resolvent equation}. Using a similar argument with \eqref{eq:cGdistanceBound}, we can conclude that for all $\eta \geq \eta_L$, $\cG$ is a contraction mapping in $\cD^\alpha_{\cQ}$, hence there exists a unique solution $v\in \sC^{\alpha+1}$ (with a proper $v'\in \sC^{\alpha}) $ to \eqref{eq:resolventEquation for v}. Since we have set before $Y = v + \frac{1}{2} ( Z^{\tree} + Z^{\ttree})$, we get the unique solution $Y\in \sC^{2\alpha}$ to \eqref{eq:resolvent eq for Y}.

\end{proof}

\begin{proof}[Proof of Theorem~\ref{thm:Feynman-Kac representation}] 
By Proposition~\ref{prop:resolvent eqaution}, we know that for $\epsilon\in [0,1]$, $L>0,$ there exists a unique solution $Y_{L,\epsilon}$ to 

\begin{equation}\label{eq:resolvent equation for epsilon, L}
  (\eta_{L} - \frac{1}{2}\Delta )Y_{L,\epsilon} = \frac{1}{2}|\nabla Z_{L,\epsilon}|^2 -c_\epsilon + \nabla Y_{L,\epsilon}\cdot \nabla Z_{L,\epsilon} + \frac{1}{2} |\nabla Y_{L,\epsilon}|^2,
\end{equation} where we define $Z_{L,0} :=Z_L =  \lim_{\epsilon\rightarrow 0 } Z_{L,\epsilon}$ and $Y_{L,0} := Y_L =  \lim_{\epsilon\rightarrow 0 } Y_{L,\epsilon}$. For simplicity, we write $Z$ for $Z_{L,\epsilon}$ and $Y, \xi, \eta,c$ similarly. Since $Z- \frac{1}{2}\Delta Z = \xi$, we have 
\begin{equation*}
   \xi - c = Z + \eta Y - \frac{1}{2}\Delta (Z+Y) -\frac{1}{2}|\nabla (Z+Y)|^2, 
\end{equation*} or equivalently, 
\begin{equation}\label{eq:equation for Delta(Z+Y)}
  \frac{1}{2}\Delta(Z+Y) = -(\xi -c ) + Z+\eta Y - \frac{1}{2}|\nabla (Z+Y)|^2.
\end{equation} On the other hand, by Ito's formula, we obtain
\begin{equation}\label{eq:Ito formula for Z+Y}
  (Z+Y)(B_t)= (Z+Y)(B_0) + \frac{1}{2}\int_0^t \Delta(Z+Y)(B_s)ds + \int_0^t \nabla(Z+Y)(B_s)ds,
\end{equation} where $B$ is a Brownian motion. By substituting \eqref{eq:equation for Delta(Z+Y)} into \eqref{eq:Ito formula for Z+Y}, we have 
\begin{equation}\label{eq:rearrange}
  \exp\left(\int_0^t [\xi(B_s) -c ] ds \right) = N_t \cdot \exp\left( \int_0^t (Z+\eta Y) (B_s)ds + (Z+Y)(B_0) - (Z+Y)(B_t) \right),
\end{equation} where 
\begin{equation}
  N_t : = \exp \left( \int_0^t \nabla(Z+Y)(B_s)\cdot dB_s - \frac{1}{2} \int_0^t |\nabla(Z+Y)(B_s)|^2 ds \right).
\end{equation}
Notice that the right hand side of \eqref{eq:rearrange} is equal to $N_t\cdot \sD_{L,y}(r,t)$. Applying Girsanov's theorem shows that the right hand side of \eqref{eq:rearrange} is equal to  $\E_{\dQ^x_{L,\epsilon}} \big[\sD_{L,0}(0,t) \phi(X_t) \1_{X_{[0,t]} \subset Q_L}  \big]$
Therefore, \eqref{eq:feynman-Kac representation} holds when $\epsilon\in(0,1]$. 

In order to prove \eqref{eq:convergence of Feynman-Kac}, we first show that $\dQ^{x}_{L,\epsilon}$ converges weakly to $\dQ^{x}_{L}$. To this end, note that by Lemma~\ref{lem:lifting lemma}, we can apply Theorem~\ref{thm:well-posedenss of martingale problem} to the martingale problem associated with \eqref{eq:diffusion_X}. By Theorem~\ref{thm:well-posedenss of martingale problem} and Theorem~\ref{def:enhanced noise}, we obtain the weak convergence of $\dQ^{x}_{L,\epsilon}$ by the fact that  $\nabla(Z_{L,\epsilon}+Y_{L,\epsilon}) \rightarrow\nabla(Z_{L}+Y_{L}) $ in $L_\infty$ as $\epsilon \rightarrow \infty$. Now the proof of \eqref{eq:convergence of Feynman-Kac} follows from that the set $\{ X[0,t] \subset Q_L\}$ is a $\dQ^x_L-$continuity set, which was proven in the proof of \cite[Theorem~2.17]{KPZ20}. This completes the proof.

\end{proof}

\section{Existence of transition kernel \& its estimate}\label{sec:transition density estimate}

The goal of this section is two-fold. We first show that the existence of transition kernel for the diffusion $X_t$ of \eqref{eq:diffusion_X} of Theorem~\ref{thm:Feynman-Kac representation}. Secondly we derive bounds on the transition kernel. This bound will later be used to derive suitable bound on the escape probability such as the probability of the event $\{X_{[0,t]} \not\subset Q_L \}$.

\subsection{Existence of the transition kernel}\label{subsec:existence of transition density}
We fix $T>0$. Recall that $X$ is the solution to the SDE with a distributional drift: For any $x\in \dR^d$

\begin{equation}\label{eq:SDE with distributional distribution}
  X_t = x + \int_{0}^t \mu  (X_s) ds + B_t, \quad t\in [0,T]
 \end{equation} where  $\mu : = \nabla (Z+Y)$ and $Z,Y$ are defined in Section~\ref{sec:feynman-kac representation} (see Section~\ref{prop:resolvent eqaution}). Throughout this section, we assume that $Z\in \sC^{\alpha}, Y\in \sC^{2\alpha} $ where $\alpha<1/2$ as in the case $d=3$. Our goal is to show that $X$ has a transition density $\Gamma_t : \dR^d \times \dR^d \rightarrow \dR$ for all $t>0$. The following is the main result of this section. 

\begin{theorem}\label{thm:existence of the transition kernel}
  Suppose $\mu  = \nabla(Z+Y)$ in \eqref{eq:SDE with distributional distribution}. Then the solution $X$ of the SDE~\eqref{eq:SDE with distributional distribution} admits the transition density $\Gamma_t(x,y) = w^{\delta_y,\mu}(t,x)  $  for $(t,x,y) \in (0,T) \times \dR^{2d}$ where $w^{\delta_y,\mu}$ is the solution to the Cauchy problem \eqref{eq:cauchy problem} associated with $\mu$ and the terminal condition $\delta_y$.
\end{theorem}

\begin{remark}
It is worthwhile to mention that $Y$ defined via \eqref{eq:original resolvent eq for Y} in $d=2$ case belongs to $ \sC^{\frac{3}{2}-}$. In this occasion, \cite[Proposition~2.9]{PZ22} has proven the existence of the heat kernel. However the results in \cite{PZ22} can only be applied when $Y \in \sC^{2\alpha}$ for some $\alpha >1/2$ which is not true for $d=3$. Therefore, Theorem~\ref{thm:existence of the transition kernel} is quintessential for $d=3$.
\end{remark}

In fact, we prove a more general result than Theorem~\ref{thm:existence of the transition kernel}.  We show that any solution of \eqref{eq:SDE with distributional distribution} with a suitable $\mu\in \sC^\beta $ with $\beta \in(-\frac{2}{3}, -\frac{1}{2}]$ admits a transition density. Later we will show that the suitable class for $\mu$ is the space of {\it rough distributions} (see Definition~\ref{def:rough distribution}) and $\mu = \nabla(Z+Y)$ can be lifted to a rough distribution $\bmu$.


In order to study the SDE~\eqref{eq:SDE with distributional distribution}, we consider the martingale problem associated with the following Cauchy problem on $[0,T] \times \dR^d$: Define the differetial operator 
\begin{equation*}
  \cL w: =\frac{1}{2}\Delta + \mu \cdot \nabla
  \end{equation*} and the differential equation
\begin{equation}\label{eq:cauchy problem}
  \begin{cases}
    \partial_t w(t,x) + \cL w(t,x) =0, \quad (t,x) \in [0,T) \times \dR^d,\\
    w(T,x) = \phi, \quad x \in \dR^d
  \end{cases}
\end{equation} where $\phi:\dR^d \rightarrow \dR$ is the terminal conditions. We denote the solution to \eqref{eq:cauchy problem} with the terminal function $\phi$ by $w^\phi$. We recall the definition of the martingale problem associated to \eqref{eq:cauchy problem}. 

\begin{definition}\label{def:martingale problem}
  Let us denote $\Omega = C([0,T], \dR^d)$. A stochastic process $\boldsymbol{X} = \{X_t\}_{t\in [0,T]}$ on the probability space $(\Omega, \mathcal{B}(\Omega), \dP)$ endowed with the canonical filtration $\{\mathfrak{F}_t\}_{t\geq 0}$ is said to be a solution to the SDE \eqref{eq:SDE with distributional distribution} if $X_0 =x$ and it satisfies the martingale problem for $( \cL, x) $: for all $ f\in C([0,T], L^\infty(\dR^d))$ and all $\phi\in C^\infty_c(\dR^d)$, $w = w^\phi $ is the solution to the Cauchy problem~\ref{eq:cauchy problem} and 
  \begin{equation*}
     \left \{ w(t,X_t) - \int_0^t f( s,X_s ) ds \right\}_{t\in [0,T]} 
  \end{equation*}is a martingale.
\end{definition}

Before we mention the well-posedness of the martingale problem, we introduce the definition of the {\it rough distribution} (\cite[Definition~3.6]{CC18}) of a suitable class of the drift $\mu$ for the well-posedness.

\begin{definition}[\bf Rough distribution]\label{def:rough distribution}
  Let $\beta\in (-\frac{2}{3}, -\frac{1}{2})$, $\gamma < \beta +2$. Set $\cH^\gamma := \sC_{\dR^d}^{\gamma-2} \times \sC_{\dR^d}^{2\gamma-3} $. We define the space of {\it rough distributions} as 
  \begin{equation*}
     \sX^\gamma  = \textrm{cl}_{\cH^{\gamma}} \left \{ \cK( \theta) := ( \theta ,\cI ( \partial_j\theta^i ) \circ \theta^j )_{i,j=1,...,d},  \theta \in \sC^\infty_{\dR^d}\right\}.
   \end{equation*} We denote by $\boldsymbol{\mu} = (\mu, \mu')$ a generic element of $\sX^\gamma$ and we say that $\boldsymbol{\mu}$ is a lift (or enhancement) of $\mu.$
\end{definition}

We now present the well-posedness of the martingale problem associated to \eqref{eq:cauchy problem} proven in \cite[Theorem~1.2]{CC18}.

\begin{theorem}[Theorem~1.2 of \cite{CC18}]\label{thm:well-posedenss of martingale problem}
  Let $\beta\in (-\frac{2}{3}, -\frac{1}{2}]$ and $\mu \in \sC^{\beta}_{\dR^d}$. We assume that $\mu$ can be enhanced to a rough distribution $\bmu\in \sX^\gamma$ for some $\gamma<\beta + 2$. Then there exists a (stochastically) unique solution to the martingale problem for $(\cL, x)$ in the sense that there is a unique probability measure $\dP_x$ on $\Omega = C([0,T], \dR^d) $ such that the coordinate process $X_t(\omega) = \omega(t)$ satisfies the martingale problem for $(\cL, x)$. Moreover, $X$ is a strong Markov process under $\dP_x$
 and the $\dP_x$ depends (weakly) continuously on the drift $\mu$. 
 \end{theorem}
\begin{remark}
In \cite[Theorem~1.2]{CC18}, the continuity of $\dP_x$ in the drift $\mu$ is not given. However, the proof of \cite[Theorem~1.2]{CC18} directly implies this fact. See the proof of \cite[Theorem~4.3]{CC18}.
\end{remark}

As in \cite[Section~2]{PZ22}, we will show that there exists a solution $w^{\delta_y}$ to the Cauchy problem~\eqref{eq:cauchy problem} with the terminal condition $\delta_y$. We will then show that the transition density $\Gamma$ of $X$ is defined as $\Gamma_{t}(x,y) = w^{\delta_y}(t,x)$. In order to deal with \eqref{eq:cauchy problem}, we employ the paracontrolled distributions and extend \cite[Theorem~3.10]{CC18} for the delta initial data.
Recall that for any $(t,x)\in [0,T]\times \bR^d$ and function $\psi$, $\mathcal{J}(\psi)(t)$ is defined as $\int^{T}_{t}  P_{r-t} \psi(r) dr$ where $P_{t}$
is the usual heat flow, i.e., $P_t = e^{\frac{1}{2}t\Delta}$.

\begin{definition}
  Let $T>0$, $\frac{4}{3}< \alpha < \theta < \beta +2$ and $\rho> \frac{\theta-1}{2}$. For $\bar{T}\in (0,T)$, $p\in[1,\infty]$ and $\delta>0$, we define the space of {\it paracontrolled distributions} $\sD^{\alpha,\theta,\delta,p}_{\rho,\bar{T},T}$ as the set of pairs of distributions $(w,w')\in C_{\bar{T},T}\sC^{\theta}_p\times C_{\bar{T},T}\sC^{\alpha-1}_p$ (see \eqref{eq:definition of the norm for Holder-Besov valued process}) such that 
  \begin{equation}\label{eq:definition of w sharp}
     w^{\#}(t) := w(t) - w'(t)\prec \cJ(\mu)(t) \in \sC^{2\alpha-1}_p
   \end{equation}  for all $t \in(T-\bar{T},T] $. Here '$\prec$' denotes the paraproduct which is defined in Section~\ref{sec:Paracontrolled}. We equip $\sD^{\alpha,\theta,\delta,p}_{\rho,\bar{T},T}$ with the norm
\begin{equation}\label{eq:norm_D}
  \| (w,w') \|_{\sD^{\alpha,\theta,\delta,p}_{\rho,\bar{T},T}} : = \| w \|_{C^\delta_{ \bar{T},T}\sC_p^\theta}+ \|\nabla w \|_{C^\delta_{\rho,\bar{T},T}L^\infty } + \| w' \|_{C^\delta_{\bar{T},T}\sC_p^{\alpha-1}} + \|w^\#\|_{C^\delta_{\bar{T},T}\sC_p^{2\alpha-1}}. 
\end{equation} and the metric $d_{\sD^{\alpha,\theta,\delta,p}_{\rho,\bar{T},T}}$ defined as 
\begin{equation}
  d_{\sD^{\alpha,\theta,\delta,p}_{\rho,\bar{T},T}} ( (w,w'), (\tilde{w},\tilde{w}')) : = \| (w-\tilde{w}, w'-\tilde{w}') \|_{\sD^{\alpha,\theta,\delta,p}_{\rho,\bar{T},T}}. 
\end{equation} Equipped with this metric, the space $(\sD^{\alpha,\theta,\delta,p}_{\rho,\bar{T},T},d_{\sD^{\alpha,\theta,\delta,p}_{\rho,\bar{T},T}})$ is complete metric space, thus it is closed.

\end{definition}

 We now introduce a solution map for the construction of a fixed point problem for \eqref{eq:cauchy problem}.

\begin{definition}
  Let $\frac{4}{3}< \alpha< \theta<\gamma <\beta+2$, $p\in[1,\infty]$, $\rho\in ( \frac{\theta-1}{2}, \frac{\gamma-1}{2})$, $T>0$ and $\bmu\in \sX^\theta$ be an enhancement of $\mu$. For $\bar{T}\in(0,T)$, define the map $M_{\bar{T}} : \sD^{\alpha,\theta,\delta,p}_{\rho,\bar{T},T} \rightarrow C_{\bar{T},T}\sC^\alpha_p$ by 
  \begin{equation}
    M_{\bar{T}}(w,w')(t) : = P_{T-t}\phi + \cJ(\nabla w\cdot \mu)(t),
  \end{equation}  for $(u,u') \in \sD^{\alpha,\theta,\delta,p}_{\rho,\bar{T},T}$ and $\phi\in\sC^\gamma_p$. We also define the map 
  \begin{equation}\label{eq:fixed point map for cM}
    \begin{aligned}
      \cM_{\bar{T}} : \sD^{\alpha,\theta,\delta,p}_{\rho,\bar{T},T} 
      &\rightarrow C_{\bar{T},T}\sC^{\alpha}_p \times C_{\bar{T},T}\sC^{\alpha-1}_p \\
      (w,w') &\mapsto (M_{\bar{T}}(w,w'), \nabla w).
    \end{aligned}
  \end{equation}
\end{definition}

The following proposition is a generalization of \cite[Proposition~3.9]{CC18} in the sense that we have used different norms for $(w,w')$ in \eqref{eq:norm_D} to allow a blowup at time $T$, which depends on $\delta>0$ (see \eqref{eq:definition of the norm for Holder-Besov valued process}). We will then present a key estimate of the solution map for the rough terminal data which lies in $\sC^{-\epsilon}_p$ for some $\epsilon>0$.

\begin{proposition}\label{prop:contraction mapping cM}
  Let $0<T<1$, $\frac{4}{3}< \alpha< \theta<\gamma <\beta+2$, $\rho\in ( \frac{\theta-1}{2}, \frac{\gamma-1}{2})$ and $\delta>2\alpha-1$. Let the terminal data $\phi\in \sC^{\gamma}_p$, $\mu \in \sC^\beta $ and $\bmu \in \sX^\gamma$ be an enhancement of $\mu$. Then, there exists $\kappa>0$ which depends only on $\alpha,\theta,\rho,\gamma,p$ such that for any for $(w,w'),(\tilde{w},\tilde{w}')\in\sD^{\alpha,\theta,\delta,p}_{\rho,\bar{T},T} $ and $\bar{T}\in (0,T)$,
\begin{equation}\label{eq:barTnorm}
  \begin{aligned}
    \| \cM_{\bar{T}}(w,w') \|_{\sD^{\alpha,\theta,\delta,p}_{\rho,\bar{T},T} } \lesssim (1+ \| \bmu \|_{\sX^\gamma})^2( 1+ \bar{T}^\kappa \| (w,w')) \|_{\sD^{\alpha,\theta,\delta,p}_{\rho,\bar{T},T}}) + \| \phi \|_{\sC^{\gamma}_p},
  \end{aligned} 
\end{equation} and 
\begin{equation}\label{eq:barTdistance}
  \begin{aligned}
    d_{\sD^{\alpha,\theta,\delta,p}_{\rho,\bar{T},T}} ( \cM_{\bar{T}}(w,w'), \cM_{\bar{T}}(v,v')) \lesssim \bar{T}^\kappa(1+ \| \bmu \|_{\sX^\gamma})^2 d_{\sD^{\alpha,\theta,\delta,p}_{\rho,\bar{T},T}} ( (w,w'), (\tilde{w},\tilde{w}')).
  \end{aligned}
\end{equation}
\end{proposition}

\begin{proof}
The proof follows similar line of ideas as in \cite[Proposition~3.9]{CC18}. The difference is that we use two more parameters $p\in[1,\infty]$ and $\delta>2\alpha-1$ than in \cite[Proposition~3.9]{CC18}. However, with minor modifications we can get \eqref{eq:barTnorm} and \eqref{eq:barTdistance}. Indeed, since Bony's estimate (Proposition~\ref{prop:Bony's estimates (II)}) holds for all $p\in[1,\infty]$, we can bound the $\|\cdot \|_{\sC^a_p}$-norm as in the same way with the $\|\cdot \|_{\sC^a}$-norm for any $a\in\dR.$ Moreover, by multiplying $(T-t)^\delta$ with $\delta>2\alpha-1$ on the $\|\cdot \|_{C_{\bar{T},T}\sC^a_p}$-norm, we can proceed every estimation similarly as in the proof of \cite[Proposition~3.9]{CC18}.

\end{proof}

The next lemma bounds the distance between two fixed points of the map $\mathcal{M}_{\bar{T}}$ starting from two distinct initial data. 

\begin{lemma}\label{lem:contraction mapping for fixed points of CM}
Let $0<T<1$, $\frac{4}{3}< \alpha< \theta<\gamma <\beta+2$, $\rho\in ( \frac{\theta-1}{2}, \frac{\gamma-1}{2})$, $\epsilon>0$ and $\delta >\frac{2\alpha -1 + \epsilon}{2}$. Let $\mu \in \sC^\beta $ and $\bmu \in \sX^\gamma$ be an enhancement of $\mu$. Then, there exists $\kappa>0$ which depends only on $\alpha,\theta,\rho,\gamma,p$ such that for any $\bar{T}\in (0,T)$ and $(w,w'), (\tilde{w},\tilde{w}') \in \sD^{\alpha,\theta,\delta,p}_{\rho,\bar{T},T}$ being fixed points of the map $\cM_{\bar{T}}$ with the terminal data $\phi,\tilde{\phi} \in \sC^\gamma_p$ respectively, we have 
\begin{equation}\label{eq:barTdistance for fixed points}
   \begin{aligned}
    d_{\sD^{\alpha,\theta,\delta,p}_{\rho,\bar{T},T}}( (w,w'), (\tilde{w},\tilde{w}')) \lesssim \bar{T}^\kappa(1+ \| \bmu \|_{\sX^\gamma})^2 d_{\sD^{\alpha,\theta,\delta,p}_{\rho,\bar{T},T}} ( (w,w'), (\tilde{w},\tilde{w}')) + \| \phi -\tilde{\phi} \|_{\sC^{-\epsilon}_p}.
  \end{aligned}
\end{equation}
\end{lemma}
\begin{proof}
Note that since $(w,w')$ is a fixed point of $\cM_{\bar{T}}$ associated with a terminal function $\phi$, we know that $M_{\bar{T}}(w,w') = w$ and $M_{\bar{T}}(w,w')' = \nabla w $. Moreover, 
\begin{equation}\label{eq:fixed point w}
   w(t) = P_{T-t}\phi + \cJ(\nabla w \cdot \mu)(t) 
\end{equation} holds for all $t \in (T-\bar{T}, T]$ and the same facts hold for $\tilde{w}$. Since \eqref{eq:fixed point w} is linear, we prove the lemma for $w$ for simplicity. By Proposition~\ref{prop:contraction mapping cM}, we have the first term on the r.h.s in \eqref{eq:barTdistance for fixed points}. For the second term, it suffices to obtain the upper bound \eqref{eq:barTnorm} for $w$ with $\|\phi\|_{\sC^{-\epsilon}_p}$ instead of $\|\phi  \|_{\sC^{\gamma}_p}$. Using \eqref{eq:fixed point w}, we need to bound 
\begin{equation}\label{eq:sD-norm of fixed point w}
    \|( w,w') \|_{\sD^{\alpha,\theta,\delta,p}_{\rho,\bar{T},T}} =  \|w \|_{C^\delta_{ \bar{T},T}\sC_p^\theta}+ \|\nabla w \|_{C^\delta_{\rho,\bar{T},T}L^\infty } 
      + \| \nabla w \|_{C^\delta_{ \bar{T},T}\sC_p^{\alpha-1}} + \|w^\#\|_{C^\delta_{ \bar{T},T}\sC_p^{2\alpha-1}} 
\end{equation} to extract the bound by $\|\phi\|_{\sC^{-\epsilon}_p}$. The first term on the r.h.s. of the above display, i.e., $ \|w \|_{C^\delta_{ \bar{T},T}\sC_p^\theta}$ is the upper bound by $\|\phi\|_{\sC^{-\epsilon}_p}$ since $\delta >\frac{2\alpha-1+\epsilon}{2} > \frac{\theta +\epsilon}{2}$ and 
\begin{equation}\label{eq:condition for delta 1}
     \| P_{T-\cdot} \phi \|_{C^\delta_{ \bar{T},T}\sC_p^\theta}  = \sup_{t\in (T-\bar{T},T]}(T- t)^{\delta } \| P_{T-t} \phi \|_{\sC_p^\theta}  \lesssim \| \phi\|_{\sC_p^{\theta-2\delta} } \leq \| \phi\|_{\sC^{-\epsilon}_p},
\end{equation} where we used Lemma~\ref{lem:schauder estimate} in the first inequality. For $\|\nabla w \|_{C^\delta_{\rho,\bar{T},T}L^\infty } $, we can see that for $t<s\in(T-\bar{T},T]$ 
\begin{equation*}
    \begin{aligned}
      \| \nabla P_{T-t}\phi -  \nabla P_{T-s}\phi\|_{L^\infty } &= \|(P_{s-t} - \textrm{Id})P_{T-s}\nabla \phi\|_{L^\infty}
   \\&\lesssim |t-s|^\rho \| P_{T-s} \nabla \phi\|_{\sC_p^{2\rho}}
   \\&\lesssim |t-s|^\rho  (T-s)^{-\delta }\| \phi\|_{\sC_p^{2\rho+1-2\delta}}.
    \end{aligned}
\end{equation*} This implies that 
\begin{equation}\label{eq:condition for delta 2}
    \| \nabla P_{T-\cdot}\phi\|_{C^\delta_{\rho,\bar{T},T}L^\infty }  \lesssim \| \phi\|_{\sC_p^{2\rho+1-2\delta}} \leq \| \phi\|_{\sC_p^{-\epsilon}},
    \end{equation} since $2\alpha-1 > 2\rho +1$ by the conditions on the parameters. For $\| \nabla w \|_{C^\delta_{ \bar{T},T}\sC_p^{\alpha-1}}$, it is enough to see that $\| \nabla w \|_{C^\delta_{ \bar{T},T}\sC_p^{\alpha-1}} \lesssim \|  w \|_{C^\delta_{ \bar{T},T}\sC_p^{\alpha}}$. To bound $\|w^\#\|_{C^\delta_{ \bar{T},T}\sC_p^{2\alpha-1}} $, note that 
    \begin{equation}
        \begin{aligned}
            w^\#(t) = w(t) - (\nabla w \prec \cJ(\mu))(t)
        \end{aligned}
    \end{equation} by \eqref{eq:definition of w sharp}. The only contribution of $\phi$ appears in the first term on the r.h.s. We use Lemma~\ref{lem:schauder estimate} to get 
    \begin{equation}\label{eq:condition for delta 3}
        (T-t)^\delta \| P_{T-t}\phi \|_{\sC_p^{2\alpha-1}} \lesssim  \| \phi\|_{\sC_p^{2\alpha-1-2\delta}}\lesssim \| \phi\|_{\sC_p^{-\epsilon}}.
    \end{equation} Putting all together and replacing $w$ by $w-\tilde{w}$, we have
    \begin{equation*}
        \| (w-\tilde{w},w' - \tilde{w}') \|_{\sD^{\alpha,\theta,\delta,p}_{\rho,\bar{T},T}} \lesssim \bar{T}^\kappa  (1+ \| \bmu \|_{\sX^\gamma})^2\| (w-\tilde{w},w' - \tilde{w}') \|_{\sD^{\alpha,\theta,\delta,p}_{\rho,\bar{T},T}}+ \| \phi -\tilde{\phi} \|_{\sC^{-\epsilon}_p},
    \end{equation*} which completes the proof.
    
\end{proof} 
 
Now we are ready to present one of the main results of this section which extends \cite[Theorem~3.10]{CC18} for a more larger class of terminal data including Dirac delta terminal data. We write $w^\phi_\mu$ for the solution $w$ of \eqref{eq:cauchy problem} with the terminal function $\phi$ and the drift $\mu$.

\begin{theorem}\label{thm:extension to the delta initial data}
  Let $p\in[1,\infty], \beta \in (-\frac{2}{3}, -\frac{1}{2}]$, $\frac{4}{3} < \theta < \gamma < \beta +2$ and $\epsilon, T>0.$ For $\phi \in  \sC^{-\epsilon}_p$ and $\mu \in\sC^{\beta }$ that can be enhanced to a lift $\bmu\in \sX^{\gamma}$ , the Cauchy problem~\eqref{eq:cauchy problem} has a unique mild solution $w_\mu^\phi$ in $C([0,T], \sC^{-2\epsilon \wedge \beta })$ such that $w_\mu^\phi(t) \in \sC^\alpha $ for all $t\in[0,T)$. Moreover, for all $t>0$ the map $\sC^{-\epsilon}_p \times \sC^{\beta} \rightarrow \sC^\alpha $ given by $(\phi , \mu) \mapsto w^{\phi}_{\mu}(t,\cdot)$ is locally Lipshitz.
\end{theorem}

\begin{proof}

 Fix $\phi  \in \sC^{-\epsilon}_p$. Suppose  $\{ \phi_n\}_{n\in \dN} $ be a sequence of functions in $\sC^{\gamma}_p$ such that 
\begin{equation}\label{eq:sequence converging to initial data}
  \|\phi_n - \phi\|_{\sC^{-\epsilon}_p} \rightarrow 0
\end{equation} as $n\rightarrow \infty$. Then by Proposition~\ref{lem:contraction mapping for fixed points of CM}, for each $n,$ and $\bar{T}\in (0, T)$, we have a unique solution $(w_n, w_n')\in \sD^{\alpha,\theta,\delta,p}_{\rho,\bar{T},T}$ of the fixed point problem of \eqref{eq:fixed point map for cM}. Since $(w_n,w_n')$ are fixed point of the map $\mathcal{M}_T$,  we have $\mathcal{M}_T(u_n, u_n') = (u_n, u_n')$ for all $n$ and use Lemma~\ref{lem:contraction mapping for fixed points of CM} to obtain 
\begin{equation}
  \|(w_n, w_n') - (w_m,w_m')\|_{\sD^{\alpha,\theta,\delta,p}_{\rho,\bar{T},T}} \lesssim \bar{T}^\kappa (1+ \| \bmu \|_{\sX^\gamma})^2  \|(w_n,w_n') - (w_m,w_m')\|_{\sD^{\alpha,\theta,\delta,p}_{\rho,\bar{T},T}} + \| \phi_n-\phi_m\|_{\sC^{-\epsilon}_p}.
\end{equation} 
We choose $\bar{T}>0$ (not depending on $n$) small such that we may rewrite the above inequality as 
$$\|(w_n,w_n') - (w_m,w_m')\|_{\sD^{\alpha,\theta,\delta,p}_{\rho,\bar{T},T}} \lesssim \| \phi_n-\phi_m\|_{\sC^{-\epsilon}_p}.$$
The above display implies that 
$ \{(w_n,w_n')\}_{n\in \dN}$ is a Cauchy sequence in $\sD^{\alpha,\theta,\delta,p}_{\rho,\bar{T},T}$ due to \eqref{eq:sequence converging to initial data}. Since $\sD^{\alpha,\theta,\delta,p}_{\rho,\bar{T},T}$ is closed,  there is a unique limit point $(w,w')\in \sD^{\alpha,\theta,\delta,p}_{\rho,\bar{T},T}$ of $ \{(w_n,w_n'))\}_{n\in \dN}$. Since $(w,w')$ is a limit of the set of fixed points under the map $\mathcal{M}_{\bar{T}}$, $(w,w')$ is also a fixed point under of the same map. As a result, we get  
\begin{equation}
  \begin{aligned}
    w = M_{\bar{T}}(w,w') = \cJ(\nabla w \cdot \mu) + P_{T-\cdot}\phi
  \end{aligned}
\end{equation} in $C^\delta_{\bar{T},T}\sC^{\theta}_p$.  
We now verify that $w$ solves this equation started from the initial data $\phi\in \sC^{-\epsilon}_p$. This follows from the following estimates:
\begin{equation}
  \| P_{T-t}\phi - \phi \|_{\sC^{-2\epsilon}_p}\lesssim (T-t)^{\epsilon}\|u_0 \|_{\sC^{-\epsilon}_p},
\end{equation} and 
\begin{equation}
  \| \cJ(\nabla u \cdot \mu)(t)\|_{\sC^{-2\epsilon}_p}\lesssim  (T-t)^{\vartheta} \sup_{s\in(T-\bar{T},T]} (T-s)^{\gamma} \| \nabla w \cdot \mu \|_{\sC^{\alpha+\beta-1}}, \quad \vartheta := \frac{\alpha+\beta-1}{2}+\epsilon -\gamma +1, 
\end{equation}
for some $\gamma \in (0,1)$. 
We obtain the first inequality by applying the second Schauder's estimates from Lemma~\ref{lem:schauder estimate}. The second inequality is obtained by applying part (1) of Corollary 2.5 from \cite{CC18}. Recall that $\alpha>\frac{4}{3}$, $\beta \in (-\frac{2}{3}, -\frac{1}{2}]$. We choose $\gamma$ to be very small such that $\vartheta>0$. Therefore the right hand side of the inequalities of the above display tends to $0$ as $t$ goes to $\infty$. 
This shows $w\in C([T-\bar{T},T],\sC^{-2\epsilon}_p) \cap C^\delta_{\bar{T},T}\sC^{\theta}_p  $ is the solution to the Cauchy problem
\begin{equation}
   w(t) = \cJ(\nabla w \cdot \mu)(t) + P_{T-t} u_0, \quad t\in[0,T)
\end{equation} and $w(T,\cdot)=\phi$, or equivalently to \eqref{eq:cauchy problem}.

 Now we claim that for $t>0$ the map $\sC^{-\epsilon}_p \times \sX^\gamma \rightarrow \sD^{\alpha,\theta,\delta,p}_{\rho,\bar{T},T}$ given by $(\phi, \bmu) \mapsto (u^{\phi}_{\bmu_1}(t,\cdot), \nabla u^{\phi}_{\bmu_2}(t,\cdot))$ is locally Lipschitz. If we let $\phi_1,\phi_2\in \sC^{-\epsilon}_p$, $\bmu_1,\bmu_2 \in \sC^{\beta}$, $w_i= w^{\phi_i}_{\bmu_i}$ and $w'_i= \nabla w^{\phi_i}_{\bmu_i}$ be the solution for each $i=1,2$, then we have 
\begin{equation}\label{eq:Lipschitz}
   \begin{aligned}
     \|(w_1,w'_1) - (w_2,w'_2)\|_{\sD^{\alpha,\theta,\delta,p}_{\rho,\bar{T},T}} & \leq  \|\mathcal{M}^{\bmu_1}_{\bar{T}}(w_1,w'_1) - \mathcal{M}^{\boldsymbol{\mu}_1}_{\bar{T}}(w_2,w'_2)\|_{\sD^{\alpha,\theta,\delta,p}_{\rho,\bar{T},T}} +  \|\mathcal{M}^{\boldsymbol{\mu}_1}_{\bar{T}}(w_2,w'_2)- \mathcal{M}^{\boldsymbol{\mu}_2}_{\bar{T}}(w_2,w'_2)\|_{\sD^{\alpha,\theta,\delta,p}_{\rho,\bar{T},T}}
   \end{aligned}
 \end{equation}
 Due to \eqref{eq:barTnorm} of Proposition~\ref{prop:contraction mapping cM}, we have 
 \begin{align}\label{eq:1stInequality}
 \|\mathcal{M}^{\bmu_1}_{\bar{T}}(w_1,w'_1) - \mathcal{M}^{\boldsymbol{\mu}_1}_{\bar{T}}(w_2,w'_2)\|_{\sD^{\alpha,\theta,\delta,p}_{\rho,\bar{T},T}} 
 \leq  \|\phi_1 - \phi_2\|_{\sC^{-\epsilon}_p} + \bar{T}^{\kappa} (1+\| \bmu_1 \|_{\sX^{\gamma}})^2\| (w_1, w'_1) - (w_2, w'_2)\|_{\sD^{\alpha,\theta,\delta,p}_{\rho,\bar{T},T}} 
 \end{align}
 Furthermore, notice that $\mathcal{M}^{\bmu_1}_{\bar{T}}(w_2,w'_2) = (P_{T-\cdot}\phi_2+ \mathcal{J}(\nabla w_2\cdot\bmu_1), \nabla w_2)$ and $\mathcal{M}^{\bmu_2}_{\bar{T}}(w_2,w'_2) = (P_{T-\cdot}\phi_2+ \mathcal{J}(\nabla w_2\cdot\bmu_2), \nabla w_2)$. As a result, we get 
 $$\mathcal{M}^{\bmu_1}_{\bar{T}}(w_2,w'_2) - \mathcal{M}^{\bmu_2}_{\bar{T}}(w_2,w'_2) = (\mathcal{J}(\nabla w_2\cdot(\bmu_1-\bmu_2)), 0).$$
Using Corollary~2.5 of \cite{CC18}, we have 
\begin{align}\label{eq:2ndInequality}
\|(\mathcal{J}(\nabla w_2\cdot(\bmu_1-\bmu_2)), 0)\|_{\sD^{\alpha,\theta,\delta,p}_{\rho,\bar{T},T}} 
\lesssim \bar{T}^\kappa (1+\| \bmu_1-\bmu_2\|_{\sX^{\gamma}})\| \bmu_1-\bmu_2\|_{\sX^{\gamma}} (1+ \|(w_2, w'_2) \|_{\sD^{\alpha,\theta,\delta,p}_{\rho,\bar{T},T}}).
\end{align}
Applying the bounds from \eqref{eq:1stInequality} and \eqref{eq:2ndInequality} into the right hand side of \eqref{eq:Lipschitz} yields
\begin{align*}
    \text{r.h.s. of \eqref{eq:Lipschitz}} & \lesssim  \|\phi_1 - \phi_2\|_{\sC^{-\epsilon}_p} + \bar{T}^{\kappa} (1+\| \bmu_1 \|_{\sX^{\gamma}})^2\| (w_1, w'_1) - (w_2, w'_2)\|_{\sD^{\alpha,\theta,\delta,p}_{\rho,\bar{T},T}} \\
     &+ \bar{T}^\kappa (1+\| \bmu_1-\bmu_2\|_{\sX^{\gamma}})\| \bmu_1-\bmu_2\|_{\sX^{\gamma}} (1+ \| (w_2,w'_2) \|_{\sD^{\alpha,\theta,\delta,p}_{\rho,\bar{T},T}})
 \end{align*}
The above inequality proves $(\phi,\bmu)\mapsto u^{\phi}_{\bmu}(t,\cdot)$ is locally Lipschitz. This completes the proof. 
\end{proof}

To apply Proposition~\ref{prop:contraction of resolvent equation} and Theorem~\ref{thm:extension to the delta initial data} to the SDE in Theorem~\ref{thm:Feynman-Kac representation}, we need to verify that $\nabla(Z+Y)$ can be lifted to the rough distribution which is done as follows. 

\begin{lemma}\label{lem:lifting lemma}
  Let $\alpha \in(\frac{1}{3} , \frac{1}{2}]$. Suppose that $\mu = \nabla( Z+Y)$ where $Z = (1-\frac{1}{2}\Delta)^{-1} \xi \in \sC^\alpha$ and $Y \in \sC^{2\alpha}$. Then $\mu$ can be lifted a rough distribution $\bmu \in \sX^\gamma$ in the sense of Definition~\ref{def:rough distribution} with $\frac{1}{3}<\gamma<\alpha$. 
\end{lemma}
\begin{proof} We need to verify that the second component of the rough distribution $\cJ(\mu) \circ \mu $ is well-defined, i.e., can be approximated by a suitable sequence in the sense of Definition~\ref{def:rough distribution}. Recalling $Y\in\sC^{2\alpha}$ and $Z\in \sC^{\alpha}$, we know that $\cJ ( \nabla \cdot \nabla Z )\circ \nabla Y$, $\cJ(\nabla \cdot \nabla Y) \circ \nabla Z$ and   $\cJ ( \nabla \cdot \nabla Y) \circ \nabla Y $ are well defined since their regularities are at least $3\alpha-1>0$.  Furthermore, since $\cJ $ and $\nabla $ commute, we have  
   \begin{equation*}
      \cJ(\nabla \cdot \nabla Z) \circ \nabla Z = \nabla\cJ(\nabla Z) \circ \nabla Z.
   \end{equation*} Note that Theorem~\ref{thm:renormalization of enhanced noise} implies that $\nabla \cI ( \nabla Z) \circ \nabla Z)$ is well defined. Since $\cI$ and $\cJ$ have no significant difference on their estimation, Theorem~\ref{thm:renormalization of enhanced noise} ensures that the above term is well defined. Therefore, we can conclude that $\cJ( \mu ) \circ \mu$ given by 
  \begin{equation*}
     \cJ(\nabla \cdot \nabla Z) \circ \nabla Z + \cJ( \nabla \cdot \nabla Z )\circ \nabla Y + \cJ(\nabla \cdot \nabla Y) \circ \nabla Z + \cJ ( \nabla \cdot \nabla Y) \circ \nabla Y
   \end{equation*}  is well defined.
\end{proof}


Since $(\phi,\bmu)\mapsto u^{\phi}_{\bmu}(t,\cdot)$ is locally Lipschitz by Theorem~\ref{thm:extension to the delta initial data}, we can get the following proposition.

\begin{proposition}\label{prop:approximation of transition kernel} Let $\frac{4}{3} < \alpha < \gamma<\beta +2$ with $\beta\in (-\frac{2}{3},-\frac{1}{2}]$. Let $\bmu \in \sX^\gamma$ and $\{\bmu_n\}_{n\in\dN}\subseteq \sC^\infty\times \sC^\infty_{\dR^d}$ be a sequence which converges to $\bmu \in \sX^\gamma$ in $\sH^\gamma$ (see Definition~\ref{def:rough distribution}). Let $\Gamma_t(x,y) =  w^{\delta_y, \mu}(t,x) $ and $\Gamma^{(n)}_t(x,y) = w^{\delta_y, \mu_n}(t,x)$. Then  $\Gamma_t(x,y)$ and $\Gamma^{(n)}_t(x,y)$ are continuous on $\dR^d \times \dR^d$ and for all $t\in[0,T)$, we have 
\begin{equation}
  \sup_{x,y\in \dR^d} | \Gamma_t(x,y)-\Gamma^{(n)}_t(x,y)| \rightarrow 0,
\end{equation} as $n \rightarrow \infty.$

\end{proposition}

 Now we are in the position to prove Theorem~\ref{thm:existence of the transition kernel}, which is the main result of this section.
 \begin{proof}[\bf Proof of Theorem~\ref{thm:existence of the transition kernel}] Lemma~\ref{lem:lifting lemma} shows that $\mu = \nabla(Z+Y)\in\sC^\alpha$ can be lifted to a rough distribution $\bmu\in\sX^\gamma$. Thus, we can apply Theorem~\ref{thm:extension to the delta initial data} to get the unique solution $w^{{\delta_y},\mu}$ for $y\in\dR^d$ since $\delta_y\in \sC^{-d(1-1/p)}_p$ for any $p\in[1,\infty]$. Since $\bmu$ is a rough distribution, we can take a sequence $\bmu_n : = (\mu_n, \mu'_n)\in \sC^{\infty}\times \sC^\infty_{\dR^d}$ for $n\in \dN$, which converges to $\bmu=(\mu , \mu')$ in $\sH^\gamma$ for $\alpha < \gamma< \beta +2$ as in Proposition~\ref{prop:approximation of transition kernel}. It is known that for each $n\in\dN$, there exists the transition density $\Gamma^{(n)}_t(x,y) = w^{\delta_y, \bmu_n}(t,x)$ for all $(t,x,y) \in (0,T)\times \dR^{2d}$ (see for instance \cite[Theorem~6.5.4]{friedman2010stochastic}). Let $\dP^{(n)}_x$ be the unique probability measure on $C([t,T], \dR^d)$ such that the coordinate process $X^{(n)}$ is the solution to the martingale problem for $(\cL^{(n)}, x)$ where $\cL^{(n)}:= \frac{1}{2}\Delta + \mu_n \cdot \nabla$.  By Theorem~\ref{thm:well-posedenss of martingale problem}, we know that $\dP^{(n)}_x$ weakly converges to $\dP_x $ that appears in Theorem~\ref{thm:well-posedenss of martingale problem}. Suppose $\Gamma_t(\cdot,\cdot)$ is the transition kernel of $\mathbb{P}_x$. Therefore, with Proposition~\ref{prop:approximation of transition kernel}, we can see that for all $\psi\in C_c(\dR^d)$ 
 \begin{equation*}
     \dE_x[\psi(X_T)] = \lim_{n\rightarrow \infty}\dE_x[\psi(X^{(n)}_T)] = \lim_{n\rightarrow \infty} \int_{\dR^d}\psi(y) \Gamma^{(n)}_t(x,y) dy = \int_{\dR^d} \psi(y) \Gamma_t(x,y) dy.
 \end{equation*} This shows $\Gamma_t(x,y) = w^{\delta_y,\mu}(t,x)$ and hence, completes the proof. 
 
 \end{proof}



\subsection{Transition kernel estimate}\label{subsec:transition kernel estimates}

We now provide the upper and lower bounds for the estimate of $\Gamma_t(x,y) : = w^{\delta_y,\mu}(t,x)$. We first consider $\mu$ is the form of $\nabla U$ where $U$ denotes a smooth and bounded function. Note that Proposition~\ref{prop:approximation of transition kernel} allows us to extend the estimates to the case $\mu=\nabla (Z+Y)$ as in Section~\ref{subsec:existence of transition density}.

\begin{theorem}\label{thm:transition density estimate}
 Suppose that $\mu$ is the form of $\nabla U$ where $U:\dR^d \rightarrow \dR^d$ is a smooth and bounded function. Then, for the transition density $\Gamma_t(x,y)$ of the solution $X$ to the SDE \eqref{eq:SDE with distributional distribution}, there exist constants $C_1,C_2,C_3>0$ such that for all $t>0, x,y\in \dR^d$, we have 
\begin{align}
  \Gamma_t(x,y) &\leq \frac{1}{t^{d/2}} \exp \left( C_1 \lVert U \rVert_\infty - \frac{|y-x|^2}{4t}\right)\label{eq:upper bound of transition density},\\
  \Gamma_t(x,y) &\geq \frac{1}{t^{d/2}} \exp \left( -C_2\lVert U \rVert_\infty -\frac{C_3 \lVert U \rVert_\infty |y-x|^2}{t}\right).\label{eq:lower bound of transition density}
\end{align} In addition, the same estimates hold when $U=Z^y_L+Y^y_L$ where $Z^y_L$ and $Y^y_L$ are defined in Section~\ref{sec:feynman-kac representation}.
\end{theorem}

\begin{proof}
  The proof is based on ideas from \cite[Section~4.3]{Str08}. It was shown in (4.3.6) of \cite{Str08} that if $\mu$ is of the form $\nabla U$ for some $U\in C^{\infty}(^d)$, then 
  $$\Gamma_t(x,y)\leq \frac{K_d(U)}{t^{d/2}}\exp\Big(\frac{|x-y|^2}{4t}\Big).$$
  where $K_d(U)$ is a constant which is upper bounded by $\kappa_d\exp((d+4)\delta(U)/4)$ and $\delta(U):= \max_{x\in ^d} U(x)- \min_{x\in ^d}U(x)$. Note that $\delta(U)\leq 2\|U\|_{\infty}$. Thus the upper bound \eqref{eq:upper bound of transition density} follows.
 We use \cite[Lemma~4.3.8]{Str08} to show the lower bound. It is shown in the paragraph before Lemma~4.3.8 of \cite{Str08} that there exists $C_2$ such that 
 $$\Gamma_t(x,y)\geq \frac{2^{d/2}}{t^{d/2}}\exp(-C_2\|U\|_{\infty}), \quad \text{whenever} |x-y|\leq \sqrt{t}.$$
This bound is further used in Lemma~4.3.8 of \cite{Str08} to show that 
$$\Gamma_t(x,y)\geq \frac{2^{d/2}}{t^{d/2}}\exp\Big(-C_2\|U\|_{\infty}- \frac{C_3\|U\|_{\infty}|y-x|^2}{t}\Big)$$
for some constant $C_3>0$. From this bound, \eqref{eq:lower bound of transition density} follows.
  
\end{proof}

We also obtain the estimate for the escape probability of $X.$

\begin{corollary}\label{cor:escape probability of X} 
Suppose that $\mu$ is the form of $\nabla U$ where $U:\dR^d \rightarrow \dR^d$ is a smooth and bounded function. Then, for the solution $X$ to the SDE \eqref{eq:SDE with distributional distribution}, $x\in \dR^d$, $K>0$, and $T\geq 1 $,
\begin{equation}
\begin{aligned}
  \dP&\Big( \sup_{t\in [0,T]} |X_t - x| \geq K \Big)\leq C\exp \Big( C T \| U\|_\infty - \frac{K^2}{CT} \Big).
\end{aligned}
\end{equation}  Furthermore, this same upper bound holds for $U=Z^y_L+Y^y_L$ where $Z^y_L$ and $Y^y_L$ are defined in Section~\ref{sec:feynman-kac representation}.
\end{corollary}

\begin{proof}
Note that from \eqref{eq:upper bound of transition density}, we know that $\Gamma_t(x,y) \lesssim p_t(x,y) e^{C_1 \| U\|_\infty}$ where $p_t$ denotes the usual heat kernel density in $d-$dimension. Then the proof follows from the same line with the proof of \cite[Corollary~5.2]{PZ22}. 
\end{proof}

\section{Asymptotic bounds for the PAM}\label{sec:asymptotic bounds for the pam}

This section is devoted to estimating the asymptotic bounds for the solution $u$ to the PAM~\eqref{eq:PAM}. We aim to showcase two main things in this section. We first derive bounds on the moments of the \emph{enhanced noise} (see Definition~\ref{def:DefinitionOfTree}). We use these estimates to find bounds on the Feynman-Kac formula found in Section~\ref{sec:feynman-kac representation}. We first recall few facts about the white noise.

\begin{definition}
On a probability space $(\Omega, \dP),$ we define a (spatial) white noise on $\dR^d$ by a random variable $\xi:\Omega \rightarrow \sS'(\dR^d)$ such that for all $f\in\sS(\dR^d)$, $\langle \xi,f \rangle $ is a centered Gaussian random variable with covariance $\dE[\langle \xi,f \rangle\langle \xi,g \rangle]=\langle f,g\rangle$ for all $f,g\in \sS(\dR^d)$. 
\end{definition}
Note that since $f\mapsto \langle \xi ,f \rangle$ is linear and $\| \langle \xi ,f \rangle \|_{L^2(\Omega)}=\| f\|_{L^2(\dR^d)}$, we can extend $\xi $ to a bounded linear operator $W:L^2(\dR^d)\rightarrow L^2(\Omega ,\dP)$ such that $W(f)$ is also a centered Gaussian random variable with $\dE[W(f)W(g)]= \langle f,g\rangle$ for all $f,g\in L^2(\dR^d)$. We will abuse the notation using $\langle \xi, f \rangle$ for $W(f)$ when $f\in L^2(\dR^d)$. Now we define a radially symmetric and nonnegative function $\psi\in C^\infty_c(\dR^d)$ such that $\int_{\dR^d} \psi(x) dx =1 $. We define $\psi_\epsilon(x)= \frac{1}{\epsilon^d} \psi(\frac{x}{\epsilon})$. Let $\xi_{\epsilon} = \psi_\epsilon * \xi  $ be the mollification of $\xi.$ Recall that the mollification of the noise on the box $[-\frac{L}{2},\frac{L}{2}]^d$ by $\xi_{L,\epsilon}$ given as 
\begin{equation}
    \xi_{L,\epsilon} = \sum_{k\in\dN^d} \tau(\frac{\epsilon}{L}k) \langle \xi , \fn_{k,L} \rangle \fn_{k,L},
\end{equation} where $\tau\in C^\infty_c(\dR^d , [0,1])$ is an even function such that $\tau(x) =1$ for $|x|\leq \frac{1}{2}$ and $\{\fn_{k,L}\}_{k\in \dN^d}$ is the Neumann basis for the $Q_L= [-\frac{L}{2},\frac{L}{2}]^d$ as shown in Section~\ref{sec:AndersonHamiltonian}. We also define $\xi^\epsilon_{L,y}$ on $Q^y_L:= y+Q_L$ for $y\in \dR^d$ by a shifted white noise 
\begin{equation}
    \xi^y_{L,\epsilon} : = \cT_y \left(  \sum_{k\in\dN^d} \tau(\frac{\epsilon}{L}k) \langle \xi , \cT_y \fn_{k,L} \rangle \fn_{k,L}\right)  
\end{equation} where $\cT_y f(\cdot) = f(\cdot +y)$ for all $f\in C(\dR^d)$. Define $\sigma_\eta(x)=\frac{1}{\eta +\pi^2|x|^2}$ for all $x\in\dR^d$. Then, for $u\in \sS_{\fn}'(Q_L),$  
\begin{equation}
    (\eta-\Delta)^{-1} u =: \sigma_\eta(D) u: = \sum_{k\in \dN^d} \sigma_\eta(\frac{k}{L})\langle u, \fn_{k,L}\rangle \fn_{k,L}, 
\end{equation}  We write $\sigma(D) := \sigma_1(D)$. For any $\delta>0$, it is known that almost surely $\xi^y_{L,\epsilon} \rightarrow \xi^y_L$ in $\sC^{-d/2-\delta}$ for all $y\in \dR^d$ from \cite[Theorem~6.7]{CZ21}.

\subsection{Bound on the enhanced noise}\label{subsec:bound on the enhanced noise}

From the Feynman-Kac representation (see \eqref{eq:feynman-Kac representation} and \eqref{eq:exponent in Feynman-Kac representation} in Theorem~\ref{thm:Feynman-Kac representation}), we need to bound the terms $Z,Y$ and $\eta$ therein in order to estimate the solution. From the view of Proposition~\ref{prop:resolvent eqaution} and Definition~\ref{def:enhanced noise}, these terms are bounded in terms of enhanced white noise.

Recall the following notations from the previous sections:
For $L\in (1,\infty]$ and $\epsilon\in[0,1]$, we denoted $\xi_{L,\epsilon}$ to be the mollification of the spatial white noise in $d-$dimension restricted on $Q_L$ with $Q_\infty:=\dR^d$.
Set $Z_{L,\epsilon} = (1-\frac{1}{2}\Delta)^{-1}\xi_{L,\epsilon} $. Let $\eta_{L,\epsilon}>0$ be $\eta_0$ in Proposition~\ref{prop:resolvent eqaution} which will be specified later and  $Y_{L,\epsilon}$ be the solution to \eqref{eq:original resolvent eq for Y} associated with $Z_{L,\epsilon}$ and $\eta=\eta_{L,\epsilon}.$ $\fZ_{L,\epsilon}$ is defined by Theorem~\ref{thm:renormalization of enhanced noise} with $\xi_{L,\epsilon}$ instead of $\xi_\epsilon$. 

The following proposition says that $\fZ_{L,\epsilon}$ is bounded by the logarithm of the length $L$, which enables us to bound $Z_{L,\epsilon}$, $Y_{L,\epsilon}$, and $\eta_{L,\epsilon}$. The following result is for $d=3$. We refer to \cite[Lemma~6.15]{KPZ20} for $d=2$ case. 

\begin{proposition}\label{prop:bounds on the enhance noise} Let $\frac{2}{5} <\varrho<\frac{1}{2}$. Let $\epsilon\in[0,1]$ and define 
\begin{equation}\label{eq:definition of fa_epsilon}
\fa_{\epsilon}: = \max \left\{1, \sup_{L>e,L\in \dN}\frac{\| \fZ_{L,\epsilon}\|_{\cZ^\varrho}}{(\log L)^2}\right\}.
\end{equation} Then $\fa_\epsilon$ is almost surely finite. Moreover, there exists $h_0>0$ such that for all $h\in[0,h_0]$ we have $\sup_{\epsilon\in[0,1]} \dE[e^{h\sqrt{\fa_\epsilon}}]<\infty.$
\end{proposition}

\begin{proof}
Before proceeding to the main body of the proof, let us briefly explain the main idea. Below, we first show how to express $\mathfrak{a}_{\epsilon}$ with the help of inhomogeneous terms (of bounded order) from the Wiener chaos expansion of the white noise. Due to the hyper-contractivity of the Wiener chaos (see \cite{Nual2006}), for every random variable $X$ in the $k$-th inhomogeneous Wiener chaos generated by the white noise and for any $p>2$ we have 
\begin{equation}\label{eq:hyper-contractivity of the Wiener chaos}
  \dE[|X|^p]\leq C_{k,p} \dE[|X|^2]^{p/2},
\end{equation} where $C_{k,p} = (p-1)^{pk/2}.$ This allows us to bound higher moments $\mathbb{E}[(\sqrt{\mathfrak{a}_{\epsilon}})^p]$ by some constant multiple time $\mathbb{E}[\mathfrak{a}_{\epsilon}]^p$ which finally leads to the desired results. Details of this argument is given as follows. 
Recall that $ \Xi_{L,\epsilon}:= \{Z_{L,\epsilon},  Z^{\tree}_{L,\epsilon}-c^{\tree}_\epsilon, Z^{\ttree}_{L,\epsilon}  ,  Z^{\tttree}_{L,\epsilon} ,Z^{\Tree}_{L,\epsilon}-c^{\Tree}_\epsilon , \nabla Q_{L,\epsilon} \circ \nabla Z_{L,\epsilon}  \}$. We know that 
\begin{align}
\|\fZ_{L,\epsilon} \|_{\cZ^\varrho} =\sum_{\zeta_{L,\epsilon}\in\Xi_{L,\epsilon}} \| \zeta_{L,\epsilon}\|_{\sC^{-\alpha}} 
\end{align} where $\alpha>0$ depends on each $\zeta_{L,\epsilon}\in \Xi_{L,\epsilon}$. We bound each term of the right hand side of the above identity to derive the desired bound.
Let $\zeta_{L,\epsilon} \in \Xi_{L,\epsilon}$. Observe that $\zeta_{L,\epsilon}$ is in the $k$-th Wiener chaos where $k=k(\zeta)$ denotes the number of the occurrences of the noise $\xi_{L,\epsilon}$ in $\zeta_{L,\epsilon}$. 
For instance, $k(Z)=1$, $k({Z^{\ttree}})=3$, and $k({Z^{\tttree}})=4$. Note also that $\max_{\zeta}k(\zeta) =4$.  At this moment, we assume that for every $\zeta_{L,\epsilon} \in \Xi_{L,\epsilon} $, there exist some constant $a\in \dR$ and $C_0>0$ (not depending on $L,\epsilon$) such that all $i\in \dN_{-1}, x\in [-L/2,L/2]^d$, we have 
\begin{equation}\label{eq:bound of Delta zeta}
     \dE[|\Delta_i \zeta_{L,\epsilon}(x)|^2] \leq C_0 2^{a i }.
\end{equation} 
We prove these bounds towards the end of the proof of this proposition. 
Set $C_{\kappa}= \sum_{i=-1}^{\infty} 2^{-\kappa i}$. Observe that we can use the definition of the Besov space (see Section~\ref{sec:Paracontrolled}) and  \eqref{eq:hyper-contractivity of the Wiener chaos} to obtain
$$\mathbb{E}\Big[\|\zeta_{L,\epsilon}\|^p_{B^{n,-\frac{a}{2}-\kappa}_{p,p}}\Big]= \sum_{i=-1}^{\infty} 2^{(-\frac{a}{2}-\kappa)pi} \mathbb{E}\big[\|\Delta_i\zeta_{L,\epsilon}\|^{p}_{L^p}\big]\leq C_0^{\frac{p}{2}}p^{\frac{pk}{2}}L^d\Big(\sum_{i=-1}^{\infty} 2^{-p\kappa i}\Big)\leq C_\kappa (C_0p^k)^{\frac{p}{2}} L^d .$$
By the embedding property of the Besov space, there exists $C>0$ such that $\|\cdot\|_{\sC^{-\frac{a}{2}-\kappa -\frac{2}{p}}}\leq C\|\cdot\|_{B^{n,-\frac{a}{2}-\kappa}_{p,p}}$. Combining this with the observation in the above display implies that
for all $\kappa>0$ there exists $C>0$ independent of $\zeta$ such that for all $p\geq1$,
\begin{equation}
  \dE[\|\zeta_{L,\epsilon}\|^p_{\sC^{-\frac{a}{2}-\kappa -\frac{2}{p}}} ] \leq C_\kappa C^p L^d p^{\frac{pk}{2}}. 
\end{equation} Choose $p_0\in \dN$ large so that $\frac{2}{p_0}<\kappa$ and $\frac{2p_0}{k}> 2 $. Then there exists $C_k>0$ such that for $p\geq p_0$
\begin{equation}
     \dE[\|\zeta_{L,\epsilon}\|^{\frac{2p}{k}}_{\sC^{-\frac{a}{2}-2\kappa}} ] \leq C_\kappa C_{k}^{p} L^d p^p. 
\end{equation}
For $h\geq0$, by the above upper bounds, we have 
 \begin{equation*}
   \begin{aligned}
     \dE [\exp(h \|\zeta_{L,\epsilon}\|^{\frac{2}{k}}_{\sC^{-\frac{a}{2}-2\kappa}})] & = \sum_{n=0}^\infty \frac{h^n}{n!} \dE[\|\zeta_{L,\epsilon}\|^{\frac{2n}{k}}_{\sC^{-\frac{a}{2}-2\kappa}} ] \leq \sum_{n=0}^{p_0} \frac{h^n}{n!} \dE[\|\zeta_{L,\epsilon}\|^{2p_0}_{\sC^{-\frac{a}{2}-2\kappa}} ]^{\frac{n}{kp_0}} + \sum_{n=p_0}^\infty \frac{h^n}{n!} \dE[\|\zeta_{L,\epsilon}\|^{\frac{2n}{k}}_{\sC^{-\frac{a}{2}-2\kappa}} ] \\
     &\leq \sum_{n=0}^{p_0-1} \frac{h^n }{n!}  ( C_\kappa  C^{2p_0} L^d (2p_0)^{p_0k})^{\frac{n}{kp_0}} + \sum_{n=p_0}^\infty \frac{h^n}{n!} C_\kappa  C_k^n L^d n^n\\
     &\leq C_\kappa^k L^{\frac{d}{k}} \exp ( 2hC^{\frac{2}{k}}p_0 ) + C_\kappa L^d \sum_{n=p_0}^\infty (heC_k)^n,
   \end{aligned} 
 \end{equation*} where we have used Jensen's inequality and $\frac{1}{n!} \leq (\frac{e}{n})^n$. Then, there exists $h_0$ such that for all $h\in [0,h_0]$ we have 
 \begin{equation*}
   \dE [\exp(h \|\zeta_{L,\epsilon}\|^{\frac{2}{k}}_{\sC^{-\frac{a}{2}-2\kappa}})] \leq A L^{\frac{2}{k}\vee d}
 \end{equation*} for some $A>0$ which is independent of $k$. We can choose $b>\frac{d}{k} \vee d+1$ such that 
 \begin{equation*}
    \sum_{L\in \dN} L^{-b} \dE [ \exp (h_0 \|\zeta_{L,\epsilon}\|^{\frac{2}{k}}_{\sC^{-\frac{a}{2}-2\kappa}})] < \infty.
  \end{equation*} This implies that for $\tilde{\zeta}_{\epsilon} : = \sum_{L\in \dN} L^{-b} \exp (h_0 \|\zeta_{L,\epsilon}\|^{\frac{2}{k}}_{\sC^{-\frac{a}{2}-2\kappa}})$, we have 
  \begin{equation*}
    \frac{ \|\zeta_{L,\epsilon}\|^{\frac{2}{k}}_{\sC^{-\frac{a}{2}-2\kappa}}}{\log L } \leq \frac{1}{h_0} (b + \log \tilde{\zeta}_{\epsilon}) 
  \end{equation*} for all $L\in \dN$ with $L>e$. We can rewrite this as 
  \begin{equation}\label{eq:upper bound of zeta_L,epsilon}
    \|\zeta_{L,\epsilon}\|_{\sC^{-\frac{a}{2}-2\kappa}} \leq \fa_{\zeta,\epsilon} (\log L)^{\frac{k}{2}},
  \end{equation} where $\fa_{\zeta,\epsilon}:= (\frac{b+\log \tilde{\zeta}_\epsilon}{h_0})^2$ since $k(\zeta)\leq 4$. Moreover, Jensen's inequality implies that 
 \begin{equation}\label{eq:finiteness of a_zeta,epsilon}
     \sup_{\epsilon\in[0,1]} \dE[e^{ h \sqrt{\fa_{\zeta,\epsilon}}}] \leq e^{\frac{hb}{h_0}} \sup_{\epsilon\in[0,1]} \dE[\tilde{\zeta}_{\epsilon}]^{\frac{h}{h_0}}<\infty
 \end{equation} for all $h\in[0,h_0]$. From \eqref{eq:upper bound of zeta_L,epsilon}, we have 
  \begin{equation}
      \|\fZ_{L,\epsilon} \|_{\cZ^\varrho} =\sum_{\zeta_{L,\epsilon}\in\Xi_{L,\epsilon}} \| \zeta_{L,\epsilon}\|_{\sC^{-\frac{a}{2}-2\kappa}} \leq (\log L)^2\sum_{\zeta_{L,\epsilon}\in\Xi_{L,\epsilon}} \fa_{\zeta,\epsilon} =: (\log L)^2\tilde{\fa}_\epsilon,
  \end{equation} where $\tilde{\fa}_\epsilon: = \sum_{\zeta_{L,\epsilon}\in\Xi_{L,\epsilon}} \fa_{\zeta,\epsilon} .$  Since $\sqrt{x+y}\leq \sqrt{x}+\sqrt{y}$ for $x,y\geq0$, we can take $\tilde{h}_0>0$ such that for all $\tilde{h}\in[0,\tilde{h_0}]$
  \begin{equation}
      \sup_{\epsilon\in[0,1]}\dE [e^{\tilde{h}\sqrt{\tilde{\fa}_\epsilon}}] \leq   \sup_{\epsilon\in[0,1]}\dE \Big[\exp\Big(\tilde{h}\sum_{\zeta_{L,\epsilon}\in\Xi_{L,\epsilon}} \sqrt{\fa_{\zeta,\epsilon}}\Big)\Big] \leq \sup_{\epsilon\in[0,1]}\prod_{\zeta_{L,\epsilon}\in \Xi_{L,\epsilon}} (\dE[e^{5\tilde{h}_0\sqrt{\fa_{\zeta,\epsilon}}}])^{\frac{1}{5}}<\infty,
  \end{equation} where we used H\"older's inequality and \eqref{eq:finiteness of a_zeta,epsilon}. Since $\fa_\epsilon$ in \eqref{eq:definition of fa_epsilon} satisfies $\fa_\epsilon \leq \tilde{\fa}_\epsilon$, we get the desired result.  Now it remains to show \eqref{eq:bound of Delta zeta}. 
  In particular, we can show that for any fixed $\delta\in(0,1)$,
   \begin{equation*}
     \begin{aligned}
     &\dE[|\Delta_i Z_{L,\epsilon}(x)|^2] \lesssim 2^{(-1+\delta)i}, \quad 
\dE[|\Delta_i (Z^{\tree}_{L,\epsilon}(x)-c^{\tree}_\epsilon)|^2] \lesssim 2^{(-2+\delta)i}, \quad
      \dE[|\Delta_i Z^{\ttree}_{L,\epsilon}(x)|^2] \lesssim 2^{(-3+\delta)i}\\
     & \dE[|\Delta_i Z^{\tttree}_{L,\epsilon}(x)|^2] \lesssim 2^{(-3+\delta)i},\quad \dE[|\Delta_i (Z^{\Tree}_{L,\epsilon}(x) - c^{\Tree}_\epsilon)|^2] \lesssim 2^{(-4+\delta)i},\quad \dE[|\Delta_i (\nabla Q_{L,\epsilon} \circ \nabla Z_{L,\epsilon})(x)|^2] \lesssim 2^{\delta i}.
     \end{aligned}
 \end{equation*}
 Indeed, these estimates were shown for the parabolic case on the torus in the proof of \cite[Theorem~6.12]{CC18}. With a similar argument as in the proof of Proposition~\ref{prop:contraction of resolvent equation}, we can obtain the same estimates for the $L_2$ bound of the Wiener-chaos components for the elements in $\Xi_{L,\epsilon}$ with minor modifications. We demonstrate below the bound on $\dE[|\Delta_i Z_{L,\epsilon}(x)|^2] $. Other cases follow by similar arguments. 

Recall from Appendix~\ref{sec:Paracontrolled} that for $f\in \sS'$
 and for any $\delta>0$,
$$\Delta_i f  = \sum_{k\in\dN_0^d} \langle f, \fn_{k,L} \rangle \varrho_i\Big(\frac{k}{L}\Big)\fn_{k,L},\quad \text{where } \quad  \varrho_i(x) \lesssim \left( \frac{2^i}{1+|x|}\right)^{\frac{3+3\delta}{2}}.
$$  
Since $\varrho_i(x) = 0$ if $|x| \leq 2^i$, we have 
\begin{equation*}
    \sigma \Big(\frac{k}{L}\Big)\varrho_i\Big(\frac{k}{L}\Big) \lesssim 2^{-2i} \varrho_i\Big(\frac{k}{L}\Big) \lesssim 2^{\frac{-1+3\delta}{2}i} \left( \frac{1}{1+\left|\frac{k}{L}\right|}\right)^{\frac{3+3\delta}{2}} ,
\end{equation*} where $\sigma(x) = (1 + \pi|x|^2)^{-1}$. 
Using similar computation as in \cite[Lemma~6.11]{KPZ20}, one gets $\dE[\langle \xi_{L,\epsilon}, \mathfrak{n}_{k,L}\rangle \langle \xi_{L,\epsilon}, \mathfrak{n}_{l,L}\rangle]\leq \prod_{i=1}^3(\epsilon\wedge ((k_i+\ell_i)\vee 1)^{-1})$. 
Since $\|\fn_{k,L}\|_{L^\infty} \lesssim L^{-3/2}$ and  noting that $Z_{L,\epsilon}(x) = \cI(\xi_{L,\epsilon})(x) = (1-\frac{1}{2}\Delta)^{-1}\xi_{L,\epsilon}(x)$, we have  
\begin{equation*}
    \begin{aligned}
       2^{-(-1+3\delta)i} \dE[|\Delta_i Z_{L,\epsilon}(x)|^2] &\lesssim L^{-3} \sum_{k,l\in \dN^3_0} \frac{1}{(1+ |\frac{k}{L}|)^{\frac{3+3\delta}{2}}} \frac{1}{(1+ |\frac{l}{L}|)^{\frac{3+3\delta}{2}}} \big| \dE[\langle \xi_{L,\epsilon}, \mathfrak{n}_{k,L}\rangle \langle \xi_{L,\epsilon}, \mathfrak{n}_{l,L}\rangle]\big|\\
       &\lesssim \left( \sum_{k,l \in \frac{1}{L}\dN_0 } L^{-3} \frac{1}{(1+k)^{\frac{1+\delta}{2}}} \frac{1}{(1+l)^{\frac{1+\delta}{2}}} \frac{1}{(k+l)\vee 1 }\right)^{3}\\
       &\lesssim \left( \sum_{k\in \frac{1}{L} \dN_0 } L^{-1} \frac{1}{(1+k)^{ 1+ \frac{\delta}{2}}}  \right) ^9 \leq C,
    \end{aligned}
\end{equation*}  where $C$ is an universal constant. The second inequality is obtained by symmetrizing the sum. The above display shows the bound on $\dE[|\Delta_i Z_{L,\epsilon}(x)|^2]$. This completes the proof. 
 
\end{proof} 

Using Proposition~\ref{prop:bounds on the enhance noise}, we can obtain the following bounds on $Z_{L,\epsilon}$, $Y_{L,\epsilon}$, and $\eta_{L,\epsilon}$.

\begin{proposition}\label{prop:bounds on Z,Y,eta} Let $\frac{2}{5} < \alpha <\frac{1}{2}$ and $\aleph>0$ be the constant in \eqref{eq:definition of eta_L}. Then for any $\epsilon\in[0,1]$, we have 
\begin{equation}\label{eq:estimates of Z,Y,eta}
  \|Z_{L,\epsilon} \|_{\sC^\alpha} \leq \fa_\epsilon (\log L)^2 , \quad \| Y_{L,\epsilon}\|_{\sC^{2\alpha}} \leq 2\fa_\epsilon (\log L)^2 , \quad \eta_{L,\epsilon}\leq C\fa_\epsilon^\aleph  (\log L)^{2\aleph}
\end{equation} where $\fa_\epsilon$ is defined in Proposition~\ref{prop:bounds on the enhance noise}.
\end{proposition}

\begin{proof}
From the definition of $\mathfrak{a}_{\epsilon}$, we know $\|Z_{L,\epsilon} \|_{\sC^\alpha} \leq \fa_\epsilon (\log L)^2$. Furthermore, it also clear that if we take $\eta_{L,\epsilon}:= C \| \fZ \|^\upsilon_{\cZ^\varrho}$, then $\eta_{L,\epsilon}\leq C^\aleph \fa^\aleph_\epsilon(\log L)^{2\aleph}$. We are left to show the bound on $\|Y_{L,\epsilon}\|_{\sC^{2\alpha}}$. Recall from Section~\ref{sec:feynman-kac representation} that $Y_{L,\epsilon} =v+\frac{1}{2}(Z^{\tree}_{L,\epsilon}+ Z^{\ttree}_{L,\epsilon})$ where $v$ is the fixed point solution of the equation \eqref{eq:resolventEquation for v}. Recall the map $\cG : \cD^\alpha_{\cQ} \rightarrow \sC^{\alpha+1} \times \sC^\alpha $ from Proposition~\ref{prop:contraction of resolvent equation}. Note that $(v,v')$ is the fixed point of the map $\cG$. By choosing $\eta$ large in \eqref{eq:cGBound} of Proposition~\ref{prop:contraction of resolvent equation}, it follows that $\|v\|_{\sC^{\alpha+1}}\leq C\| \fZ \|_{\cZ^\varrho}$ for some $C>0$. Since $Y_{L,\epsilon} =v+\frac{1}{2}(Z^{\tree}_{L,\epsilon}+ Z^{\ttree}_{L,\epsilon})$, we have for $\alpha<\varrho<\frac{1}{2}$
\begin{equation*}
    \|Y_{L,\epsilon}\|_{\sC^{2\alpha}} \leq (C+1)\| \fZ\|_{\cZ^\varrho} \leq (C+1) \fa_\epsilon (\log L)^2.
\end{equation*}
This completes the proof.
\end{proof}


\subsection{Asymptotics of PAM started from constant initial data}\label{subsec:reduction to a box and asymtotic bounds for the solution}
In this section, we reveal how the solution of the parabolic Anderson model started from constant initial data is related to the largest point of the spectrum of Anderson Hamiltonian. We achieve this goal in Lemma~\ref{lem:large time asymptotics of the local proxy} and~\ref{lem:large time aysmptotics of the solution}. It is worthwhile to note that similar claims have been proved for $d=2$ case in \cite{KPZ20}. Showing those results for $d=3$ needs new estimates which are provided in Lemma~\ref{lem:BoxEstimate} and Proposition~\ref{prop:upper and lower bound of the localized solution}. 

\begin{lemma}\label{lem:BoxEstimate}
Recall $\dQ^{x,y}_{L,\epsilon}$ and the diffusion $X$ from Section~\ref{sec:AndersonHamiltonian}. Furthermore, recall $\sD^y_{L,\epsilon}$ defined in terms of $Z^y_{L,\epsilon}$, $Y^y_{L,\epsilon}$ and $\eta$ from Section~\ref{sec:transition density estimate}. Then, we have 
\begin{equation}\label{eq:escape probability on box}
  \begin{aligned}
    \dQ^{x,y}_{L,\epsilon}( X{[0,t]} \not\subset Q^y_L )\leq C\exp \left(Ct\fa_\epsilon (\log L)^2- \frac{L^2}{C t}\right),
  \end{aligned}
\end{equation} and 
\begin{equation}\label{eq:upper bound of localized solution}
    \dE_{\dQ^{x,y}_{L,\epsilon}} \left[ \1_{ X[0,t] \not\subset Q^y_r, X[0,t] \subset Q^y_L} \cdot \sD^y_{L,\epsilon}(0,t)\right] \leq C  \exp \left( C\fa_\epsilon^{\aleph+1}(\log L)^{2\aleph +2} - \frac{r^2}{Ct}\right)
\end{equation}
where $\fa_\epsilon$ is same as in the one defined in \eqref{eq:definition of fa_epsilon}. 
\end{lemma}

\begin{proof}
  In what follows, we  use the upper bound on $\|Z^y_{L,\epsilon} \|_{\sC^\alpha}$, $ \| Y^y_{L,\epsilon}\|_{\sC^{2\alpha}} $ and  $\eta^y_{L,\epsilon}$ from Proposition~\ref{prop:bounds on Z,Y,eta} to derive upper and lower bound on $\sD^{y}_{L,\epsilon}$. 
  Combining \eqref{eq:estimates of Z,Y,eta} with the  definition of $\sD^y_{L,\epsilon}$ \eqref{eq:exponent in Feynman-Kac representation}, for $L>1$, $s,t\geq 0$ and $s<t$, we have 
\begin{equation}\label{eq:upper and lower bound of sD}
  \begin{aligned}
    e^{-C\fa_\epsilon^{\aleph+1}(t-s)(\log L)^{2\aleph+2}} \leq \1_{X{[s,t]} \subset Q^y_L} \cdot \sD^{y}_{L,\epsilon}(s,t) \leq e^{C\fa_\epsilon^{\aleph+1}(t-s)(\log L)^{2\aleph+2}},
  \end{aligned}
\end{equation} where $C>0$ is an absolute constant and $\fa_{\epsilon}$ is same as in \eqref{eq:definition of fa_epsilon}. To complete the proof of the inequalities in \eqref{eq:escape probability on box} and \eqref{eq:upper and lower bound of sD}, we further need bound on the transition probability of the diffusion $X$ as defined in \eqref{eq:diffusion_X}. We derive those in the following using the estimates of Theorem~\ref{thm:transition density estimate}.      
Recall the upper and lower bound on  the transition kernel $\Gamma_t^L(x,y)$ from \eqref{eq:upper bound of transition density} and \eqref{eq:lower bound of transition density} respectively. We may bound $\|U_{L,\epsilon}\|=\|Z_{L,\epsilon}+Y_{L,\epsilon}\|_{\infty}$ by $\|Z_{L,\epsilon}\|_{\sC^{\alpha}}+ \|Y_{L,\epsilon}\|_{2\sC^{\alpha}}$ which is further bounded above by $C\mathfrak{a}_{\epsilon}(\log L)^2$  due to \eqref{eq:estimates of Z,Y,eta}. As a result, we get for $L>r>e$ and $L,r\in\dN$
\begin{equation}\label{eq:upper and lower bound of transition kernel on box}
  \begin{aligned}
   &\Gamma^L_t(x,y)  \geq  \frac{1}{t^{d/2}} \exp\left( -C \fa_\epsilon(\log L )^2 (1+ \frac{r^2}{t})\right) \\
   &\Gamma^L_t(x,y) \leq p_t(x,y) e^{\fa_\epsilon(\log L)^{2}},
  \end{aligned}
\end{equation} 
where $p_t$ is the transition density of the $d-$dimensional Brownian motion. Note that $\dQ^{x,\epsilon}_{L,y}( X{[0,t]} \not\subset Q^y_L )$ is bounded above by $\mathbb{P}\big(\sup_{s\in [0,t]}|X_s-x|\geq L/2\big)$. Corollary~\ref{cor:escape probability of X} bounds the last probability by $C\exp(Ct\|Z_{L,\epsilon}+Y_{L,\epsilon}\|_{\infty}- \frac{L^2}{CT})$. Due to Proposition~\ref{prop:bounds on Z,Y,eta}, we may bound $\|Z_{L,\epsilon}+Y_{L,\epsilon}\|_{\infty}$ by $C\fa_{\epsilon}(\log L)^2$. Plugging this into the bound $C\exp(Ct\|Z_{L,\epsilon}+Y_{L,\epsilon}\|_{\infty}- \frac{L^2}{CT})$ yields the inequality \eqref{eq:escape probability on box}. To prove \eqref{eq:upper bound of localized solution}, we first apply \eqref{eq:upper and lower bound of sD} to obtain
$$\dE_{\dQ^{x,y}_{L,\epsilon}} \left[ \1_{ X[0,t] \not\subset Q^y_r, X[0,t] \subset Q^y_L} \cdot \sD^y_{L,\epsilon}(0,t)\right] \leq \dQ^{x,y}_{L,\epsilon}( X{[0,t]} \not\subset Q^y_L )e^{C\fa_\epsilon^{\aleph+1}(t-s)(\log L)^{2\aleph+2}}.$$
   Substituting \eqref{eq:escape probability on box} into the right hand side of the above display yields \eqref{eq:upper bound of localized solution}.  
\end{proof}

 Now we use the bounds of Lemma~\ref{lem:BoxEstimate} to upper and lower bound of the solution of PAM \eqref{eq:localized PAM}
 started from constant initial data. Our proof ideas will be similar in spirit to \cite[Lemma~3.2]{KPZ20}.
\begin{proposition}\label{prop:upper and lower bound of the localized solution}
Let $L>1$. Recall that $ u^{\1,y}_{L}$ is the solution of PAM restricted on box $Q^y_L$ under Dirichlet boundary condition started from constant initial data $\1$. Then we have
\begin{equation}\label{eq:upper bound of the local proxy}
   u^{\1,y}_{L}(t,x)\leq C\exp \Big( t \blambda_1(Q^y_{L})+ C\fa_0^{\aleph+1} (\log L)^{2\aleph+2} \Big),
\end{equation} where $\fa_0: = \lim_{\epsilon\to 0}\fa_\epsilon$. Moreover, for $t>\delta>1$
\begin{align}\label{eq:lower bound of the local proxy}
   u^{\1,y}_L(t,x) &\geq\exp\left( -C \fa_0^{\aleph+1}\delta (\log L)^{2\aleph +2} -\frac{C\fa_0(\log L)^2r^2}{\delta } +(t-\delta)\blambda_1(Q^y_r) \right) \\ &-\exp\left(C\fa_0^{2\aleph+1}t(\log L)^{2\aleph+2}-\frac{L^2}{C\delta } \right)
\end{align}
where $\aleph$ is the same constant as in Proposition~\ref{prop:bounds on Z,Y,eta}. 
\end{proposition}

\begin{proof}  We start by writing that 
\begin{equation}\label{eq:representation of u on the box with flat initial data}
  u^{\1,y}_{L}(t,x) = \dE_{\dQ^{x,y}_{L} } \left[\sD^{y}_{L}(0,t) \1_{X[0,t]\subset Q^y_{L}} \right] = \int_{Q^y_{L}} u^{x,y}_{L}(t,z) dz.
\end{equation} 
where  $u^{x,y}_{r}(t,z) : =u^{\delta_x,y}_{L}(t,z)$. Note that $u^{x,y}_{r}(t,z) : = u^{\delta_x,y}_{L}(t,z) =  \lim_{\epsilon\rightarrow 0} u^{\psi^x_\epsilon,y}_{L}(t,z)$ where $\psi^x_\epsilon(z): = \psi_\epsilon(z-x)\in C^{\infty}(Q_L)$ (see \cite{Lab19}). For $\delta\in(1,t), \epsilon\in [0,1]$ and $0<r<L$, we have 
\begin{equation}\label{eq:tower property}
\begin{aligned}
  u^{x,y}_{r}(t,z) & = \lim_{\epsilon\rightarrow 0 } \dE_{\dQ^{x,y}_{r,\epsilon}}\left[\sD^y_{r,\epsilon}(0,t) \psi^x_\epsilon(X_t)\1_{X[0,t]\subset Q^y_r} \right]\\
  &\leq \lim_{\epsilon\rightarrow 0} e^{C\fa^{\aleph+1}_\epsilon\delta(\log r)^{2\aleph +2}} \dE_{\dQ^{x,y}_{r,\epsilon}}\left[ \sD^y_{r,\epsilon}(0,t-\delta) \1_{X[0,t-\delta]\subset Q^y_r} \dE_{\dQ^{x,y}_{r,\epsilon}} [\psi^x_\epsilon(X_t)|X_{t-\delta}]\right]\\
  &= \lim_{\epsilon\rightarrow 0} e^{C\fa^{\aleph+1}_\epsilon\delta(\log r)^{2\aleph +2}} \dE_{\dQ^{x,y}_{r,\epsilon}}\left[ \sD^y_{r,\epsilon}(0,t-\delta) \1_{X[0,t-\delta]\subset Q^y_r}\int_{\dR^3} \Gamma^r_\delta(X_{t-\delta},z') \psi^x_\epsilon(z')dz'\right]\\
  &\leq Ce^{C\fa^{\aleph+1}_0\delta(\log r )^{2\aleph+2} + \fa_0(\log r)^2}\dE_{\dQ^{x,y}_{r}}\left[ \sD^y_{r}(0,t-\delta) \1_{X[0,t-\delta]\subset Q^y_r} \right]\\
    &\leq C e^{C\fa^{\aleph+1}_0\delta (\log r)^{2\aleph +2}} u_r^{\1,y}(t-\delta,z).
\end{aligned}
\end{equation} The first inequality in the above display is obtained by splitting $\sD^y_{r,\epsilon}(0,t) \1_{X[0,t]\subset Q^y_r} $ as a product of $\sD^y_{r,\epsilon}(0,t-\delta) \1_{X[0,t-\delta]\subset Q^y_r} $ and $\sD^y_{r,\epsilon}(t-\delta,t) \1_{X[t-\delta,t]\subset Q^y_r} $ and  \eqref{eq:upper and lower bound of sD} to bound $\sD^y_{r,\epsilon}(t-\delta,t) \1_{X[t-\delta,t]\subset Q^y_r} $. The second to last inequality is obtained by applying the bound on the transition kernel $\Gamma^r_{\delta}(X_{t-\delta},z')$ as shown in \eqref{eq:upper and lower bound of transition kernel on box}. From \eqref{eq:representation of u on the box with flat initial data} and \eqref{eq:upper and lower bound of sD} again, we have for $q\in[1,\infty)$
\begin{equation*}
   \left(  \int_{Q^y_r} |u_r^{\1,y}(t-\delta,z)|^q dz \right)^{\frac{1}{q}} \lesssim e^{C\fa^{\aleph+1}_0 (t-\delta)(\log r)^{2\aleph+2}}.
\end{equation*} This implies that for $q\in[1,\infty]$,
\begin{equation}\label{eq:L_q bound of the solution with delta initial data}
  \begin{aligned}
    \| u^{x,y}_{r}(t,\cdot)\|_{L^q} \leq Ce^{ C\fa^{\aleph+1}_0 t(\log r)^{2\aleph +2}}.
  \end{aligned}
\end{equation} Now observe that for $\phi\in C(Q_L)$, using H\"older inequality,
\begin{equation}
    \begin{aligned}
     \int_{Q^y_L} u^{\phi,y}_{L}(t,x) d x= \sum_{n\in \dN } e^{t\blambda(Q^y_L)} \langle v^y_{n,L},\phi \rangle  \langle v^y_{n,L},\1_{Q^y_L} \rangle \leq e^{t\blambda(Q^y_L)} \| \phi\|_{L^2} \| \1_{Q^y_L} \|_{L^2}. 
    \end{aligned} 
\end{equation} Set $\phi = u^{x,y}_{L}(1,\cdot) $. Then by the Chapman-Kolmogorov equation, we know $u^{\1,y}_{L} (t,x) = \int_{Q^y_L} u^{\phi,y}_{L}(t-1,z) dz$. Applying \eqref{eq:representation of u on the box with flat initial data} to the right hand side of the latter equation in conjuction with \eqref{eq:tower property} and \eqref{eq:L_q bound of the solution with delta initial data} yields  
\begin{equation}
    u^{\1,y}_{L} (t,x)\lesssim \exp ( t \blambda(Q^y_L)+C\fa^{\aleph+1}_0(\log L)^{2\aleph+2}).
\end{equation}

Now we proceed to prove the lower bound. Using Markov property at time $\delta\in(1,t)$ and lower bound on $\1_{X{[0,\delta]} \subset Q^y_L} \cdot \sD^{y}_{L,\epsilon}(0,\delta) $ from \eqref{eq:upper and lower bound of sD}, we have
\begin{equation}
    u^{\1, y}_L(t,x) \geq e^{-C\fa^{\aleph+1}_0\delta(\log L)^{2\aleph+2}} \dE_{\dQ^{x,y}_{L}} \Big[\1_{X[0,\delta]\subset Q^y_L} u^{\1,y}_{L}(t-\delta,X_\delta)\Big].
\end{equation} Then for $r\in(0,L), $ we have 
\begin{equation}\label{eq:lower bound of the solution into two boxes}
    \begin{aligned}
      \dE_{\dQ^{x,y}_L} &\Big[\1_{X[0,\delta]\subset Q^y_L} u^{\1,y}_{L}(t-\delta,X_\delta)\Big] \\&\geq   \dE_{\dQ^{x,y}_L} \Big[\1_{X_\delta \in Q^y_r} u^{\1,y}_{r}(t-\delta,X_\delta)\Big]  -\dE_{\dQ^{x,y}_L} \Big[ \1_{X[0,\delta]\not\subset Q^y_L}\1_{X_\delta \in Q^y_r} u^{\1,y}_{r}(t-\delta,X_\delta)\Big].
    \end{aligned} 
\end{equation} The second term on the r.h.s. can be bounded as
    \begin{align}\label{eq:second term in the lower bound of the solution}
     \dE_{\dQ_L^{x,y}} \left[ \1_{X[0,\delta]\not\subset Q_L}\1_{X_\delta \in Q^y_r} u^{\1,y}_{r}(t-\delta,X_\delta)\right] &\leq \dP(X[0,\delta] \not\subset Q_L) \sup_{z\in Q^y_r} u^{\1,y}_{r}(t-\delta,z) \nonumber\\
     &\leq C\exp\Big( C\fa_0\delta(\log L)^2 - \frac{L^2}{C\delta} +C\fa_0^{\aleph+1}(t-\delta)(\log r)^{2\aleph+2} \Big) \nonumber\\
     &\leq C \exp \Big(C\fa_0^{\aleph+1} t(\log L)^{2\aleph+2} - \frac{L^2}{C\delta}\Big) ,
    \end{align}
 where we used \eqref{eq:escape probability on box} and \eqref{eq:upper and lower bound of sD} in the second inequality. Now we bound the first term in the r.h.s. of \eqref{eq:lower bound of the solution into two boxes}. Using \eqref{eq:upper and lower bound of transition kernel on box}, we have 
\begin{equation}
\begin{aligned}
\dE_{\dQ^{x,y}_L} \left[\1_{X_\delta \in Q^y_r} u^{\1,y}_{r}(t-\delta,X_\delta)\right]& = \int_{Q^y_r} \Gamma^L_\delta(x,z)u^{\1,y}_{r}(t-\delta,z) dz \\
 &\geq \frac{C}{t^{d/2}} e^{-C\fa_0(\log L^2) - C\frac{\fa_0(\log L)^2r^2}{\delta}} \int_{Q^y_r} u^{\1,y}_{r}(t-\delta,z)dz.
\end{aligned}
\end{equation} Note that 
\begin{equation}
    u^{\1,y}_{r}(t-\delta,z) \geq C^{-1}e^{-C\fa_0^{\aleph+1}\delta (\log r)^{2\aleph+2}} u^{\delta_x,y}_r(t-\delta,z)
\end{equation} from \eqref{eq:tower property}. This gives 
\begin{equation*}
    \dE_{\dQ^{x,y}_L} \left[\1_{X_\delta \in Q^y_r} u^{\1,y}_{r}(t-\delta,X_\delta)\right] \geq \frac{C}{t^{d/2}} e^{-C\fa_0(\log L^2) - C\frac{\fa_0(\log L)^2r^2}{\delta}-C\fa_0^{\aleph+1}\delta (\log r)^{2\aleph+2}} \int_{Q^y_r}  u^{\delta_x,y}_r(t-\delta,z) dz 
\end{equation*}

We pose to observe that using the spectral decomposition of the solution (Lemma~\ref{lemma:spectral representation}).
\begin{equation*}
   u^{\delta_z,y}_{L}(t,x) = \sum_{n \in \dN} e^{t \blambda_{n} (Q^y_L)}  v^y_{n,L}(z) v^y_{n,L}(x), \quad \text{for }x,z\in Q^y_L.
\end{equation*}
Using this, we can derive 
  $\int_{Q^y_L} u^{\delta_z,y}_{L}(t,z)dz \geq e^{t \blambda_{n}(Q^y_L) }.$
Combining these with the above lower bound, we obtain 
\begin{equation}\label{eq:first term in the lower bound of the solution}
    \dE_{\dQ^{x,y}_L} \left[\1_{X_\delta \in Q^y_r} u^{\1,y}_{r}(t-\delta,X_\delta)\right] \geq \frac{C}{t^{d/2}} \exp\Big( -C\fa^{\aleph+1}_0(\log L)^{2\aleph+2}) - \frac{C\fa_0(\log L)^2r^2}{\delta} + (t-\delta)\blambda_1(Q^y_r) \Big) .
\end{equation} Putting \eqref{eq:second term in the lower bound of the solution} and \eqref{eq:first term in the lower bound of the solution} together, we have \eqref{eq:lower bound of the local proxy}. This completes the proof.
\end{proof}

The following lemma says that the solution $u$ can be represented as the sum of the localized solution on $A^y_k$ where $A^y_k:= \left\{ X[0,t] \subset ( Q^y_{L_t^{k+1}} \setminus Q^y_{L_t^{k}}) \right\}$. This helps us to employ Proposition~\ref{prop:upper and lower bound of the localized solution} for the solution $u$ on $\dR^d$.

\begin{lemma}\label{lemma:convergence of series to the solution} Let  $L_t:= \lfloor t^b \rfloor$ for $b>1$.   With probability one, for all $x\in Q^y_{L_t}$ and $y\in \dR^d$, we have for $\epsilon\in[0,1]$,
\begin{equation*}
  u^{\1,y}_\epsilon(t,x) = \sum_{k \in \dN_0} \cU^y_{k,\epsilon}(t,x),
\end{equation*} where  
$\cU^y_{k,\epsilon}(t,x) : =\dE_{\dQ^{x,y}_{L_t^{k+1},\epsilon} } \left[\sD^{y}_{L_t^{k+1},\epsilon}(0,t) \1_{A^y_k} \right]. $
\end{lemma}
\begin{proof} For simplicity, we prove this result for $x=y=0$. Let us denote $u^\epsilon_t:=u^{\1,0}_\epsilon(t,0)$.  The general case follows easily from the stationarity of the solution $u$. We let $\cU_{k,\epsilon}(t) : = \cU^0_{k,\epsilon}(t,0)$ for $\epsilon\in[0,1]$. When $\epsilon\in(0,1]$, the lemma is proved from the classical Feynmann-Kac representation in the proof of Theorem~\ref{thm:Feynman-Kac representation}. To deal with the case of $\epsilon=0$, we first estimate $\cU_{k,\epsilon}(t)$. Recall that $\cU_{k,\epsilon}(t)$ is equal to  
$$\dE_{\dQ^{0,0}_{L^{k+1}_t,\epsilon}} \left[ \1_{ X[0,t] \not\subset Q^{0}_{L^{k}_t}, X[0,t] \subset Q^{0}_{L^{k+1}_t}} \cdot \sD^{0}_{L^{k+1}_t,\epsilon}(0,t)\right].$$
Using \eqref{eq:upper bound of localized solution} to bound the above display yields   
\begin{equation}\label{eq:upper bound of cU_k,epsilon}
  \begin{aligned}
    \cU_{k,\epsilon}(t) &\leq  C \exp \Big( Ct \fa_\epsilon^{\aleph+1} ((k+1)\log L_t)^{2\aleph+2} -\frac{L_t^{2k} }{Ct}  \Big)\\
    &\leq C \exp \Big( Ct \fa_\epsilon^{\aleph+1}( b(k+1)\log t)^{\aleph+2} -\frac{t^{2bk} }{Ct}  \Big)
  \end{aligned}
\end{equation} 
where the last inequality is obtained by substituting $L_t= \lfloor t^b\rfloor$. For a small $\delta>0$, define the following event:  
\begin{align*}
    \Upsilon_{\epsilon}: =\Big\{\fa_{\epsilon}  \leq C_b t^{\frac{2bk-2-\delta_0}{\upsilon+1}} \Big\}.
\end{align*}
Under the event $\Upsilon_{\epsilon}$, we have that for all $t\geq t_0$ and all $k\geq 1 $,
\begin{equation}\label{eq:upper bound of local proxy when fa bounded}
   \cU_{k,\epsilon}(t) \leq C\exp \Big( -\frac{t^{2bk-1}}{2C} \Big).
\end{equation} 
By the union bound of the probability to obtain that for all $K\in \dN$ and for all $\epsilon\in(0,1]$
  \begin{equation}\label{eq:UnionBd}
    \begin{aligned}
      \dP\Big( u^0_t - \sum^{K}_{k=0} \cU_k(t) > \delta \Big) \leq & \dP \Big(|u^0_t -u^\epsilon_t|> \frac{\delta}{3} \Big)+ \dP \Big( | u^\epsilon_t - \sum_{k=0}^{K} \cU_{k,\epsilon}(t)  |>\frac{\delta}{3}  \Big)\\
       &+ \sum_{k=0}^K \dP\Big( | \cU_{k,\epsilon}(t) - \cU_k(t)| > \frac{\delta }{3K} \Big)
    \end{aligned}
  \end{equation}
  By \cite[Theorem~1.1]{hairer2018multiplicative}, we know $u^{\epsilon}_t$ converges in probability to $u_t$ as $\epsilon\to 0$. This shows the first term in the right hand side of \eqref{eq:UnionBd} converges to $0$ as $\epsilon\to 0$.  
  Theorem~\ref{thm:Feynman-Kac representation} shows that the convergence of $u^{\epsilon}$ in every box $Q_{L_t}$. As a result, the third term in the right hand side of the above display also converges to $0$.  
   Therefore, we need to show that the second term goes to zero taking $K$ large. Note that $u^\epsilon_t = \sum_{k=0}^\infty \cU_{k,\epsilon}(t)$ by the classical Feynman-Kac representation. Then, we have \begin{equation*}
      \begin{aligned}
       \dP &\Big( | u^\epsilon_t- \sum_{k=0}^{K} \cU_{k,\epsilon}(t)  |>\frac{\delta}{3}  \Big)\leq \dP \Big(\Big\{ \Big|\sum_{k=K+1}^\infty \cU_{k,\epsilon}(t)\Big| > \frac{\delta}{3} \Big\}\cap \Upsilon_{\epsilon} \Big) + \dP(\neg \Upsilon_{\epsilon}).
      \end{aligned}
  \end{equation*} 
  Due to \eqref{eq:upper bound of local proxy when fa bounded}, $\sum_{k=K+1}^\infty \cU_{k,\epsilon}(t)$ is bounded above by $\exp(-t^{2b(K-1)}/C)$ for all $K>1$ on the event $\Upsilon_{\epsilon}$. This shows that the first term of the right hand side of the above display converges to $0$ as $K$ approaches to $\infty$.
   We now seek to bound $\dP(\neg \Upsilon_{\epsilon})$. By Markov's inequality, we have 
\begin{equation*}
  \dP(\Upsilon_{\epsilon}) = \dP \big( \fa_{\epsilon}  > C_b t^{\frac{2bK-2-\delta_0}{\aleph+1}}  \big) \leq \dE [e^{h_0 \sqrt{\fa_\epsilon}}] \cdot \exp\Big(- h_0\sqrt{C_b} t^{\frac{2bK-2-\delta_0}{2(\aleph+1)}} \Big).
\end{equation*} 
Recall that $\dE [e^{h_0 \sqrt{\fa_\epsilon}}]$ is uniformly bounded as shown in Proposition~\ref{prop:bounds on the enhance noise}. Letting $K\to \infty$ sends the right hand side of the above display to $0$.
This shows that the middle term of the right hand side of \eqref{eq:UnionBd} also converges to $0$ as $\epsilon\to 0$ and $K\to \infty$. As a result, we get  $u(t,0) = \sum_{k=0}^\infty \cU_k(t)$ in probability. Since each term of the series is non-negative, the latter identity also holds in the almost sure sense. This completes the proof.  
\end{proof}


\begin{lemma}\label{lem:large time asymptotics of the local proxy} Let $L_t: = t^b$ for some $b\in (\frac{1}{2},1]$. With probability one, for all $y\in \dR^d$ with $d=2,3$, 
\begin{equation}\label{eq:large time asymptotics of the local proxy}
\lim_{t\rightarrow \infty} \sup_{x\in B(y,1)} \frac{\log u^{\1,y}_{L_t}(t,x)}{t\blambda_1(Q^y_{L_t})}   = 1.    
\end{equation}
\end{lemma}
\begin{proof}
We only prove the lemma when $d=3$. The $d=2$ case follows from \cite[Lemma~3.6]{KPZ20}. By \eqref{eq:upper bound of the local proxy}, we have 
\begin{equation*}
    \sup_{x\in B(y,1)} \frac{\log u^{\1,y}_{L_t}(t,x)}{t\blambda_1(Q^y_{L_t})}  \leq 1 + \frac{ C\fa_0^{\aleph+1}(\log L_t)^{2\aleph+2}}{t \blambda_1(Q^y_{L_t})}.
\end{equation*} By Section~\ref{sec:AndersonHamiltonian}, we know that for enough large $t>0$, $\blambda_1(Q^y_{L_t}) >0$ almost surely. Since $\fa_0$ is almost surely finite, we have the upper bound
\begin{equation}
    \limsup_{t\rightarrow \infty} \sup_{x\in B(y,1)} \frac{\log u^{\1,y}_{L_t}(t,x)}{t\blambda_1(Q^y_{L_t})}   \leq 1.    
\end{equation} Let $b \in (\frac{1}{2} ,1], b_1 \in (0, b), $ and $b_2\in (2b_1-1, 2b-1)$. For the lower bound, by \eqref{eq:lower bound of the local proxy} with $L_t : = t^{b}, r_t = t^{b_1},$ and $\delta_t: = t^{b_2} $,  we have 
\begin{equation*}
   \sup_{x\in B(y,1)} u^{\1,y}_{L_t}(t,x) \geq \log A_t + \log (1- B_t) 
\end{equation*} where 
\begin{equation*}
    \begin{aligned}
     A_t  := \text{const}\cdot \exp \Big( (t-\delta_t) \blambda_1(Q^y_{r_t}) - \frac{\fa_0r^2(\log L_t)^2}{C \delta_t} - C\fa_0^{\upsilon+1}\delta_t(\log L_t)^{2\upsilon +2 }\Big)
    \end{aligned} 
\end{equation*} and 
\begin{equation*}
    B_t : =  \text{const}\cdot \exp \Big( -(t-\delta_t) \blambda_1(Q^y_{r_t}) + \frac{\fa_0r_t^2(\log L_t)^2}{C \delta_t} - \frac{L_t^2}{C\delta_t} - C\fa_0^{\upsilon+1}\delta_t(\log L_t)^{2\upsilon +2 } +C\fa_0^{\upsilon+1}t(\log L_t)^{2\upsilon+2}\Big).
\end{equation*} Since $2b -b_2 > 1$, $b> b_1$ and $\blambda(Q^y_{r_t})>0$ for all large $t,$ we have $B_t \rightarrow 0$. Furthermore, since $2b_1-b_2 <1 $ and $b_2<2b-1\leq 1$, we have 
\begin{equation*}
    \liminf_{t\rightarrow\infty}\log A_t \geq  \liminf_{t\rightarrow\infty} \frac{\blambda_1(Q^y_{r_t})}{\blambda_1(Q^y_{L_t})}. 
\end{equation*} Note that $\lim_{L\rightarrow \infty} \frac{ \blambda_1(Q^y_L)}{(\log L)^{2}}  = \chi$ for some constant $\chi>0$ (see \cite[Theorem~1]{HL22}). This shows that 
\begin{equation*}
    \liminf_{t\rightarrow \infty} \sup_{x\in B(y,1)} \frac{\log u^{\1,y}_{L_t}(t,x)}{t \blambda_1(Q^y_{L_t})} \geq \liminf_{t\rightarrow\infty} \frac{\blambda_1(Q^y_{r_t})}{\blambda_1(Q^y_{L_t})} \geq \liminf_{t\rightarrow\infty} \frac{\chi(\log r_t)^2}{\chi(\log L_t)^2} \geq \left({\frac{b_1}{b}}\right)^2.
\end{equation*} Letting $b_1 \uparrow b$, we have the lower bound. This completes the proof.

\end{proof}

\begin{lemma}\label{lem:large time aysmptotics of the solution} Let $L_t: = t(\log t)^{2\aleph+2}$. With probability one, for all $y\in \dR^d$ with $d=2,3$, 
\begin{equation}\label{eq:large time approximation of the solution by the local proxy}  
    \limsup_{t\rightarrow \infty } \left|\sup_{x\in B(y,1)}\frac{\log u^{\1,y}(t,x)}{t\blambda_1(Q^y_{L_t})} - \sup_{x\in B(y,1)}\frac{\log u^{\1,y}_{L_t}(t,x)}{t\blambda_1(Q^y_{L_t})} \right| = 0,
    \end{equation} and 
\begin{equation}
    \lim_{t\rightarrow \infty} \sup_{x\in B(y,1)} \frac{\log u^{\1,y}(t,x)}{t\blambda_1(Q^y_{L_t})}   = 1.    
\end{equation}
\end{lemma}

\begin{proof} 
 As in the previous lemma, we only prove the case when $d=3$ since the case of $d=2$ follows from \cite[Proposition~4.5]{KPZ20}. Due to the fact $\lim_{L\rightarrow \infty} \frac{\blambda_1(Q^y_L)}{(\log L)^2} = \chi$, it suffices to show 
\begin{equation}\label{eq:log of solution is similar to log of local proxy}
    \limsup_{t\rightarrow \infty } \left|\sup_{x\in B(y,1)}\frac{\log u^{\1,y}(t,x)}{t(\log L_t)^2} - \sup_{x\in B(y,1)}\frac{\log u^{\1,y}_{L_t}(t,x)}{t(\log L_t)^2} \right| = 0.
    \end{equation}  Letting $\epsilon \to 0$ in  Lemma~\ref{lemma:convergence of series to the solution}, we have $u^{\1,y} = \sum_{k\in \dN_0}\cU^y_k $. By simple observation, it follows that  
    \begin{equation}\label{eq:sum of log is similar to max of log}
        \lim_{t\rightarrow \infty} \left| \sup_{x\in B(y,1)} \frac{\log \sum_{k\in \dN_0} \cU^y_k(t,x)  }{t(\log L_t)^2 }  - \max\left\{  \sup_{x\in B(y,1)} \frac{\log\cU^y_0(t,x)}{t(\log L_t)^2 } , \sup_{x\in B(y,1)} \frac{\log\sum_{k\geq 1}\cU^y_k(t,x)}{t(\log L_t)^2 } \right\} \right| =0.
    \end{equation}  Since $\cU^y_0 = u^{\1,y}_{L_t}$, applying Lemma~\ref{lem:large time asymptotics of the local proxy} yields that
    \begin{equation}\label{eq:positivity of the first log}
        \sup_{x\in B(y,1)} \frac{\log\cU^y_0(t,x)}{t(\log L_t)^2 } =  \sup_{x\in B(y,1)} \frac{\log\cU^y_0(t,x)}{t\blambda_1(Q^y_{L_t}) } \cdot \frac{\blambda_1(Q^y_{L_t})}{(\log L_t)^2} \geq \frac{\chi}{2}
    \end{equation} for all large $t.$ Moreover, from \eqref{eq:upper bound of cU_k,epsilon}, we have  
    \begin{equation*}
       \sup_{x\in B(y,1)} \cU^y_k(t,x) \leq C\exp \Big( Ct((k+1)\log L_t)^{2\aleph+2} \Big[ \fa_0^{\aleph+1} - \frac{L_t^{2k}}{C^2t^2 ((k+1)\log L_t)^{2\aleph+2}} \Big] \Big).
    \end{equation*} Since $L_t = t(\log t)^{2\aleph+2}$, we have for all $k\geq 1 $ 
    \begin{equation}
        \sup_{x\in B(y,1)} \cU^y_k(t,x) \leq e^{-Ct (\log t)^{2\aleph+2} }
    \end{equation} for all large $t$. This implies that for all large $t$,
    \begin{equation}\label{eq:positivity of the second log}
        \sup_{x \in B(y,1)} \log \Big( \sum_{k\geq 1 }\cU^y_k(t,x) \Big)\leq 0.  
    \end{equation} By combining \eqref{eq:positivity of the first log} and \eqref{eq:positivity of the second log} with \eqref{eq:sum of log is similar to max of log}, we can conclude \eqref{eq:log of solution is similar to log of local proxy}, which is the first part of the lemma. The second part of the lemma follows immediately from the first part and Lemma~\ref{lem:large time asymptotics of the local proxy}. 
\end{proof}

\section{Spatial Multifractality and Asymptotics of the PAM: Proof of Theorem~\ref{thm:spatial multifractality}}\label{sec:proof of spatial multifractality in 2d}

In the remaining sections, we will show the multifractality of the solution to \eqref{eq:PAM}. We only consider the solution with flat initial data and write $u$ for $ u^{\1}$.

\subsection{Proof of the lower bound in Theorem~\ref{thm:spatial multifractality} }\label{subsec:lower bound in thm1}
In this section, we prove the lower bound of the dimension in Theorem~\ref{thm:spatial multifractality}. The following proposition is one of the key tools for proving such lower bound.

\begin{proposition}\label{prop:left tail probability of spatial maximum}
 Let $\epsilon>0$ and $\theta>0$. There exists $n_0>0$ such that for all $n\geq n_0$ and $x_1,...,x_m \in (e^n,e^{n+1}]^d$ satisfying $\min_{i\neq j} \|x_i-x_j\|_\infty > e^{n\theta},$ we have 
\begin{equation}
 \P \left( \max_{1\leq j \leq m } \frac{\log u(t,x_j)}{(\log \|x_j\|_\infty)^{\frac{2}{4-d}}} \leq \alpha t \right) \leq \exp \left ( -c m(\alpha +\epsilon)^{\frac{d}{2}} n^{\frac{d}{4-d}} e^{d\log r_n - \fc_d (1+\epsilon)(\alpha+\epsilon)^{\frac{4-d}{2}}n}\right) + e^{ -c m\log n },
\end{equation} where $r_n : =n^{\frac{1}{2}\log t}.$
\end{proposition}
\begin{proof} Since $x_1,\ldots x_m\in (e^n, e^{n+1}]^d$, we have 
  \begin{equation}\label{eq:left tail of solution using independence}
    \begin{aligned}
      \P \Big( \max_{1\leq j \leq m } \frac{\log u(t,x_j)}{(\log \|x_j\|_{\infty})^{\frac{2}{4-d}}} \leq \alpha t \Big)\leq \P \Big( \max_{1\leq j \leq m } u(t,x_j) \leq e^{\alpha t (n+1)^{\frac{2}{4-d}}} \Big).
    \end{aligned} 
  \end{equation} Let $L:= L_n:= t^{\log n}$. By Proposition~\ref{lemma:convergence of series to the solution}, there exists $t_0>0$ such that for all $t\geq t_0$ the series $\sum_{k\in \dN_0}\cU^y_k=u$ for all $y\in \dR^d$. Moreover, since $0\leq \cU^y_0 \leq \sum_{k\in \dN_0}\cU^y_k = u $, we have 
  \begin{equation}\label{eq:left tail of local proxy}
    \begin{aligned}
     \P \left( \max_{1\leq j \leq m } \frac{\log u(t,x_j)}{(\log \|x_j\|_\infty)^{\frac{2}{4-d}}} \leq \alpha t \right)&\leq \P \Big( \max_{1 \leq j\leq m} \cU^{x_j}_0(t,x_j) \leq e^{\alpha t (n+1)^{\frac{2}{4-d}}}\Big)\\
     &= \prod_{j=1}^m \P \Big( \cU^{x_j}_0(t,x_j) \leq e^{\alpha t (n+1)^{\frac{2}{4-d}}}\Big),
    \end{aligned}
  \end{equation} where the last equality follows from the independence of $\cU^{x_j}_0(t,x_j)$  shown in Lemma~\ref{lem:independence of eigenvalues} thanks to the fact that $\min_{i\neq j} \|x_i-x_j\|_\infty > e^{n\theta} \gg 3L_t = 3t^{\log n}.$  
We also set $r:=r_n:=t^{\frac{1}{2}\log n },$. Let $\epsilon>0$. Observe that for any fixed $t\geq t_0$, there exists $n_0>0$ such that for all $n\geq n_0$ 
\begin{equation}\label{eq:approximations}
    \begin{aligned}
     \epsilon t n \gg (\fa_0 \log L_n)^{2\aleph+2} + \frac{\fa_0(\log L_n)^2r_n^2}{C\delta}, \quad \frac{L_n^2}{C\delta} \gg  t(\fa_0 \log L_n)^{2\aleph+2}+ \frac{\fa_0(\log L_n)^2r_n^2}{C\delta}
    \end{aligned}
\end{equation} on the event $
    \Upsilon_n : = \{ \fa_0 \leq (\log n)^2 \}. $
 Using this observation, when $d=3$, there exist $t_0>0$ and $n'_0>0$ such that for all $n\geq n'_0$ and all $t\geq t_0$
\begin{equation*}
  \begin{aligned}
   &\P \Big( \cU^{x_j}_0(t,x_j) \leq e^{\alpha t (n+1)^2}\Big) \\
   &\leq \P\Big( e^{ -C \fa_0^{\aleph+1}\delta (\log L_n)^{2\aleph +2} -\frac{C\fa_0(\log L_n)^2r_n^2}{\delta } +(t-\delta)\blambda_1(Q^y_{r_n})} -e^{C\fa_0^{2\aleph+1}t(\log L_n)^{2\aleph+2}-\frac{L_n^2}{C\delta }}\leq e^{\alpha t (n+1)^2}\Big)\\
      &\leq  \P\Big(  \Big\{\frac{1}{2}\exp \Big( (t-\delta)\blambda_1(Q^{x_j}_{r_n}) -C \fa_0^{\aleph+1}\delta (\log L_n)^{2\aleph +2}- \frac{C\fa_0(\log L_n)^2r^2}{\delta } \Big) \leq e^{\alpha t (n+1)^2} \Big\} \cap \Upsilon_n \Big) + \dP(\neg \Upsilon_n ) \\
   &\leq  \P \left( \blambda_1(Q^{x_j}_{r_n})  \leq (\alpha +\epsilon)n^2  \right)+ \dP(\neg \Upsilon_n ) 
  \end{aligned} 
\end{equation*} The first inequality in the above display is obtained by using the lower bound on $\cU^{x_j}_0(t,x_j) $ from  Proposition~\ref{prop:upper and lower bound of the localized solution}. While the second inequality is just consequence of the union bound, the last inequality utilizes \eqref{eq:approximations}. When $d=2$, we use the similar argument except the fact that the lower bound on $\cU_0^{x_j}(t,x_j)$ is now provided by  Lemma~5.2 of \cite{KPZ20}. We obtain similarly
\begin{equation*}
  \begin{aligned}
    &\dP ( \cU_0^{x_j}(t,x_j) \leq e^{\alpha t (n+1)} ) \\
&\leq \dP \Big( \exp \Big( (t-\delta)\blambda_1(Q^{x^j}_{r_n})  - \frac{r_n^2}{C\delta}  - C\delta(\fa_0\log L_n)^5) -\exp(Ct(\fa_0\log L_n)^5 -\frac{L_n^2}{C\delta}\Big) \leq e^{\alpha t (n+1)}\Big)\\
&\leq \dP \Big( \Big\{ \frac{1}{2}\exp \Big( (t-\delta)\blambda_1(Q^{x^j}_{r_n})  - \frac{r_n^2}{C\delta}- C\delta(\fa_0\log L_n)^5\Big) \leq e^{\alpha t (n+1)} \Big\} \cap \Upsilon_n\Big) + \dP(\neg \Upsilon_n)\\
&\leq \dP ( \blambda_1(Q^{x^j}_{r_n})  \leq (\alpha +\epsilon)n)+\dP(\neg \Upsilon_n)
  \end{aligned} 
\end{equation*} 

Here we note that $t_0$ can be chosen independently of $\alpha$. Now we use Lemma~\ref{lemma:tail probability of eigenvalue} to obtain 
\begin{equation*}
  \max_{1\leq j \leq m} \P \left( \blambda_1(Q^{x_j}_{r_n}) \leq (\alpha +\epsilon)n^{\frac{2}{4-d}}  \right) \leq\exp \left ( -c_2 (\alpha +\epsilon)^{\frac{d}{2}} n^{\frac{d}{4-d}} e^{d\log r_n - \fc_d (1+\epsilon)(\alpha+\epsilon)^{\frac{4-d}{2}}n} \right)
\end{equation*} Substituting this into \eqref{eq:left tail of local proxy} and \eqref{eq:left tail of solution using independence}, we have 
\begin{equation*}
\begin{aligned}
  \P &\left( \max_{1\leq j \leq m } \frac{\log u(t,x_j)}{(\log \|x_j\|_\infty)^{\frac{2}{4-d}}} \leq \alpha t \right)
  \\ &\leq 2^{m-1} \exp \left ( -c_2 m(\alpha +\epsilon)^{\frac{d}{2}} n^{\frac{d}{4-d}} e^{d\log r_n - \fc_d (1+\epsilon)(\alpha+\epsilon)^{\frac{4-d}{2}}n} \right) +  2^{m-1} \dP(\neg \Upsilon_n )^m  
\end{aligned}
\end{equation*} On the other hand, since $\dE[e^{h_0 \sqrt{\fa_0}}] < \infty$ (see \ref{eq:definition of fa_epsilon}), we use Markov's inequality to get $\dP(\neg \Upsilon_n ) \leq e^{-h_0\log n }$. This shows that 
\begin{equation*}
     \P \left( \max_{1\leq j \leq m } \frac{\log u(t,x_j)}{(\log \|x_j\|_\infty)^{\frac{2}{4-d}}} \leq \alpha t \right) \leq \exp \left ( -c m(\alpha +\epsilon)^{\frac{d}{2}} n^{\frac{d}{4-d}} e^{d\log r_n - \fc_d (1+\epsilon)(\alpha+\epsilon)^{\frac{4-d}{2}}n} \right) + e^{ -c m\log n }
\end{equation*} for some constant $c>0$ and hence, completes the proof.
\end{proof}
Now we proceed to prove Theorem~\ref{thm:spatial multifractality}. Recall that 
$$\cP^d_t(\alpha) =  \left\{ x\in \R^d \, : \, u(t,x) \geq e^{\alpha t(\log|x|)^{\frac{2}{4-d}}} \right\}.$$
The following result proves a lower bound to the macroscopic Hausdorff dimension of the set $\cP^d_t(\alpha)$. The proof of the upper bound is deferred to the subsection after the following result.
\begin{theorem}\label{thm:lower bound of dimension for spatial peaks} There exists a non-random finite constant $t_0>0$ such that for all $t\geq t_0$, 
\begin{equation*}
   \Dim[\cP^d_t(\alpha)] \geq d-\alpha^{\frac{4-d}{2}} \fc_d, \quad \text{a.s.}
\end{equation*}
\end{theorem}

\begin{proof} 
  Let us choose $\alpha>0$ satisfying $\alpha^{\frac{4-d}{2}} \fc_d <d$. Fix $\epsilon>0$ such that $\fc_d(1+\epsilon)(\alpha+\epsilon)^{\frac{4-d}{2}} <d$. Define
  \begin{equation*}
    \widetilde{\cP}^d_t(\alpha) := \cP^d_t(\alpha) \cap \bigcup_{n=0}^\infty \left( e^n, e^{n+1}\right]^d.
  \end{equation*} Then it suffices to show that 
  \begin{equation*}
      \Dim[\widetilde{\cP}^d_t(\alpha)]\geq d-\alpha^{\frac{4-d}{2}} \fc_d,
     \end{equation*} with probability one. We now choose $\gamma \in (\fc_d(1+\epsilon)(\alpha+\epsilon)^{\frac{4-d}{2}}/d, 1)$ and define, for all integers $n\geq0$,
     \begin{equation*}
       a_{j,n}(\gamma) : = e^n + je^{n\gamma}, \quad j \in [0, e^{n(1-\gamma)}) \cap \dZ,
     \end{equation*}
     and
     \begin{equation*}
       I_n(\gamma) : = \bigcup_{ j \in [0, e^{n(1-\gamma)}) \cap \dZ} \{ a_{j,n}(\gamma)\}, \qquad \cI_n(\gamma) : =  \prod_{k=1}^d I^k_n(\gamma), 
     \end{equation*}  where $I^k_n(\gamma)$ is a copy of $I_n(\gamma)$ for all $1\leq j\leq d$. We choose $x \in \cI_n(\gamma)$ and $\theta\in (0, \gamma -\frac{\fc_d(1+\epsilon)(\alpha+\epsilon)^{\frac{4-d}{2}}}{d})$. By the construction of the set $\mathcal{I}_n(\gamma)$, we first find the points $\{x_i\}_{i=1}^{m(n)}$ satisfying the followings: 
     (a.1) $x_i\in \cI_n(\gamma) \cap B(x,e^{n\gamma})$ for all $i=1,...,m(n)$;
    (a.2) $\|x_i-x_j\|_{\infty} \geq e^{n\theta}$ whenever $1\leq i<j\leq m(n)$; (a.3) $d^{-1}e^{dn(\gamma-\theta)}\leq m(n) \leq de^{dn(\gamma-\theta)}$. 
     Then by Proposition~\ref{prop:left tail probability of spatial maximum}, we have 
    \begin{equation*}
    \begin{aligned}
      \P &\left( \max_{1\leq j \leq m(n) } \frac{\log u(t,x_j)}{(\log \|x_j\|_\infty)^{\frac{2}{4-d}}} \leq \alpha t \right) 
      \\&\leq\exp \left ( -c m(n)(\alpha +\epsilon)^{\frac{d}{2}} n^{\frac{d}{4-d}} e^{d\log r_n - \fc_d (1+\epsilon)(\alpha+\epsilon)^{\frac{4-d}{2}}n} \right) + e^{ -c m(n)\log n }\\
      &\leq \exp \left ( -\frac{c}{2}(\alpha +\epsilon)^{\frac{d}{2}} n^{\frac{d}{4-d}}  e^{c\log n+ [d(\gamma-\theta) - \fc_d (1+\epsilon)(\alpha+\epsilon)^{\frac{4-d}{2}}]n} \right) + e^{ -\frac{c}{2}e^{dn(\gamma-\theta)} \log n }.
    \end{aligned}
    \end{equation*} By our choice of $\gamma $ and $\theta$, $\kappa : = d(\gamma-\theta) - \fc_d (1+\epsilon)(\alpha+\epsilon)^{\frac{4-d}{2}}>0$. Therefore, 

  \begin{equation*}
    \begin{aligned}
      \P \left(\max_{\{x_i\}_{i=1}^{m(n)} \subseteq \cI_n(\theta) \cap B(x,e^{n\gamma})} \frac{\log u(t,x_i)}{(\log\|x_i\|_\infty)^2} \leq \alpha t  \right)
      \leq \exp\left( -C_1 e^{ \kappa n -C_2\log n } \right)+ e^{ -\frac{c}{2}e^{dn(\gamma-\theta)} \log n },
    \end{aligned}
  \end{equation*} for some constant $C_1,C_2>0$. Then we have
\begin{equation*}
  \begin{aligned}
    \sum_{n=0}^\infty \P &\left( \min_{x\in \cI_n(\gamma)} \max_{\{x_i\}_{i=1}^{m(n)} \subseteq \cI_n(\theta) \cap B(x,e^{n\gamma})} \frac{\log u(t,x_i)}{(\log\|x_i\|_\infty)^2} \leq \alpha t   \right)\\&\leq \sum_{n=0}^\infty \sum_{x\in \cI_n(\gamma) } \P \left( \max_{1\leq i\leq m(n)}\frac{\log u(t,x_i)}{(\log\|x_i\|_\infty)^2} \leq \alpha t   \right)\\
    &\leq \sum_{n=0}^\infty C e^{2n(1-\gamma)} \left( \exp\left( -C_1 e^{ \kappa n -C_2\log n } \right)+ e^{ -\frac{c}{2}e^{3n(\gamma-\theta)} \log n }\right)<\infty.
  \end{aligned}
\end{equation*} for some constant $C>0$. Hence, the Borel-Cantelli lemma implies that $\cP^d$ contains a $\gamma -$thick set (see Definition~\ref{def:ThickSet}) almost surely. By Proposition~\ref{prop:thick set}, We get $\Dim(\widetilde{\cP}^d_t(\alpha)) \geq d(1-\gamma)$ with probability one. Letting $\gamma \downarrow \frac{\fc_d(1+\epsilon)(\alpha+\epsilon)^{\frac{4-d}{2}}}{d}$ and $\theta \downarrow 0 $ with  $0 < \theta < \gamma- \frac{\fc_d(1+\epsilon)(\alpha+\epsilon)^{\frac{4-d}{2}}}{d}$, we get $\Dim(\widetilde{\cP}^d_t(\alpha)) \geq d-\fc_d(1+\epsilon)(\alpha+\epsilon)^{\frac{4-d}{2}}.$ Since $\epsilon>0$ is arbitrary, by the monotonicity of the macroscopic Hausdorff dimension, we conclude that 
\begin{equation*}
  \Dim[\cP^d_t(\alpha)] \geq \Dim [\widetilde{\cP}^d_t(\alpha)] \geq d-\alpha^{\frac{4-d}{2}} \fc_d,
\end{equation*} almost surely.

\end{proof}

\subsection{Proof of the upper bound in Theorem~\ref{thm:spatial multifractality}}\label{subsec:upper bound in thm1}

In this section, we prove the upper bound in Theorem~\ref{thm:spatial multifractality}. The following proposition will be used to complete the proof. 

\begin{proposition}\label{prop:right tail probability of spatial maximum}
    Let $\epsilon>0$ and $M>0$. There exist $b:=b(M)>1 $ and $n_0:= n_0(M,b,\epsilon)>0$ such that for all $n\geq n_0$ and $t\geq 1$ 
    \begin{equation}\label{eq:upper bound of maximum of spatial tall peaks when a_0 bounded}
    \dP\Big( \sup_{x \in \B(y,1)}  \frac{\log u(t,x)}{ (\log\|x\|_\infty)^{\frac{2}{4-d}}} \geq \alpha t  , \fa_0 \leq M \Big) \leq c_1 (\alpha-\epsilon)^{\frac{d}{2}}n^{\frac{d}{4-d}} e^{d \log L_t - (1-\epsilon)\fc_d (\alpha-\epsilon)^{\frac{4-d}{2}}n },
\end{equation} where $L_t : = t^b$ and  $y\in \dS_n$ for $n \in \dN$. 
\end{proposition}

\begin{proof}
    Since we can write $u(t,x) = \sum_{k=0}^\infty \cU_k^y (t,x)$ due to Lemma~\ref{lemma:convergence of series to the solution}, we have 
    \begin{equation*}
       \dP\Big( \sup_{x \in \B(y,1)}  \frac{\log u(t,x)}{ (\log\|x\|_\infty)^{\frac{2}{4-d}}} \geq \alpha t  , \fa_0 \leq M \Big) \leq   \dP\Big( \sup_{x\in B(y,1)} u(t,x) \geq \frac{1}{2}e^{\alpha t n^{\frac{2}{4-d}}}  , \fa_0 \leq  M \Big) \leq ({\bf C_1}) + ({\bf C_1}),
    \end{equation*} where 
    \begin{equation*}\begin{aligned}
         ({\bf C_1}) &: = \dP\Big( \sup_{x\in B(y,1)} \cU_0^y (t,x) \geq \frac{1}{2}e^{\alpha t n^{\frac{2}{4-d}}}, \fa_0 \leq  M \Big),\\  ({\bf C_2})&: =  \dP\Big( \sup_{x\in B(y,1)} \sum_{k=1}^\infty \cU_k^y (t,x) \geq \frac{1}{2}e^{\alpha t n^{\frac{2}{4-d}}}, \fa_0 \leq  M \Big)
    \end{aligned}
    \end{equation*}We first bound $({\bf C_2})$. When $d=3$, using \eqref{eq:upper bound of cU_k,epsilon}, on the event $\{ \fa_0 \leq M\}$ we have 
    \begin{equation*}
        \cU^y_k(t,x) \leq C \exp \left( Ct M^{\aleph+1} ( b ( k+1) \log t )^{2\aleph+2} - \frac{t^{2bk}}{Ct} \right). 
    \end{equation*} Therefore,  we can choose a large $b:=b(M)>0$ such that for all $t\geq e$ and $n\geq 1$,
    \begin{equation*}
       \sum_{k=1}^\infty \cU_k^y(t,x) \leq\sum_{k=1}^\infty C \exp \left ( - C_1 t^{2bk -1 } \right) \leq \frac{1}{2}e^{\alpha t n^{\frac{2}{4-d}}} . 
    \end{equation*} for some constant $C_1>0$. This implies that $({\bf C_2}) = 0 $ for all $t\geq e$. For $d=2$ case, one can use Lemma~5.2 of \cite{KPZ20} to get the same result. 

     Now we proceed to bound $({\bf C_1})$. Applying  Proposition~\ref{prop:upper and lower bound of the localized solution} for $d=3$, there exists $n_0 = n_0(M,b,\epsilon)>0$ such that for all $n \geq n_0$  
    \begin{equation*}
    \begin{aligned}
         ({\bf C_1})  &\leq \dP \Big ( C\exp ( t \blambda_1(Q^y_{L_t}) + CM^{\upsilon+1} (\log L_t)^{2\upsilon +2 }) \geq  \frac{1}{2}e^{\alpha t n^{\frac{2}{4-d}}}   \Big)\\
         &\leq \dP\Big( \blambda_1(Q^y_{L_t}) \geq \alpha n^{\frac{2}{4-d}} -\frac{\log (2C) + CM^{\upsilon+1} (b\log t)^{2\upsilon +2 }}{t}  \Big)\\
         &\leq \dP ( \blambda_1(Q^y_{L_t}) \geq (\alpha-\epsilon) n^{\frac{2}{4-d}}),
    \end{aligned}
    \end{equation*}  The $d=2$ case follows similarly from using Lemma~5.2 of \cite{KPZ20}. By Lemma~\ref{lemma:tail probability of eigenvalue}, we further have 
    \begin{equation*}
        ({\bf C_1})\leq  \dP ( \blambda_1(Q^y_{L_t}) \geq (\alpha-\epsilon) n^{\frac{2}{4-d}}) \leq c_1 (\alpha-\epsilon)^{\frac{d}{2}}n^{\frac{d}{4-d}} e^{d b\log t - (1-\epsilon)\fc_d (\alpha-\epsilon)^{\frac{4-d}{2}}n }.
    \end{equation*} 
Summing the bounds $({\bf C_1})$ and $({\bf C_2})$, we arrive at 
    \begin{equation*}
         ({\bf C_1})  +  ({\bf C_2}) \leq   c_1 (\alpha-\epsilon)^{\frac{d}{2}}n^{\frac{d}{4-d}} e^{d b \log t - (1-\epsilon)\fc_d (\alpha-\epsilon)^{\frac{4-d}{2}}n },
    \end{equation*} which completes the proof. 
    \end{proof} 

\begin{theorem}
      For all $t\geq e$ and all $\alpha \in (0, \left(d/\fc_d\right)^{\frac{2}{4-d}})$, 
  \begin{equation}\label{eq:upper bound of dimension for spatial peaks}
    \Dim[\cP^d_t(\alpha)]= \Dim\Big[ \Big\{ x\in \R^d \, : \, u(t,x) \geq e^{\alpha t(\log\|x\|_\infty)^{\frac{2}{4-d}}} \Big\}\Big]  \leq d- \alpha^{\frac{4-d}{2}} \fc_d,
  \end{equation} with probability one. 
\end{theorem}

\begin{proof} Fix $M>0$. By Markov's inequality and Proposition~\ref{prop:bounds on the enhance noise}, we have 
\begin{equation}\label{eq:probability of fa bounded}
    \dP (\fa_0 \geq M ) \leq  \delta_M : =\dE[e^{h_0 \sqrt{\fa_0}}]\cdot e^{-h_0 \sqrt{M}}.
\end{equation} Note that $\delta_M \rightarrow 0 $ as $M \rightarrow \infty$. Choose  $\epsilon\in(0, \alpha \wedge 1 ) $ and $ \rho \in ( d- (1-\epsilon) \fc_d (\alpha-\epsilon)^{\frac{4-d}{2}} , d)$. Using Proposition~\ref{prop:right tail probability of spatial maximum}, we have that for all $t\geq 1$
 \begin{equation}\label{eq:finiteness of Hausdorff content when fa bounded}
    \begin{aligned}
          \sum_{n=0}^\infty &e^{-n\rho} \sum_{\substack{y\in \dN^d \\ B(y,1) \subset \dS_n }} \dP\Big( \sup_{x\in B(y,1) } \frac{\log u (t,x) }{(\log \|x \|_\infty)^{\frac{2}{4-d}}} \geq \alpha t  , \fa_0 \leq M  \Big)\\ 
          &\leq \sum_{n=0}^\infty c_1 (\alpha-\epsilon)^{\frac{d}{2}}n^{\frac{d}{4-d}} \exp\Big(d b\log t + [d- \rho - \fc_d \alpha_\epsilon^{\frac{4-d}{2}} ] n \Big) < \infty,
    \end{aligned} 
    \end{equation} where $b=b(M)>0$ is taken as in Proposition~\ref{prop:right tail probability of spatial maximum} and $\alpha_\epsilon: = (1-\epsilon)^{\frac{2}{4-d}}(\alpha-\epsilon)$. Recall the definition of Hausdorff content $\nu^n_\rho(E)$ of any set $E$ from Section~\ref{subsec2.1}. The first line in \eqref{eq:finiteness of Hausdorff content when fa bounded} is an upper bound to $\mathbb{E}[\sum_{n=0}^\infty \nu^n_\rho (\cP^d_t (\alpha) ) \mathbbm{1}(\mathfrak{a}_0\leq M)]$. As a result, we have 
    \begin{equation*}
        \dP \Big(\{ \fa_0 \leq M \} \cap \Big\{\sum_{n=0}^\infty \nu^n_\rho (\cP^d_t (\alpha) ) < \infty  \Big\}  \Big) = \dP(\{ \fa_0 \leq M \}).
    \end{equation*} Combining \eqref{eq:probability of fa bounded} with the above display yields  
    \begin{equation*}
        \dP \Big( \sum_{n=0}^\infty \nu^n_\rho (\cP^d_t (\alpha) ) < \infty \Big) \geq 1 - \dE[e^{h_0 \sqrt{\fa_0}}]\cdot e^{-h_0 \sqrt{M}}.
        \end{equation*} Note that the l.h.s. of the above inequality does not depend on $M$. Thus, by letting $M\rightarrow \infty$, we have $\sum_{n=0}^\infty \nu^n_\rho (\cP^d_t (\alpha) ) < \infty  $ a.s., which implies that $\Dim [\cP^d_t (\alpha) ] \leq \rho $ for any $\rho \in (d- \fc_d \alpha_\epsilon^{\frac{4-d}{2}},d)$, a.s. By taking $\rho \downarrow d - \fc_d \alpha_\epsilon^{\frac{4-d}{2}}$ and $\epsilon\downarrow 0 $, we get $\Dim [\cP^d_t (\alpha) ] \leq d - \fc_d \alpha^{\frac{4-d}{2}}.$

\end{proof}

\subsection{Spatial asymptotics of the PAM}\label{subsec:spatial asymptotics}

This section is devoted to proving the spatial asymptotics  \eqref{eq:limsup of spatial tall peaks}. This result is a consequence of the first part of Theorem~\ref{thm:spatial multifractality}, i.e., there exists $t_0>0$ such that 
\begin{equation}\label{eq:dimension of spatial peaks for spatial asymptotics}
   \Dim[\cP^d_t(\alpha)] = \Dim \left( \left\{ x\in \R^d \, : \, \frac{\log u(t,x)}{\log|x|} \geq \alpha t  \right\}\right) = (d-\alpha^{\frac{4-d}{2}}\fc_d) \vee 0, \quad \text{a.s.,}
\end{equation} for $t\geq t_0$ and $\alpha>0$. We provide the details of the proof below.

\begin{proof}[\bf Proof of \eqref{eq:limsup of spatial tall peaks}] 
   By the definition of the macroscopic Hausdorff dimension, $\Dim(A)>0,A\subset\dR^d $ implies that $A$ is an unbounded set. This fact and \eqref{eq:dimension of spatial peaks for spatial asymptotics} imply that for any $\alpha < (d/\fc_d)^{\frac{2}{4-d}}$, the set $\cP^d_t(\alpha)$ is unbounded, hence in particular
\begin{equation}
  \limsup_{\|x\|_\infty\rightarrow \infty } \frac{\log u(t,x)}{(\log\|x\|_\infty)^{\frac{2}{4-d}}} \geq \alpha t, \quad \text{a.s.} 
\end{equation} This implies that 
\begin{equation}
  \limsup_{\|x\|_\infty\rightarrow \infty } \frac{\log u(t,x)}{(\log\|x\|_\infty)^{\frac{2}{4-d}}}  \geq \left(\frac{d }{\fc_d}\right)^{\frac{2}{4-d}}t, \quad \text{a.s.} 
\end{equation}
Now we prove the upper bound. Fix $\epsilon\in(0,\alpha\wedge1)$. Let $M>0$ and $\delta_M$ be defined as in \eqref{eq:probability of fa bounded}. Note that we can rewrite \eqref{eq:finiteness of Hausdorff content when fa bounded} as 
\begin{equation}
    \begin{aligned}
          \sum_{n=0}^\infty  \sum_{\substack{y\in \dN^d \\ B(y,1) \subset \dS_n }} \dP\Big( \sup_{x\in B(y,1) } \frac{\log u (t,x) }{(\log \|x \|_\infty)^{\frac{2}{4-d}}} \geq g(d,\epsilon) t  , \fa_0 \leq M  \Big)< \infty,
    \end{aligned} 
\end{equation} where 
$g(d,\epsilon): = \big( d/\fc_d(1-\epsilon)\big)^{\frac{2}{4-d}} +\epsilon.$ 
The Borel-Cantelli lemma yields that with probability greater than $1-e^{-h\sqrt{M}}\mathbb{E}[e^{h\sqrt{\mathfrak{a}_0}}]$,
\begin{equation*}
     \limsup_{\|x\|_\infty\rightarrow \infty } \frac{\log u(t,x)}{(\log\|x\|_\infty)^{\frac{2}{4-d}}} \leq \left(\frac{d }{\fc_d(1-\epsilon)} +\epsilon\right)^{\frac{2}{4-d}}t,
\end{equation*} for all $t\geq t_0.$ Letting $\epsilon\to 0$ and $M\to \infty$, we can conclude that for all $t\geq t_0$ 
\begin{equation*}
     \limsup_{\|x\|_\infty\rightarrow \infty } \frac{\log u(t,x)}{(\log\|x\|_\infty)^{\frac{2}{4-d}}} \leq\left(\frac{d }{\fc_d}\right)^{\frac{2}{4-d}}t
\end{equation*} with probability one. This completes the proof.
\end{proof}

\section{Spatio-temporal Multifractality: Proof of Theorem~\ref{thm:spatio-temporal multifractality}}\label{sec:proof of spatio-temporal multifractality}

\subsection{Proof of the lower bound in Theorem~\ref{thm:spatio-temporal multifractality}}\label{subsec:lower bound in thm2}
 
\begin{proposition}\label{prop:left tail probability of spatio-temporal maximum}
 Let $\epsilon>0$ and $\theta>0$. There exists $c,t_0>0$ such that for all $t\geq t_0$ and $x_1,...,x_m \in \dR^d$ satisfying $\min_{i\neq j} \|x_i-x_j\|_\infty > 3L_t$ where $L_t:=t$, we have 
\begin{equation}\label{eq:Spatio-temporalUpperBound}
 \P \left( \max_{1\leq j \leq m } \log u(t,x_j) \leq \beta t^{\frac{6-d}{4-d}} \right) \leq \exp \left ( -c m(\beta  +\epsilon)^{\frac{d}{2}} t^{\frac{d}{4-d}} e^{d\log r_t - \fc_d (1+\epsilon)(\beta+\epsilon)^{\frac{4-d}{2}}t} \right) + e^{ -c m\log t },
\end{equation} where $r_t : =t^{\frac{1}{2}}.$
\end{proposition}

\begin{proof} The proof is similar to the proof of Proposition~\ref{prop:left tail probability of spatial maximum}. By Lemma~\ref{lemma:convergence of series to the solution}, we have 
\begin{equation*}
    \begin{aligned}
      \P \left( \max_{1\leq j \leq m } \log u(t,x_j) \leq \beta t^{\frac{6-d}{4-d}} \right) \leq \dP \Big( \max_{1\leq j \leq m}\cU^{x_j}_0(t,x_j) \leq e^{\beta t^{\frac{6-d}{4-d}}}  \Big) = \prod_{j=1}^m \dP \Big( \cU^{x_j}_0(t,x_j) \leq e^{\beta t^{\frac{6-d}{4-d}}}  \Big),
    \end{aligned}
\end{equation*} whenever $\min_{i\neq j} \|x_i-x_j\|_\infty > 3L_t$ where $L_t : = t$. Let $r:=r_t:= t^{1/2}$. The last equality follows due to the independence  between $\{\cU^{x_j}_0(t,x_j)\}_{1\leq j\leq m}$. There exist $t_0>0$ such that for all $t\geq t_0$ 
\begin{equation*}
\epsilon t^2 \gg (\fa_0 \log L_t)^{2\aleph+2} + \frac{\fa_0(\log L_t)^2r_t^2}{C\delta} , \quad \frac{L_t^2}{C\delta} \gg  t(\fa_0 \log L_t)^{2\aleph+2}+ \frac{\fa_0(\log L_t)^2r_t^2}{C\delta}
\end{equation*} on the event 
    $\Upsilon_t: = \{ \fa_0 \leq (\log t)^2\}.$
For $d=3$, by Proposition~\ref{prop:upper and lower bound of the localized solution} there exists $t_0>0$ such that for all $t\geq t_0$,
\begin{equation*}
  \begin{aligned}
   \P &\Big( \cU^{x_j}_0(t,x_j) \leq e^{\beta t^3}\Big) \\
   &\leq \P\Big( e^{ -C \fa_0^{\aleph+1}\delta (\log L_t)^{2\aleph +2} -\frac{C\fa_0(\log L_t)^2r_t^2}{\delta } +(t-\delta)\blambda_1(Q^y_{r_t})} -e^{C\fa_0^{2\aleph+1}t(\log L_t)^{2\aleph+2}-\frac{L_t^2}{C\delta }}\leq e^{\beta t^3}\Big)\\
      &\leq  \P\Big(  \Big\{ \frac{1}{2}\exp \Big( (t-\delta)\blambda_1(Q^{x_j}_{r_n}) -C \fa_0^{\aleph+1}\delta (\log L)^{2\aleph +2}- \frac{C\fa_0(\log L)^2r^2}{\delta } \Big) \leq e^{\beta t^3}  \Big\} \cap \Upsilon_t \Big) + \dP(\neg \Upsilon_t ) \\
   &\leq  \P \left( \blambda_1(Q^{x_j}_{r_t})  \leq (\beta +\epsilon)t^2  \right)+ \dP(\neg \Upsilon_t). 
  \end{aligned} 
\end{equation*}  For $d=2$, we can proceed similarly using Lemma~5.2 of \cite{KPZ20} to obtain 
\begin{equation*}
  \begin{aligned}
    &\dP ( \cU_0^{x_j}(t,x_j) \leq e^{\beta t^2} ) \\
&\leq \dP \Big( \exp \Big( (t-\delta)\blambda_1(Q^{x^j}_{r_t})  - \frac{r_t^2}{C\delta}  - C\delta(\fa_0\log L_t)^5) -\exp(Ct(\fa_0\log L_t)^5 -\frac{L_t^2}{C\delta}\Big) \leq e^{\beta t^2}\Big)\\
&\leq \dP \Big( \Big\{ \frac{1}{2}\exp ( (t-\delta)\blambda_1(Q^{x^j}_{r_t})  - \frac{r_t^2}{C\delta}- C\delta(\fa_0\log L_t)^5) \leq e^{\beta t^2} \Big\} \cap \Upsilon_t\Big) + \dP(\neg \Upsilon_n)\\
&\leq \dP ( \blambda_1(Q^{x^j}_{r_t})  \leq (\beta +\epsilon)t)+\dP(\neg \Upsilon_t)
  \end{aligned} 
\end{equation*}

By Lemma~\ref{lemma:tail probability of eigenvalue}, we have 
\begin{equation*}
    \max_{1\leq j \leq m} \P \left( \blambda_1(Q^{x_j}_{r_t})  \leq (\beta +\epsilon)t^{\frac{6-d}{4-d}}  \right) \leq \exp ( -c_2 (\beta  +\epsilon)^{\frac{d}{2}} t^{\frac{d}{4-d}} e^{d\log r_t - \fc_d (1+\epsilon)(\beta+\epsilon)^{\frac{4-d}{2}}t} ). 
\end{equation*} Moreover, we have $\dP(\neg \Upsilon_t) \leq e^{-h_0 \log t}$ for some $h_0>0$ by the fact that $\dE[e^{h_0 \sqrt{\fa_0}}]<\infty$.  This yields that 
\begin{equation*}
     \P \left( \max_{1\leq j \leq m } \log u(t,x_j) \leq \beta t^3 \right) \leq 2^{m-1}\exp ( -c m(\beta  +\epsilon)^{\frac{d}{2}} t^{\frac{d}{4-d}} e^{d\log r_t - \fc_d (1+\epsilon)(\beta+\epsilon)^{\frac{4-d}{2}}t} ) + 2^{m-1} e^{-h_0 m\log t}.
\end{equation*} By choosing $t_0$ large such that $2^{m-1}e^{-h_0m\log t }< e^{-cm\log t}$ for some constant $c$, we achieve the bound in \eqref{eq:Spatio-temporalUpperBound}. This completes the proof. 
\end{proof}

Now we are ready to prove the lower bound in Theorem~\ref{thm:spatio-temporal multifractality}. Recall that 
\begin{align*}
    \cP^d(\beta,v) = \left\{ (t,x) \in (1,\infty) \,:\, u(v\log t ,x) > e^{\beta (v\log t )^{\frac{2}{4-d}}}\right\}
\end{align*}

\begin{theorem}\label{thm:lower bound of dim for spatio-temporal peaks in 2d}
  For every $\beta ,v>0$, with probability one,
  \begin{equation}
    \Dim[\cP^d(\beta,v)]\geq (d+1-\beta^{\frac{4-d}{2}} v \fc) \vee d .
  \end{equation}
\end{theorem}
\begin{proof}
 Choose $\beta,\epsilon,v>0$ such that $(\beta +\epsilon)^{\frac{4-d}{2}} (1+\epsilon)v\fc_d <d+1$.
  Let us define 
\begin{equation*}  
  \widetilde{\cP}^d(\beta,v) : = \cP^d(\beta, v) \cap \bigcup_{n=0}^\infty \left( e^{n}, e^{n+1}\right]^{d+1} .
\end{equation*}  Since $ \widetilde{\cP}^d(\beta,v) \subseteq  \cP^d(\beta,v) $, it is enough to show  $\Dim[\widetilde{\cP}^d(\beta,v)] \geq (d +1- \beta^{\frac{4-d}{2}} v \fc_d) \vee d$ with probability one. We choose $\gamma \in (\frac{(\beta +\epsilon)^{\frac{4-d}{2}} (1+\epsilon)v\fc_d }{d}, 1)$  and $\theta \in (0, \gamma -\frac{(\beta +\epsilon)^{\frac{4-d}{2}} (1+\epsilon)v\fc_d }{d})$.  
We borrow the notations $a_{j,n}$ and $I_n(\gamma)$ 
 from Theorem~\ref{thm:lower bound of dimension for spatial peaks}. Suppose $I^k_n(\gamma)$ are copies of $I_k(n)$ for $1\leq k\leq d+1$. We introduce further 
     \begin{equation*}
       \tilde{\cI}_n(\gamma) : =\prod_{k=1}^{d+1} I^k_n(\gamma), 
     \end{equation*} where $I^k_n(\gamma)$ is a copy of $I_n(\gamma)$ for all $1\leq j\leq d+1$. We choose $x \in \tilde{\cI}_n(\gamma)$ and take the points $\{x_i\}_{i=1}^{m(n)}$ such that they satisfy the following conditions: $(b.1)$  $x_i\in \tilde{\cI}_n(\gamma) \cap B(x,e^{n\gamma})$ for all $i=1,...,m(n)$; $(b.2)$ $|x_i-x_j| \geq e^{n\theta}$ whenever $1\leq i<j\leq m(n)$; $(b.3)$ $d^{-1}e^{dn(\gamma-\theta)}\leq m(n) \leq de^{dn(\gamma-\theta)}$. 
     
     Observe that $e^{n\theta} \gg 3 L_{v\log t}  = 3 (v\log t )$ for all $t\in (e^n , e^{n+1}]$. We now notice that there exists $n_0>0$ such that for all $n\geq n_0$,
     \begin{equation*}
       \begin{aligned}
         \P &\Big( \cP^d_1 \cap ( \{ t\} \times B(x,e^{n\gamma})) = \varnothing \text{ for some }t\in (e^n,e^{n+1}] \text{ and }x\in \tilde{\cI}_n(\gamma)\Big)\\
         &\leq \P \Big( \min_{\substack{t\cap \dZ \\t\in (e^n,e^{n+1}]}} \min_{x\in \tilde{\cI}_n(\gamma)} \max_{\{x_i\}_{i=1}^{m(n)} \subseteq \tilde{\cI}_{n}(\theta)\cap B(x,e^{n\gamma})} \Big\{e^{-\beta (v\log t )^{\frac{6-d}{4-d}}}u(v\log t, x ) \Big\}\leq 1 \Big)\\
         &\leq \sum_{\substack{t\cap \dZ \\t\in (e^n,e^{n+1}]}}\sum_{x\in \tilde{\cI}_n(\gamma)} \P \left( \max_{\{x_i\}_{i=1}^{m(n)} \subseteq \tilde{\cI}_{n}(\theta)\cap B(x,e^{n\gamma})} u(v\log t ,x) \leq e^{\beta (v\log t )^{\frac{6-d}{4-d}}}\right)\\
         &\leq C e^{dn(1-\gamma) + n} \cdot\exp \left ( -c m(n)(\beta  +\epsilon)^{\frac{d}{2}} (v(n+1))^{\frac{d}{4-d}} e^{d\log r_{v(n+1)} - \fc_d (1+\epsilon)(\beta+\epsilon)^{\frac{4-d}{2}}v(n+1)} \right) + e^{ -c m(n)v(n+1) }\\
         &\leq C e^{dn(1-\gamma) + n} \cdot \Bigg[\exp \left( -c(\beta+\epsilon)^{\frac{d}{2}}v(n+1)^{\frac{d}{4-d}} e^{\log[v(n+1)] + \kappa n -\fc_d(1+\epsilon)(\beta+\epsilon)^{\frac{4-d}{2}}v} \right) \\ &\quad + \exp\left(-\frac{c}{2}e^{dn(\gamma-\theta)}v(n+1) \right)\Bigg],
       \end{aligned}
     \end{equation*} where $\kappa: = d(\gamma-\theta) - \fc_d(1+\epsilon)(\beta+\epsilon)^{\frac{4-d}{2}}v>0$ by the choice of $\gamma$ and $\theta$. The first inequality is straightforward. The inequality follows by applying the union bound and the third inequality is obtained by applying Proposition~\ref{prop:left tail probability of spatio-temporal maximum}. The right hand side of the above inequality is  summable w.r.t. $n$. Hence, by the Borel-Cantelli lemma, there exists $n_0>0$ such that for all $n\geq n_0$ 
    \begin{equation}
        \widetilde{\cP}^d(\beta,v)\cap ( \{ t\} \times B(x,e^{n\gamma})) \neq \varnothing \text{ for all }t\in (e^n,e^{n+1}] \text{ and all }x\in \cI_n(\gamma).     
     \end{equation} This implies that $\mu_n( \widetilde{\cP}^d(\beta,v)) \geq Ce^{dn(1-\gamma)+n}$ where $\mu_n$ is defined in Proposition~\ref{prop:DensityTheorem}. Therefore, by Proposition~\ref{prop:DensityTheorem} we can deduce that $\sum_{n}\nu_{n,d+1-d\gamma}(\widetilde{\cP}^d(\beta,v))= \infty $ almost surely, which shows that $\Dim(\widetilde{\cP}^d(\beta,v))\geq 4-3\gamma $. Letting $\gamma \downarrow \frac{\fc_d(1+\epsilon)(\beta+\epsilon)^{\frac{4-d}{2}}v}{d}$ and $\theta \downarrow 0 $ without violating $\theta \in (0, \gamma- \frac{\fc_d(1+\epsilon)(\beta+\epsilon)^{\frac{4-d}{2}}v}{d})$, we get $\Dim(\widetilde{\cP}^d(\beta,v))\geq d+1-\fc_d(1+\epsilon)(\beta+\epsilon)^{\frac{4-d}{2}}v $. Since $\epsilon>0$ is arbitrary, $\Dim(\widetilde{\cP}^d(\beta,v)) \geq d+1- \beta^{\frac{4-d}{2}}  v \fc_d$.

     Now it remains to show $\Dim(\widetilde{\cP}^d(\beta,v)) \geq d $, a.s. for any $\beta,v>0$. First note that 
     \begin{equation*}
       \left\{ (s,x) \in (1,\infty) \times \dR^d \,:\, u(v\log s ,x ) > e^{\beta (v\log s )^{\frac{6-d}{4-d}}} \right\} \supseteq \{ t\} \times \left\{ x\in \dR^d \,:\, u(v\log t ,x) > e^{\beta (v\log t )^{\frac{6-d}{4-d}}} \right\},
     \end{equation*} for all $t\geq 1 $. Let us define 
     \begin{equation*}
       \cP^{(t)} : = \left\{ x\in \dR^d \,:\, u(v\log t ,x) > e^{\beta (v\log t )^{\frac{6-d}{4-d}}} \right\}, 
     \end{equation*} for any $t>1$. Then, it suffices to show $\Dim(\cP^{(t)}) \geq d$, a.s. for some $t>1$. Indeed, 
     \begin{equation*}
        \Dim \left(\{ t\} \times \left\{ x\in \dR^d \,:\, u(v\log t ,x) > e^{\beta (v\log t )^{\frac{6-d}{4-d}}} \right\} \right) = \Dim(\cP^{(t)}),
      \end{equation*} for any fixed $t\geq1$ (see \cite[Section~9]{BT92}). Let $t_0$ be the constant in Theorem~\ref{thm:lower bound of dimension for spatial peaks}. Let $\beta, v>0$. Observe that for all $(M,\alpha)\in \dR^2$ such that 
      \begin{equation*}
          M\geq e , \quad \alpha \geq \frac{\beta (v\log t_0)^{\frac{6-d}{4-d}}}{t_0 (\log M)^{\frac{2}{4-d}}},
      \end{equation*} we have 
      \begin{equation*}
      \begin{aligned}
           \cP^{(t_0)}_M := \left\{  \|x\|_\infty\geq M \,:\, u(v\log t_0, x) > e^{\beta (v\log t_0)^{\frac{6-d}{4-d}}}  \right\}\supseteq \left\{ \|x\|_\infty\geq M  \,:\, u(v\log t_0 , x ) > e^{\alpha t_0 (\log \|x\|_\infty)^{\frac{2}{4-d}}} \right\}.
      \end{aligned} 
     \end{equation*} Note that $\cP^{(t_0)}_M$ and $\cP^{(t_0)}$ have the same macroscopic Hausdorff dimension since $\Dim[E]=0$ for every bounded set $E\in \dR^d$. Therefore, by Theorem~\ref{thm:lower bound of dimension for spatial peaks}, we have 
     \begin{equation*}
          \Dim( \cP^{(t_0)} ) = \Dim(\cP^{(t_0)}_M) \geq d - \frac{\beta (v\log t_0)^{\frac{6-d}{4-d}}}{t_0 (\log M)^{\frac{2}{4-d}}}.
     \end{equation*} By taking $M \uparrow \infty$, we can conclude that $\Dim( \cP^{(t_0)} )\geq d$, a.s.

\end{proof}

\subsection{Proof of the upper bound in Theorem~\ref{thm:spatio-temporal multifractality}}\label{subsec:upper bound in thm2}

\begin{proposition}\label{prop:right tail probability of spatio-temporal maximum}

    Let $0<\epsilon<\beta$ and $M>0$. There exist $b:=b(M)>1 $ and $n_0:= n_0(M,b,\epsilon)>0$ such that for all $n\geq n_0$ and $t\geq 1$ 
    \begin{equation}\label{eq:upper bound of maximum of spatio-temporal tall peaks when a_0 bounded}\begin{aligned}
        \dP &\Big( \text{ For some } t\in (a,a+l] \text{ s.t. }\sup_{x\in B(y,1)} \log u(t,x) \geq \beta t^{\frac{6-d}{4-d}} , \fa_0 \leq M \Big)\\ & \leq c_1 (\beta-\epsilon)^{\frac{d}{2}}a^{\frac{d}{4-d}} e^{d b\log (a+l) - (1-\epsilon)\fc_d (\beta-\epsilon)^{\frac{4-d}{2}}a },
    \end{aligned}
\end{equation} where $b>1$ and  $y\in \dR^d$.

\end{proposition}
\begin{proof}
    We use similar argument as in Proposition~\ref{prop:right tail probability of spatial maximum}. Applying Lemma~\ref{lemma:convergence of series to the solution} with $L_t : = t^b$, we can write $u(t,x) = \sum_{k=0}^\infty \cU_k^y(t,x)$ for any $y\in \dR^d.$ Then we have 
    \begin{equation*}
        \begin{aligned}
            \dP \Big( \text{For some $t\in (a,a+l]$}\sup_{x\in B(y,1)} \log u(t,x) \geq \beta t^{\frac{6-d}{4-d}} , \fa_0 \leq M \Big) \leq ({\bf D_1}) + ({\bf D_2}),
        \end{aligned}
    \end{equation*} where 
    \begin{equation*}
        \begin{aligned}
            ({\bf D_1}) &:=  \dP \Big( \text{For some $t\in (a,a+l]$}\sup_{x\in B(y,1)}\cU_0^y(t,x) \geq \frac{1}{2} e^{\beta t^{\frac{6-d}{4-d}}} , \fa_0 \leq M \Big), \\
            ({\bf D_2}) &:=  \dP \Big( \text{For some $t\in (a,a+l]$}\sup_{x\in B(y,1)}\sum_{k=1}^\infty\cU_k^y(t,x) \geq \frac{1}{2} e^{\beta t^{\frac{6-d}{4-d}}} , \fa_0 \leq M \Big).
        \end{aligned}
    \end{equation*} For $d=3$, on the event $\{ \fa_0 \leq M \}$,  we have 
     \begin{equation*}
        \cU^y_k(t,x) \leq C \exp \left( Ct M^{\aleph+1} ( b ( k+1) \log t )^{2\aleph+2} - \frac{t^{2bk}}{Ct} \right). 
    \end{equation*} by applying \eqref{eq:upper bound of cU_k,epsilon}. For $d=2$, similar bounds follows again from Lemma~5.2 of \cite{KPZ20}. Now we can choose a large $b:= b(M)>0$ such that for all $t\geq e $ and $n\geq 1 $
    \begin{equation}
         \sum_{k=1}^\infty \cU_k^y(t,x) \leq\sum_{k=1}^\infty C \exp \left ( - C_1 t^{2bk -1 } \right) \leq \frac{1}{2}e^{\alpha t n^{\frac{2}{4-d}}},
    \end{equation} which shows that $({\bf D_2}) =0$ for all $a\geq e $. We now bound $({\bf D_1})$. Fix $\epsilon\in ( 0, \beta \wedge 1 )$ and use Proposition~\ref{prop:upper and lower bound of the localized solution} to obtain that there exists $a_0:= a_0 (b,M,\epsilon)>0$ such that for all $a\geq a_0$ 
    \begin{equation}
        \begin{aligned}
            ({\bf D_1}) &\leq \dP \Big ( \text{For some $t\in (a,a+l]$, } C\exp ( t \blambda_1(Q^y_{L_t}) + CM^{\upsilon+1} (\log L_t)^{2\upsilon +2 }) \geq  \frac{1}{2}e^{\beta t^{\frac{6-d}{4-d}}}   \Big)\\
         &\leq \dP\Big( \text{For some $t\in (a,a+l]$, }\blambda_1(Q^y_{L_t}) \geq \beta t^{\frac{2}{4-d}} -\frac{\log (2C) + CM^{\upsilon+1} (b\log t)^{2\upsilon +2 }}{t}  \Big)\\
         &\leq \dP ( \text{For some $t\in (a,a+l]$, } \blambda_1(Q^y_{L_t}) \geq (\beta -\epsilon) a^{\frac{2}{4-d}}).
        \end{aligned}
    \end{equation} By Lemma~\ref{lemma:tail probability of eigenvalue} and Lemma~\ref{lemma:monotonicity of eigenvalue}, we have 
    \begin{equation}
      ({\bf D_1})\leq   \dP ( \blambda_1(Q^y_{L_{a+l}}) \geq (\beta -\epsilon) a^{\frac{2}{4-d}}) )\leq c_1 (\beta-\epsilon)^{\frac{d}{2}}a^{\frac{d}{4-d}} e^{d b\log (a+l) - (1-\epsilon)\fc_d (\beta-\epsilon)^{\frac{4-d}{2}}a }.
    \end{equation} This completes the proof.
\end{proof}

Now we proceed to prove the upper bound of macroscopic Hausdorff dimension of the set $\cP^d(\beta,v)$.
\begin{theorem}\label{thm:upper bound of dim for spatio-temporal peaks} 
  For every $v>0$ and $\beta \in(0, \left(d/(v\fc_d)\right)^{\frac{2}{4-d}})$, with probability one,
  \begin{equation}\label{eq:upper bound of dimension of spatio-temporal peaks}
    \Dim[\cP^d(\beta,v)]\leq (d+1-\beta^{\frac{4-d}{2}} v \fc_d) \vee d.
  \end{equation}
\end{theorem}

\begin{proof} For $\epsilon : = (\epsilon_1, ..., \epsilon_d)\in \{ -1,1\}^d$, define an (open) orthant as 
\begin{equation*}
    \cO_\epsilon:= \left\{ (t,x) = (t,x_1,...,x_d) \in (1,\infty) \times \dR^d \,:\, \epsilon_1x_1>0, \epsilon_2x_2>0 ,..., \epsilon_dx_d> 0  \right\}. 
\end{equation*} We then define 
\begin{equation*}
    \cP^d_{\epsilon}(\beta,v) : = \cP^d(\beta,v) \cap \cO_\epsilon.
\end{equation*} In order to prove this theorem, it suffices to prove that $\Dim[\cP^d_{\epsilon_+}(\beta,v)] (d+1-\beta^{\frac{4-d}{2}} v \fc_d) \vee d$ for any $\epsilon \in \{ -1,1\}^d$. Due to symmetry between different orthants, it further suffices to prove for $\epsilon_+: = (1,...,1)\in\{-1,1\}^d$ that 
\begin{equation}\label{eq:upper bound of dimension in first orthant}
    \Dim[\cP^d_{\epsilon_+}(\beta,v)] \leq  (d+1-\beta^{\frac{4-d}{2}} v \fc_d) \vee d.
\end{equation}  For $q>1$ and $n\in \dN$, let us denote $
    \cL_n : = \cL_n(q,n,\beta,v,d) : = \cP^d_{\epsilon_+}(\beta,v) \cap \cI^{(q)}_n$ where $\cI^{(q)}_n : = (e^{n/q},e^{n+1}]^{d+1}$. By Lemma~\ref{lemma:block lemma}, we have 
\begin{equation*}
    \Dim\left[ \cP^d_{\epsilon_+}(\beta,v) \setminus \bigcup_{n=0}^\infty \cL_n\right] \leq d.
\end{equation*} Since $\Dim(A\cup B) = \max \{ \Dim(A) , \Dim(B)\}$ for any two sets $A,B\subseteq \dR^d$, we have 
\begin{equation*}
     \Dim\left[ \cP^d_{\epsilon_+}(\beta,v) \right]\leq  \Dim\left[ \cP^d_{\epsilon_+}(\beta,v) \setminus \bigcup_{n=0}^\infty \cL_n\right] \vee \Dim\left[  \bigcup_{n=0}^\infty \cL_n\right].
\end{equation*} Let $\Bar{\cL}:= \bigcup_{n=0}^\infty \cL_n $. The above inequality implies that to show \eqref{eq:upper bound of dimension in first orthant}, it is enough to prove
\begin{equation*}
    \Dim\left(\Bar{\cL}\right)\leq (d+1-\beta^{\frac{4-d}{2}} v \fc_d).
\end{equation*} To this end, observe that  Proposition~\ref{prop:right tail probability of spatio-temporal maximum} implies for all $a\in (e^{n/q}, e^{n+1}]$ 
\begin{equation}\label{eq:right tail probability of spatio-temporal peaks in L_n}
 \begin{aligned}
      \dP &\Big( \text{For some $t\in (a,a+1]$}\sup_{x\in B(y,1)} \log u(v\log t,x) \geq \beta (v \log t)^{\frac{6-d}{4-d}} , \fa_0 \leq M \Big) \\
       & = \dP ( \text{For some $t\in (v\log a,v\log(a+1)]$}\sup_{x\in B(y,1)} \log u(t,x) \geq \beta t^{\frac{6-d}{4-d}} , \fa_0 \leq M \Big)\\       
       &\leq c_1 (\beta-\epsilon)^{\frac{d}{2}}\left( \frac{vn}{q}\right)^{\frac{d}{4-d}} \exp\left( d b\log \left( 2v n\right) - \frac{(1-\epsilon)\fc_d (\beta-\epsilon)^{\frac{4-d}{2}}vn}{q} \right),
 \end{aligned}
\end{equation} For all sufficiently large $n\in \dN$, we cover $\cL_n \subseteq \cI^{(q)}_n$ with $O(e^{d+1})$-many boxes of the form $(a,a+1] \times B(y,1)$ satisfying that for some $t\in (a,a+1]$
\begin{equation}\label{eq:covering of spatio-temporal peaks}
    \sup_{x\in B(y,1)} \log  u(v \log t ,x ) \geq \beta (v\log t )^{\frac{6-d}{4-d}}.
\end{equation} on the event $\Upsilon_M : = \{ \fa_0 \leq M \}$. Choose 
\begin{equation*}
    \rho \in \Big(  d+1 -\frac{\fc_d \beta_\epsilon^{\frac{4-d}{2}} v n}{q} , d+1 \Big].
\end{equation*}

By \eqref{eq:right tail probability of spatio-temporal peaks in L_n}, we have that for all sufficiently large $n\geq 1$ and for all $\rho>0$,
\begin{equation*}
    \begin{aligned}
        \dE[ \nu_\rho^n(\cL_n) \1_{\Upsilon_M} ] &\leq  \dE\Big[ \sum_{ \substack{(a,a+1] \times B(y,1) \subseteq \cI_n^{q} :  \\ \text{\eqref{eq:covering of spatio-temporal peaks} holds}}} e^{-n\rho} \cdot \1_{\Upsilon_M}\Big]\\
        &\leq C \exp\Big( \Big\{ d+1 - \rho - \frac{\fc_d \beta_\epsilon^{\frac{4-d}{2}} v n}{q}    \Big\}n     + C b \log n  \Big),
    \end{aligned}
\end{equation*} where $C>0$ is a constant which depends only on $(v,\beta)$ and $\beta_\epsilon:= ( 1-\epsilon)^{\frac{2}{4-d}} (\beta-\epsilon) $. This implies that 
\begin{equation*}
   \dP \Big( \Upsilon_M \cap \Big\{  \sum_{n=0}^\infty\nu_\rho^n ( \Bar{\cL}) <  \infty \Big\} \Big) =\dP \big( \Upsilon_M \big).
 \end{equation*} Because $\dP ( \Upsilon_M) \geq 1-e^{-h\sqrt{M}}\mathbb{E}[e^{\sqrt{\mathfrak{a}_0}}]$ (see \eqref{eq:probability of fa bounded}), we have 
 \begin{equation*}
      \dP \Big( \sum_{n=0}^\infty\nu_\rho^n ( \Bar{\cL}) <  \infty \Big) \geq 1-e^{-h\sqrt{M}}\mathbb{E}[e^{\sqrt{\mathfrak{a}_0}}],
 \end{equation*} which in turn implies 
 $
    \dP(  \Dim( \Bar{\cL}) \leq \rho ) \geq 1- e^{-h\sqrt{M}}\mathbb{E}[e^{\sqrt{\mathfrak{a}_0}}].
 $ Since the definition of $\Bar{\cL}$ does not depend on $M$ and $M>0$ can be arbitrarily large, we have $ \Dim( \Bar{\cL}) \leq \rho $ almost surely. Taking $\rho \downarrow d+1 -\frac{\fc_d \beta_\epsilon^{\frac{4-d}{2}} v n}{q} $, we also have $ \Dim( \Bar{\cL}) \leq  d+1 -\frac{\fc_d \beta_\epsilon^{\frac{4-d}{2}} v }{q}$.  Moreover, since $q>1$ and $\epsilon\in(0, \beta\wedge 1 )$ are arbitrary, we can let $q\to 1$ and $\epsilon \to 0$ to conclude that $\Dim( \Bar{\cL}) \leq d+1 - \fc_d \beta^{\frac{4-d}{2}} v $. This completes the proof.

\end{proof}

\appendix

\section{Besov space and paracontrolled generator} \label{sec:Paracontrolled}

We start with introducing few notations about the function spaces and the paracontrolled calculus. Let $\chi$ and $\varrho$ be non-negative radial function such that 

\begin{enumerate}
\item the support of $\chi$ is contained in a ball and the support of $\varrho$ is contained in an annulus $\{x \in \dR^d : 1< |x| <2 \}$. 

\item $\chi(\xi) + \sum_{j\geq 0} \varrho(2^{-j}\xi) = 1$ for all $\xi \in \bR^d$. 

\item $\mathrm{Supp}(\chi)\cap \mathrm{Supp}(\varrho(2^{-j}\cdot)) = \emptyset$ for $i\geq 1$ and $\mathrm{Supp}(\varrho(2^{-i}\cdot)\cap \mathrm{Supp}(\varrho(2^{-j}\cdot)) = \emptyset$ when $|i-j|>1$. 
\end{enumerate}
To this end,  $(\chi,\varrho)$ satisfying the above properties are said to form a dyadic partition of unity. For
the existence, we refer to \cite[Proposition 2.10]{BCD11}.

For any Schwartz distribution $f$, we define the Littlewood-Paley blocks by 
\begin{equation}
    \Delta_i f  = \sum_{k\in\dN_0^d} \langle f, \fn_{k,L} \rangle \varrho_i\Big(\frac{k}{L}\Big)\fn_{k,L}
\end{equation} where $\{\fn_{k,L}: k\in \dN_0 \}$ forms an orthonormal basis of $L^2([0,L]^d)$ given in \cite[Section~4]{CZ21} and $\varrho_{j}(\cdot) = \varrho(2^{-j}\cdot)$. We also note that since $\varrho_j$ is supported in a ball with radius $2^j$ and $\varrho_j \leq 1$, for all $j\in\dN_0$, $x\in \dR^d$ and $\gamma>0$
\begin{equation}\label{eq:bound of varrho}
    \varrho_j(x) \lesssim \left( \frac{2^j}{1+|x|}\right)^{\gamma}.
\end{equation} 

For $u\in \sS'$, we define $(1-\frac{1}{2}\Delta )^{-1} u $ by 
\begin{equation}\label{eq:definition of inverse laplacian}
    (1-\frac{1}{2}\Delta )^{-1} u  := \sum_{k\in \dN^d_0} \sigma\Big( \frac{k}{L}\Big) \langle  u, \fn_{k,L}\rangle \fn_{k,L},
\end{equation} where $\sigma (x) := (1+\pi|x|^2)^{-1}$.

Denote the Fourier transform operator by $\mathfrak{F}$ and let $(\chi,\varrho)$ be the dyadic partition of unity. Then the Littlewood-Paley blocks are defined as 
\begin{align}
    \Delta_{-1} u = \mathfrak{F}^{-1}(\chi\mathfrak{F}(u)), \quad \Delta_{j} u  = \mathfrak{F}^{-1}\big(\varrho_{j}(\cdot) \mathfrak{F}(u)\big), \quad j\geq 0  
\end{align}
where $\varrho_{j}(\cdot) = \varrho(2^{-j}\cdot)$ and, for $\alpha \in \bR$, $p,q\in [1,\infty]$, the Besov space $B^{\alpha}_{p,q}(\bR^d, \bR^n)$ is 
\begin{align}
    B^{\alpha}_{p,q}(\bR^d, \bR^n) :=\{u\in \mathscr{S}^{\prime}(\bR^d, \bR^n); \quad \|u\|^{q}_{B^{\alpha}_{p,q}(\bR^d, \bR^n)} = \sum_{j\geq -1}2^{jq\alpha}\|\Delta_{j} u\|^{q}_{L^p(\bR^d, \bR^n)}<\infty\}
\end{align}
 We often use the notation $\mathscr{C}_p^{\alpha}(\bR^d, \bR^n)$ to denote $B^{\alpha}_{p,p}(\bR^d, \bR^n)$ for $p\in [1,\infty]$ and write $\sC^\alpha(\bR^d, \bR^n)$ for $\sC_\infty^\alpha(\bR^d, \bR^n)$. This notation is consistent with the fact that $\mathscr{C}^{\alpha}(\bR^d, \bR^n)$ is indeed space of all $\alpha$-H\"older continuous function. For simplicity, we sometimes use the notation $\sC^\alpha_p$ for $\sC^\alpha_p(\dR^d,\dR)$. 
 
 Let $\delta,\rho>0$, $T>0,$ and $\bar{T}\in [0,T)$. Let $(D, \| \cdot \|_D)$ be a Banach space and $u,v:[T-\bar{T},T]\rightarrow D $ be function (or distribution) valued processes. We say that $u\in C^{\delta}_{\rho, \bar{T},T}D $ if $ \| u\|_{C^{\delta}_{\rho,\bar{T},T}D} < \infty $ and $v\in C^{\delta}_{\bar{T},T}D$ if $\|v\|_{C^{\delta}_{\bar{T},T}D}<\infty$ where 
 \begin{equation}\label{eq:definition of the norm for Holder-Besov valued process}\begin{aligned}
     \| u\|_{C^{\delta}_{\rho,\bar{T},T}D} &: = \sup_{s<t\in(T-\bar{T},T]} (T-t)^\delta\frac{ \| u(t)-u(s)\|_D}{|t-s|^\rho},\\
     \|v\|_{C^{\delta}_{\bar{T},T}D} &: = \sup_{t\in(T-\bar{T},T]}(T-t)^\delta\|v(t)\|_D.
 \end{aligned}
 \end{equation} In the case of $\delta = 0$ or $\bar{T} = T$, then we drop the respective subscripts in the above definition.

\subsection{Some properties of the Besov-H\"older continuous distributions}

Let $f$ and $g$ be two distributions in $\mathscr{S}^{\prime}(\bR^d)$. Then the Paley-Littlewood decomposition of $fg$ is written as 
\begin{align*}
    fg = f\prec g+ f\circ g+ f\succ g
\end{align*}
 where $f\prec g$ and $f\succ g$ are called \emph{paraproducts} and $f\circ g$ is called the \emph{resonant terms} and they are defined as 
 \begin{align*}
     f\prec g = f\succ g = \sum_{j\geq -1} \sum_{i<j-1} \Delta_{i}\Delta_{j} g, \quad \text{and } \quad f\circ g = \sum_{j\geq -1}\sum_{|i-j|\leq 1} \Delta_{i}\Delta_{j} g.  
 \end{align*}
In the following propositions, we note few useful properties of the paraproduct. 

\begin{proposition}[Bony's estimates (I), \cite{BCD11}]\label{prop:Bony's estimates (I)}
    Let $\alpha,\beta \in \dR$. Let $f\in \mathscr{C}^{\alpha}$ and $f\in \mathscr{C}^{\beta}$,
    \begin{enumerate}
        \item If $\alpha>0$, then $f\prec g \in \mathscr{C}^{\beta}$ and $\|f\prec g\|_{\beta}\leq \|f\|_{L^{\infty}}\|g\|_{\beta}$ 
      \item If $\alpha <0$, then $f\prec g \in \mathscr{C}^{\alpha+\beta}$  and $\|f\prec g\|_{L^{\infty}}\lesssim \|f\|_{\alpha} \|g\|_{\beta}$. 
     \item If $\alpha +\beta >0$, then $f\circ g \in \mathscr{C}^{\alpha+\beta}$ and $\|f\circ g\|_{\alpha+\beta} \lesssim \|f\|_{\alpha}\|g\|_{\beta}$.
    \end{enumerate}
\end{proposition}
 
\begin{proposition}[Bony's estimates (II), \cite{BCD11}]\label{prop:Bony's estimates (II)} Let $\alpha<0,\beta>0$ and $\alpha+\beta >0$. Let $p, p_1,p_2,q_1,q_2\in [1,\infty]$ be satisfy $\frac{1}{p}= \frac{1}{p_1} + \frac{1}{p_2}$. Let $f\in B^\alpha_{p_1,q_1}$ and $g\in B^\beta_{p_2,q_2}$. For $q\geq q_1$
\begin{enumerate}
    \item $\| f \prec g\|_{B^{\alpha+\beta}_{p,q}} \lesssim \| f \|_{B^\alpha_{p_1,q_1}} \|g \|_{ B^\beta_{p_2,q_2}}$.
    \item $\| f \succ g  \|_{B^{\alpha}_{p,q}} \lesssim \| f \|_{B^\alpha_{p_1,q_1}} \|g \|_{ B^\beta_{p_2,q_2}}.$
    \item  $\| f \circ g \|_{B^{\alpha+\beta}_{p,q}} \lesssim \| f \|_{B^\alpha_{p_1,q_1}} \|g \|_{ B^\beta_{p_2,q_2}}.$
\end{enumerate}
    
\end{proposition}

 \begin{proposition}[Schauder's estimate,Lemma~2.5 of \cite{CC2018}, Lemma~A.8 of  \cite{GIP2015}]\label{lem:schauder estimate} Let $P_t$ be the heat semigroup for $\frac{1}{2}\Delta$. Let $\theta\geq 0, p\in [1,\infty]$ and $\alpha\in \dR$. Then for $\phi \in \sC^\alpha_p$ and $0\leq s \leq t$ we have 
\begin{equation}
  \| P_t \phi \|_{\sC^{\alpha+2\theta}_p} \lesssim t^{-\theta}\| \phi \|_{\sC^\alpha_p} , \qquad \|(P_{t-s} - \textrm{Id} )\phi\|_{\sC^{\alpha-2\theta}_p} \lesssim |t-s|^\theta \| \phi \|_{\sC^\alpha_p}.
\end{equation}
\end{proposition}

\section{Macroscopic Hausdorff dimension}\label{sec2}
In this section, we introduce the notion of the macroscopic Hausdorff dimension given by Barlow and Taylor \cite{BT89,BT92}, and Khoshnevisan-Kim-Xiao \cite{KKX17}. We also present a useful propositions that help to provide lower bound and upper bounds to the macroscopic Hausdorff dimension of any given set.

\subsection{Definition}\label{subsec2.1}
For all integers $n\geq 1 $, we define the exponential cubes and shells as follows:
\begin{equation}
    \dV_n : = [-e^{n}, e^n)^d, \quad \dS_0 : = V_0, \quad \text{and} \quad \dS_{n+1}: = \dV_{n+1} \setminus \dV_n.
\end{equation}
Let $\cB$ be the collection of all cubes of the form
\begin{equation}
    B(x,r) := \prod_{i=1}^d [x_i,x_i+r), 
\end{equation}
for $x=(x_1,...,x_d)\in \R^d$, and $r \in [1,\infty)$. For any subset $E\subset \R^d$, $\rho>0$, and all integers $n\geq 1$, we define 
\begin{equation}
    \nu_\rho^n(E):= \inf \left\{\sum_{i=1}^m \left (\frac{s(B_i)}{e^n} \right)^\rho : B_i \in \mathcal{B}, B_i \subset \dS_n \text{ and }E\cap \dS_n \subset \cup_{i=1}^m B_i  \right\},
\end{equation}
where $s(B):=r$ denotes the side of $B=B(x,r)$. We now introduce the definition of the macroscopic Hausdorff dimension. 
\begin{definition}\label{def:definition of dimension} [\cite{BT89,BT92}]
The {\it macroscopic Hausdorff dimension} of $E \subset \R^d$ is defined as 
\begin{equation}
    \Dim (E) : = \inf \left\{\rho>0: \sum_{n=1}^\infty \nu_\rho^n (E) < \infty \right\}.
\end{equation}

\end{definition}

\subsection{Useful bounds for macroscopic Hausdorff dimension}
Choose and fix any $\theta\in (0,1)$. We define 
$$
a_{j,n}(\theta) := e^n + je^{n\theta}, \qquad 0\leq j < e^{n(1-\theta)},
$$
$$
I_{n}(\theta) : = \bigcup_{\substack{0\leq j \leq e^{n(1-\theta)}:\\ j\in \Z}} \{ a_{j,n}(\theta)\},
$$
and
$$
\mathcal{I}_{n}(\theta) : = \prod_{i=1}^{d} I^i_{n}(\theta),
$$
where $I^i_n(\theta)$ is a copy of $I_n(\theta)$ for each $i$. We call $\cup_{n=1}^\infty\mathcal{I}_n(\theta) $ a $\theta$-skeleton of $\R^d$ (see \cite[Definition~4.2]{KKX17}). Note that $\Dim \left( \cup_{n=k}^\infty\mathcal{I}_n(\theta) \right) = d (1-\theta)$ for any integer $k\geq 1 $.
\begin{definition}[Definition 4.3 of \cite{KKX17}]\label{def:ThickSet}
$E$ is called $\theta$-thick if there exists a positive integer $k=k(\theta)$ such that 
$$
E\cap Q(x,e^{n\theta}) \neq \emptyset,
$$
for all $x \in \mathcal{I}_n(\theta)$ and $n\geq k$.
\end{definition}
By the monotonicity of the macroscopic Hausdorff dimension, we get the following lower bound.
\begin{proposition}[Proposition~4.4 of \cite{KKX17}]\label{prop:thick set} 
Let $E \subset \R^d$. If $E$ contains a $\theta$-thick set for some $\theta\in (0,1)$, then 
$$
\Dim (E) \geq d(1-\theta).
$$

\end{proposition}

For the set of spatio-temporal peaks, we separately provide a proposition for the lower bound of the macroscopic Hausdorff dimension.
\begin{proposition}[Theorem~4.1 of \cite{BT92}]\label{prop:DensityTheorem}
Fix $\gamma \in (0,d)$. For any set $E$ and $n\in \Z$, let us define  
\begin{align}\label{eq:mu_n}
\mu_n(E) = \sum_{\substack{s\in \Z\\e^{n}< s\leq e^{n+1}}} \sum_{\substack{\vec{j}\in \dZ^d \\ \vec{j}\in [0, e^{n(1-\gamma)})}} \1\{(s,j)\in E\}. 
\end{align} 
Then, there exists a constant $C>0$ such that $\nu_{n,d+1-d\gamma}(E)\geq C e^{-nd(1-\gamma) -n}\mu_n(E)$. 
\end{proposition}
\begin{proof}
The proof follows from \cite[Theorem~4.1]{BT92}. For the condition of \cite[Theorem~4.1]{BT92}, it suffices to verify that $\mu_n( (s,s+r]\times B(x,r)) \lesssim r^{d+1}$ for all $r\geq 1 $, which is proven below (4.24) of \cite{yi2022macroscopic}. 
\end{proof}

The following lemma helps us to compute the upper bound of the macroscopic Hausdorff dimension.
\begin{lemma}[Lemma~4.2 of \cite{yi2022macroscopic}]\label{lemma:block lemma}
  For any $q>1$ and $k\in \{1,...,d-1\}$, define a set $E\subseteq \dR^d$ as 
  \begin{equation*}
     E: = \bigcup_{n=0}^\infty E_n,
   \end{equation*}  where 
   \begin{equation*}
     E_n : = (0,e^{n/q}]^k \times (e^{n/q},e^{n+1}]^{d-k}.
   \end{equation*} Then we have $\Dim[E] \leq d-k.$
\end{lemma}

\bibliographystyle{alpha}		
\bibliography{refs_PAM}

\begin{thebibliography}{ZTPSM00}

\bibitem[And58]{And58}
P.~W. Anderson.
\newblock Absence of diffusion in certain random lattices.
\newblock {\em Phys. Rev.}, 109:1492--1505, Mar 1958.

\bibitem[BCD11]{BCD11}
H.~Bahouri, J.-Y. Chemin, and R.~Danchin.
\newblock {\em Fourier analysis and nonlinear partial differential equations},
  volume 343 of {\em Grundlehren der mathematischen Wissenschaften [Fundamental
  Principles of Mathematical Sciences]}.
\newblock Springer, Heidelberg, 2011.

\bibitem[BKdS18]{BKS18}
M.~Biskup, W.~K\"{o}nig, and R.~S. dos Santos.
\newblock Mass concentration and aging in the parabolic {A}nderson model with
  doubly-exponential tails.
\newblock {\em Probab. Theory Related Fields}, 171(1-2):251--331, 2018.

\bibitem[BT89]{BT89}
M.~T Barlow and S.~J Taylor.
\newblock Fractional dimension of sets in discrete spaces.
\newblock {\em Journal of Physics A: Mathematical and General}, 22(13):2621,
  1989.

\bibitem[BT92]{BT92}
M.~T. Barlow and S.~J. Taylor.
\newblock Defining fractal subsets of $\mathbb{Z}^d$.
\newblock {\em Proceedings of the London Mathematical Society}, 3(1):125--152,
  1992.

\bibitem[CC18a]{CC18}
G.~Cannizzaro and K.~Chouk.
\newblock Multidimensional sdes with singular drift and universal construction
  of the polymer measure with white noise potential.
\newblock {\em The Annals of Probability}, 46(3):1710--1763, 2018.

\bibitem[CC18b]{CC2018}
R.~Catellier and K.~Chouk.
\newblock Paracontrolled distributions and the 3-dimensional stochastic
  quantization equation.
\newblock {\em Ann. Probab.}, 46(5):2621--2679, 2018.

\bibitem[CD20]{CD2020}
S.~Chatterjee and A.~Dunlap.
\newblock Constructing a solution of the {$(2+1)$}-dimensional {KPZ} equation.
\newblock {\em Ann. Probab.}, 48(2):1014--1055, 2020.

\bibitem[Che15]{Chen15}
X.~Chen.
\newblock Precise intermittency for the parabolic anderson equation with an
  $(1+ 1) $-dimensional time--space white noise.
\newblock In {\em Annales de l'IHP Probabilit{\'e}s et statistiques},
  volume~51, pages 1486--1499, 2015.

\bibitem[CM94]{CM94}
R.~A. Carmona and S.~A. Molchanov.
\newblock Parabolic {A}nderson problem and intermittency.
\newblock {\em Mem. Amer. Math. Soc.}, 108(518):viii+125, 1994.

\bibitem[CSZ20]{CSZ2020}
F.~Caravenna, R.~Sun, and N.~Zygouras.
\newblock The two-dimensional {KPZ} equation in the entire subcritical regime.
\newblock {\em Ann. Probab.}, 48(3):1086--1127, 2020.

\bibitem[CvZ21]{CZ21}
K.~Chouk and W.~van Zuijlen.
\newblock Asymptotics of the eigenvalues of the anderson hamiltonian with white
  noise potential in two dimensions.
\newblock {\em The Annals of Probability}, 49(4):1917--1964, 2021.

\bibitem[DG21]{DG21}
S.~Das and P.~Ghosal.
\newblock Law of iterated logarithms and fractal properties of the {KPZ}
  equation.
\newblock {\em arXiv preprint arXiv:2101.00730 (to appear in The Annals of
  Probability)}, 2021.

\bibitem[DL]{DL21a}
L.~{Dumaz} and C.~{Labb{\'e}}.
\newblock {The delocalized phase of the Anderson Hamiltonian in $1$-d}.
\newblock {\em arXiv e-prints}, page arXiv:2102.05393, February.

\bibitem[DL20]{DL2020}
L.~Dumaz and C.~Labb\'{e}.
\newblock Localization of the continuous {A}nderson {H}amiltonian in 1-{D}.
\newblock {\em Probab. Theory Related Fields}, 176(1-2):353--419, 2020.

\bibitem[DL21]{DL21b}
L.~{Dumaz} and C.~{Labb{\'e}}.
\newblock {Localization crossover for the continuous Anderson Hamiltonian in
  $1$-d}.
\newblock {\em arXiv e-prints}, page arXiv:2102.09316, February 2021.

\bibitem[Fri10]{friedman2010stochastic}
Avner Friedman.
\newblock Stochastic differential equations and applications.
\newblock In {\em Stochastic differential equations}, pages 75--148. Springer,
  2010.

\bibitem[GD05]{GD05}
J.~D. Gibbon and C.~R. Doering.
\newblock Intermittency and regularity issues in 3{D} {N}avier-{S}tokes
  turbulence.
\newblock {\em Arch. Ration. Mech. Anal.}, 177(1):115--150, 2005.

\bibitem[GGL22]{GGL22}
P.~Y. {Gaudreau Lamarre}, P.~{Ghosal}, and Y.~{Liao}.
\newblock {Moment Intermittency in the PAM with Asymptotically Singular Noise}.
\newblock {\em arXiv e-prints}, page arXiv:2206.13622, June 2022.

\bibitem[GIP15a]{GIP2015}
M.~Gubinelli, P.~Imkeller, and N.~Perkowski.
\newblock Paracontrolled distributions and singular {PDE}s.
\newblock {\em Forum Math. Pi}, 3:e6, 75, 2015.

\bibitem[GIP15b]{GIP15}
M.~Gubinelli, P.~Imkeller, and N.~Perkowski.
\newblock Paracontrolled distributions and singular {PDE}s.
\newblock {\em Forum Math. Pi}, 3:e6, 75, 2015.

\bibitem[GP17]{GP17}
M.~Gubinelli and N.~Perkowski.
\newblock K{PZ} reloaded.
\newblock {\em Comm. Math. Phys.}, 349(1):165--269, 2017.

\bibitem[GY21]{GY21}
P.~{Ghosal} and J.~{Yi}.
\newblock {Fractal Geometry of the Valleys of the Parabolic Anderson Equation}.
\newblock {\em To appear in Annales de l'IHP Probabilit{\'e}s et statistiques},
  page arXiv:2108.03810, August 2021.

\bibitem[Hai14]{Hai14}
M.~Hairer.
\newblock A theory of regularity structures.
\newblock {\em Inventiones mathematicae}, 198(2):269--504, 2014.

\bibitem[HL15]{HL2015}
M.~Hairer and C.~Labb\'{e}.
\newblock A simple construction of the continuum parabolic {A}nderson model on
  {${\bf R}^2$}.
\newblock {\em Electron. Commun. Probab.}, 20:no. 43, 11, 2015.

\bibitem[HL18a]{HL2018}
M.~Hairer and C.~Labb\'{e}.
\newblock Multiplicative stochastic heat equations on the whole space.
\newblock {\em J. Eur. Math. Soc. (JEMS)}, 20(4):1005--1054, 2018.

\bibitem[HL18b]{hairer2018multiplicative}
M.~Hairer and C.~Labb{\'e}.
\newblock Multiplicative stochastic heat equations on the whole space.
\newblock {\em Journal of the European Mathematical Society}, 20(4):1005--1054,
  2018.

\bibitem[HL22]{HL22}
Y.-S. Hsu and C.~Labb{\'e}.
\newblock Asymptotic of the smallest eigenvalues of the continuous anderson
  hamiltonian in $d\leq 3$.
\newblock {\em Stochastics and Partial Differential Equations: Analysis and
  Computations}, pages 1--34, 2022.

\bibitem[K\"16]{Konig16}
W.~K\"{o}nig.
\newblock {\em The parabolic {A}nderson model}.
\newblock Pathways in Mathematics. Birkh\"{a}user/Springer, [Cham], 2016.
\newblock Random walk in random potential.

\bibitem[KKX17]{KKX17}
D.~Khoshnevisan, K.~Kim, and Y.~Xiao.
\newblock Intermittency and multifractality: A case study via parabolic
  stochastic pdes.
\newblock {\em The Annals of Probability}, 45(6A):3697--3751, 2017.

\bibitem[KKX18]{KKX18}
D.~Khoshnevisan, K.~Kim, and Y.~Xiao.
\newblock A macroscopic multifractal analysis of parabolic stochastic pdes.
\newblock {\em Communications in Mathematical Physics}, 360(1):307--346, 2018.

\bibitem[KPvZ20]{KPZ20}
W.~K{\"o}nig, N.~Perkowski, and W.~van Zuijlen.
\newblock Longtime asymptotics of the two-dimensional parabolic anderson model
  with white-noise potential.
\newblock {\em arXiv preprint arXiv:2009.11611}, 2020.

\bibitem[Lab19]{Lab19}
Cyril Labb{\'e}.
\newblock The continuous anderson hamiltonian in d$\leq$ 3.
\newblock {\em Journal of Functional Analysis}, 277(9):3187--3235, 2019.

\bibitem[Nua06]{Nual2006}
D.~Nualart.
\newblock {\em The {M}alliavin calculus and related topics}.
\newblock Probability and its Applications (New York). Springer-Verlag, Berlin,
  second edition, 2006.

\bibitem[PvZ22]{PZ22}
N.~Perkowski and W.~van Zuijlen.
\newblock Quantitative heat-kernel estimates for diffusions with distributional
  drift.
\newblock {\em Potential Analysis}, pages 1--22, 2022.

\bibitem[Str08]{Str08}
D.~W. Stroock.
\newblock {\em Partial differential equations for probabalists}.
\newblock Cambridge University Press, 2008.

\bibitem[Yi22]{yi2022macroscopic}
Jaeyun Yi.
\newblock Macroscopic multi-fractality of gaussian random fields and linear
  stochastic partial differential equations with colored noise.
\newblock {\em Journal of Theoretical Probability}, pages 1--22, 2022.

\bibitem[ZTPSM00]{Z.et.Al00}
M.~G. Zimmermann, R.~Toral, O.~Piro, and M.~San~Miguel.
\newblock Stochastic spatiotemporal intermittency and noise-induced transition
  to an absorbing phase.
\newblock {\em Phys. Rev. Lett.}, 85:3612--3615, Oct 2000.

\end{thebibliography}

%
%
%
%
%
%

\end{document}